\documentclass[a4paper,twoside,11pt]{article}

\usepackage{a4wide, fancyhdr, amsmath, amssymb, mathtools, yfonts}
\usepackage{mathrsfs}
\usepackage{graphicx}
\usepackage{tikz}
\usepackage[all]{xy}
\usepackage[utf8]{inputenc}
\usepackage{amsthm}
\usepackage[english]{babel}
\usepackage{chngcntr}
\usepackage{ifthen}
\usepackage{calc}
\usepackage{hyperref}
\usepackage{authblk}
\numberwithin{equation}{section}
\usepackage{tikz}
\usetikzlibrary{arrows,decorations.markings}

\newcommand{\Z}{\mathbb{Z}}
\newcommand{\Q}{\mathbb{Q}}

\newcommand{\pr}{\textup{pr}}

\newcommand\FF{\mathbb{F}}

\newcommand\pp{\mathfrak{p}}

\newcommand\Gal{\mathrm{Gal}}

\DeclareMathOperator{\sep}{sep}

\newtheorem{lemma}{Lemma}[section]
\newtheorem{theorem}[lemma]{Theorem}
\newtheorem{prop}[lemma]{Proposition}
\newtheorem{corollary}[lemma]{Corollary}
\newtheorem{mydef}[lemma]{Definition}

\newtheorem{remark}{Remark}

\title{\vspace{-\baselineskip}\sffamily\bfseries On the distribution of $\text{Cl}(K)[l^\infty]$ for degree $l$ cyclic fields}
\author[1]{Peter Koymans\thanks{Niels Bohrweg 1, 2333 CA Leiden, Netherlands, p.h.koymans@math.leidenuniv.nl}}
\author[2]{Carlo Pagano\thanks{Vivatsgasse 7, 53111 Bonn, Germany, carlein90@gmail.com}}
\affil[1]{Mathematisch Instituut, Leiden University}
\affil[2]{Max Planck Institute for Mathematics}

\date{\today}

\begin{document}
\maketitle

\begin{abstract}
Using a recent breakthrough of Smith \cite{Smith}, we prove that $l^\infty$-class groups of cyclic degree $l$ fields have the distribution conjectured by Gerth under GRH.
\end{abstract}

\tableofcontents

\section{Introduction}
\label{sIntro}
Class groups have a long and rich history going back to Gauss, who studied them in the language of binary quadratic forms. In modern terms, Gauss gave an explicit description of $\text{Cl}(K)[2]$ for $K$ a quadratic number field with narrow class group $\text{Cl}(K)$. This is now known as genus theory. Since then class groups have been extensively studied leading to the development of class field theory and the Langlands conjectures.

Nowadays the class group is typically thought of as a `random' object. Cohen and Lenstra put forward conjectures on the average behavior of class groups. Their conjecture predicts that for all odd primes $p$ and all finite, abelian $p$-groups $A$
\[
\lim_{X \rightarrow \infty} \frac{\left|\left\{K \text{ quadratic} : 0 < D_K < X \text{ and } \text{Cl}(K)[p^\infty] \cong A\right\}\right|}{\left|\left\{K \text{ quadratic} : 0 < D_K < X\right\}\right|} = \frac{\prod_{i = 2}^\infty \left(1 - \frac{1}{p^i}\right)}{|A| |\text{Aut}(A)|},
\]
where $D_K$ denotes the discriminant of our field $K$. They also proposed a similar conjecture for imaginary quadratic fields, namely
\[
\lim_{X \rightarrow \infty} \frac{\left|\left\{K \text{ quadratic} : -X < D_K < 0 \text{ and } \text{Cl}(K)[p^\infty] \cong A\right\}\right|}{\left|\left\{K \text{ quadratic} : -X < D_K < 0\right\}\right|} = \frac{\prod_{i = 1}^\infty \left(1 - \frac{1}{p^i}\right)}{|\text{Aut}(A)|}.
\]
Although the Cohen and Lenstra conjectures have attracted a great deal of attention, there are very few proven instances. Davenport and Heilbronn \cite{DH} obtained partial results in the case $p = 3$, while the case $p > 3$ is still wide open. Cohen and Lenstra originally stated their conjectures only for odd $p$, but the case $p = 2$ is also very interesting. In the case $p = 2$ we have a very explicit description of $\text{Cl}(K)[2]$, and the class group can no longer be thought of as a random object.

Gerth \cite{Gerth} proposed the following modification of the Cohen--Lenstra conjectures; instead of $\text{Cl}(K)[2^\infty]$, it is $(2\text{Cl}(K))[2^\infty]$ that behaves randomly. Fouvry and Kl\"uners \cite{FK2, FK}, building on earlier work of Heath-Brown on $2$-Selmer groups \cite{HB}, proved that $(2\text{Cl}(K))[2]$ has the correct distribution for both imaginary and real quadratic fields. A major breakthrough came when Smith \cite{Smith}, extending earlier work of himself \cite{Smith8}, proved 
\[
\lim_{X \rightarrow \infty} \frac{\left|\left\{K \text{ quadratic} : -X < D_K < 0 \text{ and } (2\text{Cl}(K))[2^\infty] \cong A\right\}\right|}{\left|\left\{K \text{ quadratic} : -X < D_K < 0\right\}\right|} = \frac{\prod_{i = 1}^\infty \left(1 - \frac{1}{2^i}\right)}{|\text{Aut}(A)|}
\]
for all finite, abelian $2$-groups $A$. In the course of the proof Smith develops several powerful and versatile methods. Using the same methods, Smith also deals with the distribution of $2^k$-Selmer groups of elliptic curves. 

Essential to Smith's method is the explicit description of $\text{Cl}(K)[2]$, i.e. genus theory. This allows us to study complicated sets such as $2^{k - 1}\text{Cl}(K)[2^k]$ via its natural inclusion in $\text{Cl}(K)[2]$. Now let $l$ be an odd prime and $K$ be a cyclic degree $l$ field, so that $\text{Cl}(K)$ becomes a $\Z[\zeta_l]$-module in $l - 1$ different ways depending on the identification between $\Gal(K/\Q)$ and $\langle \zeta_l \rangle$. Fortunately, since $K$ is cyclic, the isomorphism type of $\text{Cl}(K)$ as a $\Z[\zeta_l]$-module does not depend on this identification. 

Genus theory gives an explicit description of $\text{Cl}(K)[1 - \zeta_l]$. Klys \cite{Klys} proved conditional on GRH that $((1 - \zeta_l) \text{Cl}(K))[1 - \zeta_l]$ has the expected distribution \cite[p.\ 312]{Gerth2} and also gave an unconditional proof in the case $l = 3$. Both these results use the Fouvry--Kl\"uners method \cite{FK2, FK}. Our main theorem proves that $((1 - \zeta_l) \text{Cl}(K))[(1 - \zeta_l)^\infty]$ has the expected distribution using the breakthrough method of Smith \cite{Smith}.

\begin{theorem}
\label{tMain}
Assume GRH. Then for all odd primes $l$ and all finitely generated, torsion $\Z_l[\zeta_l]$-modules $A$ the limit
\[
\lim_{X \rightarrow \infty} \frac{\left|\left\{K \textup{ cyclic of degree } l : \textup{rad}(D_K) < X \textup{ and } ((1 - \zeta_l)\textup{Cl}(K))\left[(1 - \zeta_l)^\infty\right] \cong A\right\}\right|}{\left|\left\{K \textup{ cyclic of  degree } l : \textup{rad}(D_K) < X\right\}\right|}
\]
exists and is equal to
\[
\frac{\prod_{i = 2}^\infty \left(1 - \frac{1}{l^i}\right)}{|A| \left|\textup{Aut}_{\Z_l[\zeta_l]}(A)\right|}.
\]
\end{theorem}

We order our fields by the radical of the discriminant for technical convenience. The interested reader should have no trouble proving Theorem \ref{tMain} when the fields are instead ordered by the absolute value of the discriminant. Let $\text{Field}(N, l)$ be the set of cyclic degree $l$ number fields $K$ over $\Q$ with $\text{rad}(D_K) \leq N$. For $0 \leq j \leq n$ let $P(j | n)$ be the probability that a uniformly chosen $n \times (n + 1)$ matrix with entries in $\mathbb{F}_l$ has rank $n - j$. Furthermore, for $k \geq 2$ and $n \geq 0$ let $D_{l, k}(n)$ be the set of cyclic degree $l$ fields $K$ satisfying 
\[
\dim_{\FF_l} (1 - \zeta_l)^{k - 1} \text{Cl}(K)\left[(1 - \zeta_l)^k\right] = n.
\]
Theorem \ref{tMain} will fall as a consequence of the following theorem that we prove in Section \ref{sMain}.

\begin{theorem}
\label{tCyclic}
Assume GRH and let $l$ be an odd prime. There are $c, A, N_0 > 0$ such that for all $N > N_0$, all integers $m \geq 2$ and all sequences $n_2 \geq \ldots \geq n_{m + 1} \geq 0$ of integers we have
\begin{multline*}
\left|\left|\textup{Field}(N, l) \cap \bigcap_{k = 2}^{m + 1} D_{l, k}(n_k)\right| - P(n_{m + 1} | n_m) \cdot \left|\textup{Field}(N, l) \cap \bigcap_{k = 2}^m D_{l, k}(n_k)\right|\right| \leq \\
\frac{AN}{\left(\log \log N\right)^{\frac{c}{m^2(l^2 + l)^m}}}.
\end{multline*}
\end{theorem}

One could also wonder what happens without GRH. One of the first steps in Smith's method is to fix the Redei matrix of $K$. This is a rather complicated matter, and Smith proves a weak equidistribution statement with ingenious use of the large sieve. The large sieve for $l$-th power residue symbols is currently not as well-developed as the classical quadratic large sieve. It is for this reason only that we need GRH and it may very well be possible to remove this assumption if one obtains a suitable version of the large sieve. An additional benefit of GRH is that it makes several other proofs in this paper substantially easier and shorter.

If $K$ is a quadratic field, $(2\text{Cl}(K))[2^\infty]$ was traditionally studied from the viewpoint of \emph{governing fields}. Cohn and Lagarias \cite{CL} conjectured that for each integer $k \geq 1$ and each integer $d \not \equiv 2 \bmod 4$, there exists a normal field extension $M_{d, k}$ over $\Q$ such that the $2^k$-rank of $\text{Cl}(\Q(\sqrt{dp}))$ is determined by the splitting of $p$ in $M_{d, k}$. Such a field $M_{d, k}$ is called a governing field. Stevenhagen \cite{Stevenhagen2} proved their conjecture for $k \leq 3$. For $k > 3$ the Cohn and Lagarias conjecture is not known to be true or false for any value of $d$, but widely believed to be false with compelling evidence found by Milovic \cite{Milovic} and later by Koymans and Milovic \cite{KM, KM2, KM3}.

One of the major insights in Smith's work is the notion of a relative governing field. To explain this notion, let $\{p_{1, 0}, p_{1, 1}\}, \ldots, \{p_{k, 0}, p_{k, 1}\}$ be primes and let $d$ be a negative squarefree integer. For any function $f: \{1, \ldots, k\} \rightarrow \{0, 1\}$ define
\[
K(f) := \Q\left(\sqrt{d \prod_{i = 1}^k p_{i, f(i)}}\right).
\]
Choose any function $f': \{1, \ldots, k\} \rightarrow \{0, 1\}$. Under suitable conditions Smith shows that the $2^k$-ranks of $K(f)$ with $f \neq f'$ together with the splitting of $p_{k, 0}$ and $p_{k, 1}$ in a field depending only on $\{p_{1, 0}, p_{1, 1}\}, \ldots, \{p_{k - 1, 0}, p_{k - 1, 1}\}$ determine the $2^k$-rank of $K(f')$. This field can be thought of as a relative governing field, and the resulting theorem can be seen as an extremely general `reflection principle'. Amazingly enough, this is the only algebraic result about class groups used in Smith's paper. The rest of his paper is dedicated to rather ingenious combinatorial and analytical arguments that prove the desired equidistribution.

Our paper borrows heavily from the ideas introduced by Smith; and in particular his proof strategy. We start by generalizing his reflection principle. To do so, we introduce a generalized notion of Smith's relative governing fields. One has to be slightly careful, since Smith relies on the fact that $\zeta_2 \in \Q$. Furthermore, Smith uses that $\Gal(K/\Q) \cong \FF_2$ has trivial automorphism group to make several important identifications. However, if $K$ is cyclic of degree $l$ with $l > 2$, we have $\Gal(K/\Q) \cong \FF_l$, which does not have trivial automorphism group. To work around this, we need to work with characters $\chi: G_\Q \rightarrow \langle \zeta_l \rangle$ instead of fields. During our proofs, it will be very important to carefully keep track of the characters $\chi: G_\Q \rightarrow \langle \zeta_l \rangle$ that we have chosen, since we use these characters to make the necessary identifications in a canonical way.

Once we have generalized Smith's notion of relative governing fields, the proof is mostly a straightforward adaptation of Smith's work with the exception of three major changes. Suppose that we have chosen characters $\chi_p: G_\Q \rightarrow \langle \zeta_l \rangle$ of conductor $p$ for all $p \equiv 1 \bmod l$. We need to deal with sums of the type
\[
\sum_{X <  p < Y} \chi_p(\text{Frob}(q))
\] 
for fixed $q$. If $\chi$ runs over quadratic characters, one may prove cancellation of such sums by an application of Chebotarev or the large sieve. However, if $\chi$ runs over more general characters, such a sum may be biased for a bad choice of the characters $\chi_p$. To work around this issue, we average over all choices of characters $\chi_p$, and use a mixture of Chebotarev and combinatorial arguments to show that there is cancellation for most choices of characters $\chi_p$.

The second issue is the earlier mentioned lack of an appropriate large sieve, and we work around this by assuming GRH. Finally, there is one important point where the analogy between $\text{Cl}(K)[2^\infty]$ for $K$ imaginary quadratic and $\text{Cl}(K)[(1 - \zeta_l)^\infty]$ for $K$ degree $l$ cyclic breaks down. Indeed, the relation between the ramified prime ideals in $\text{Cl}(K)[2]$ is explicitly given by Gauss genus theory. On the other hand, the relation between the ramified prime ideals in $\text{Cl}(K)[1 - \zeta_l]$ should be thought of as being random. It is for this reason that $\text{Cl}(K)[(1 - \zeta_l)^\infty]$ is more similar to $\text{Cl}(K)[2^\infty]$ for $K$ real quadratic. We work around this by using techniques from an unpublished note from the first author, where he extends Smith's work to real quadratic fields.

\section*{Acknowledgements}
We are very grateful to Alexander Smith for answering many questions regarding his paper. We would also like to thank Stephanie Chan, Jan-Hendrik Evertse, Hendrik Lenstra, Djordjo Milovic and Peter Stevenhagen for various useful discussions. The first author greatly appreciates the hospitality of the Max Planck Institute for Mathematics during two visits.

\section{Setup} 
\label{set up}
In this section we introduce the most important objects and notation. Our first subsection defines the central objects in this paper. Once this is done, we devote the next subsection to the necessary notation and conventions.

\subsection{The Artin pairing}
Let $l$ be an odd prime number, which is treated as fixed throughout the paper. Whenever we use $O(\cdot)$ or $\ll$, the implicit constant may depend on $l$ and we shall not record this dependence. Fix once and for all $\overline{\mathbb{Q}}$, an algebraic closure of $\mathbb{Q}$. If $K \subseteq \overline{\mathbb{Q}}$ is a number field, we denote by $G_K := \text{Gal}(\overline{\mathbb{Q}}/K)$. Also fix an element $\zeta_l$ of $\overline{\mathbb{Q}}^\ast$ with multiplicative order equal to $l$; this is a generator of the group $\mu_l(\overline{\mathbb{Q}}) := \{\alpha \in \overline{\mathbb{Q}} : \alpha^l =1 \}$. We define
$$
\Gamma_{\mu_l}(\mathbb{Q}) := \text{Hom}_{\text{top.gr.}}(G_{\mathbb{Q}}, \mu_l(\overline{\mathbb{Q}})).
$$
Here $G_{\mathbb{Q}}$ has the Krull topology and $\mu_l(\overline{\mathbb{Q}})$ the discrete topology. For a character $\chi \in \Gamma_{\mu_l}(\mathbb{Q})$, we denote by 
$$
K_\chi := \overline{\mathbb{Q}}^{\text{ker}(\chi)}
$$
the corresponding extension of $\mathbb{Q}$. This is a cyclic extension with degree dividing $l$, and equal to $1$ if and only if $\chi$ is the trivial character. We denote by $\text{Cl}(K_\chi)$ the class group of $K_\chi$. Observe that $\text{Cl}(K_\chi)[l^\infty]$ is a $\mathbb{Z}_l[\text{Gal}(K_\chi/\mathbb{Q})]$-module. Since $\mathbb{Z}$ is a PID, the norm element
$$
N_{\text{Gal}(K_\chi/\mathbb{Q})} := \sum_{g \in \text{Gal}(K_\chi/\mathbb{Q})} g
$$ 
acts trivially on $\text{Cl}(K_\chi)$. From this, we deduce that $\text{Cl}(K_\chi)[l^\infty]$ has naturally the structure of a $\frac{\mathbb{Z}_l[\text{Gal}(K_\chi/\mathbb{Q})]}{N_{\text{Gal}(K_\chi /\mathbb{Q})}}$-module. Moreover, $\chi$ gives a natural isomorphism of $\mathbb{Z}_l$-algebras
$$
\chi:\frac{\mathbb{Z}_l[\text{Gal}(K_\chi/\mathbb{Q})]}{N_{\text{Gal}(K_\chi/\mathbb{Q})}} \to \mathbb{Z}_l[\zeta_l] := \mathbb{Z}[\zeta_l] \otimes_{\mathbb{Z}} \mathbb{Z}_l.
$$
In this manner $\text{Cl}(K_\chi)[l^\infty]$ is naturally equipped with the structure of a $\mathbb{Z}_l[\zeta_l]$-module. In what follows, it is \emph{always} with respect to this structure that we will talk about $\text{Cl}(K_\chi)[l^\infty]$ as a $\mathbb{Z}_l[\zeta_l]$-module.

The ring $\mathbb{Z}_l[\zeta_l]$ is a local PID with the unique maximal ideal generated by $1 - \zeta_l$. Therefore for every finite $\mathbb{Z}_l[\zeta_l]$-module $A$, there is a unique function $f_{A}:\mathbb{Z}_{\geq 1} \to \mathbb{Z}_{\geq 0}$ such that
$$
A \simeq_{\mathbb{Z}_l[\zeta_l]} \bigoplus_{i \in \mathbb{Z}_{\geq 1}} \left(\frac{\mathbb{Z}_l[\zeta_l]}{(1 - \zeta_l)^{i}\mathbb{Z}_l[\zeta_l]} \right)^{f_A(i)}.
$$
Since $A$ is finite, the map $f_A$ has finite support and it can be reconstructed from the decreasing sequence of numbers
$$
k \mapsto \text{rk}_{(1 - \zeta_l)^k} A := \text{dim}_{\mathbb{F}_l} (1 - \zeta_l)^{k - 1} A[(1 - \zeta_l)^k],
$$
defined for every positive integer $k$. Therefore the sequence 
$$
\left\{\text{rk}_{(1 - \zeta_l)^k} \text{Cl}(K_\chi) \right\}_{k \in \mathbb{Z}_{\geq 1}}
$$
determines completely the structure of the $\mathbb{Z}_l[\zeta_l]$-module $\text{Cl}(K_\chi)[l^\infty]$. Here, for brevity, $\text{rk}_{(1 - \zeta_l)^k} \text{Cl}(K_\chi)$ stands for $\text{rk}_{(1 - \zeta_l)^k}\text{Cl}(K_\chi)[l^\infty]$, which has been defined above. The following $\mathbb{Z}_l[\zeta_l]$-module will have a big role for us 
$$
N := \frac{\mathbb{Q}_l(\zeta_l)}{\mathbb{Z}_l[\zeta_l]}.
$$
For a finitely generated, torsion $\mathbb{Z}_l[\zeta_l]$-module $A$, we define
$$
A^{\vee} := \text{Hom}_{\mathbb{Z}_l[\zeta_l]}(A, N).
$$
The following is not hard to see.

\begin{prop} 
\label{A isomorphic to its dual}
The $\mathbb{Z}_l[\zeta_l]$-modules $A$ and $A^{\vee}$ are isomorphic. 
\end{prop}

For every integer $k \geq 1$ we next define a pairing of $\mathbb{Z}_l[\zeta_l]$-modules
$$
\text{Art}_k(A) : (1 - \zeta_l)^{k - 1}A[(1 - \zeta_l)^k] \times (1 - \zeta_l)^{k - 1}A^{\vee}[(1 - \zeta_l)^k] \to N[1 - \zeta_l].
$$
Let $a \in (1 - \zeta_l)^{k - 1}A[(1 - \zeta_l)^k]$ and $\chi \in (1 - \zeta_l)^{k - 1}A^{\vee}[(1 - \zeta_l)^k]$. Let $\psi \in A^{\vee}$ be an element such that $(1 - \zeta_l)^{k - 1} \psi = \chi$. We put
$$
\text{Art}_k(A)(a,\chi) := \psi(a).
$$
Observe that, since $a \in (1 - \zeta_l)^{k - 1}A[(1 - \zeta_l)^k]$, the definition does not depend on the choice of $\psi$. The following fact is straightforward.

\begin{prop} 
\label{abstract Artin pairing}
The left-kernel of $\emph{Art}_k(A)$ is $(1 - \zeta_l)^k A[(1 - \zeta_l)^{k+1}]$ and the right-kernel is $(1 - \zeta_l)^k A^{\vee}[(1 - \zeta_l)^{k+1}]$.
\end{prop}

Hence, instead of directly dealing with
$$
\left\{\text{rk}_{(1 - \zeta_l)^k} \text{Cl}(K_\chi)\right\}_{k \in \mathbb{Z}_{\geq 1}},
$$
our goal is to control the sequence of pairings
$$
\left\{\text{Art}_k(\text{Cl}(K_\chi))\right\}_{k \in \mathbb{Z}_{\geq 1}}.
$$
Here $\text{Art}_k(\text{Cl}(K_\chi))$ is an abbreviation for $\text{Art}_k(\text{Cl}(K_\chi)[l^\infty])$, which has been defined above. Proving equidistribution for this sequence of pairings is the main goal of this paper. We start with some algebraic tools, which culminate in an extremely general reflection principle. In the next section we fix identifications between some cyclic groups of order $l$ that occur in this paper, as well as some other important conventions regarding notation. 

\subsection{Identifications and conventions} 
\label{conventions}
Throughout the paper we will encounter the groups $\mathbb{F}_l$, $\langle\zeta_l\rangle$ and $N[1 - \zeta_l]$. These three groups are isomorphic, but not in a canonical way. Working with each group has its own advantages. Kummer theory is most naturally stated using $\langle \zeta_l \rangle$, while $\mathbb{F}_l$ has a natural product structure that we will take advantage of. Finally, $N[1 - \zeta_l]$ is a subgroup of $N$, which is the image of the various Artin pairings. We need to identify these three groups at several points in the paper, and it is of utmost importance this is done in a consistent matter. We refer to the following diagram whenever such an identification is made.

\begin{center}
\begin{tikzpicture}
\tikzset{myptr/.style={decoration={markings,mark=at position 1 with %
    {\arrow[scale=2.5,>=stealth]{>}}},postaction={decorate}}}
\draw [->,>=stealth] (0, 0) -- (2, 2) node[anchor=south]{$\langle \zeta_l \rangle$};
\draw [myptr] (0,0) -- node[above = 0.25cm, left = 0.5cm]{$j_l(a) := \zeta_l^a$} (2,2);
\draw [->,>=stealth] (2, 2) -- (4, 0) node[anchor=north]{$N[1 - \zeta_l]$};
\draw [myptr] (2,2) -- node[above = 0.25cm, right = 0.5cm]{$i_l(\zeta_l^a) := \frac{a}{1 - \zeta_l}$} (4,0);
\draw [->,>=stealth] node[anchor=north]{$\mathbb{F}_l$} (0, 0) -- (4, 0);
\draw [myptr] (0,0) -- node[above]{$i_l \circ j_l$} (4,0);
\end{tikzpicture}
\end{center}

Any other identification is made by inverting the arrows and maps. The symbol $\mathbb{C}$ will denote, as usual, the complex numbers. The symbol $\text{i}$ denotes a fixed element of $\mathbb{C}^{\ast }$ of multiplicative order equal to $4$. The function $\exp : \mathbb{C} \to \mathbb{C}^\ast$ denotes the exponential map. The group $\mu_l(\mathbb{C})$ is generated by the element $\exp\left(\frac{2\pi \text{i}}{l}\right)$. We also fix the identification $h_l: \mu_l(\mathbb{C}) \to \langle \zeta_l \rangle$ given by
$$
h_l\left(\exp\left(\frac{2\pi \text{i}}{l}\right)\right) := \zeta_l.
$$
We denote by $\Gamma_{\mathbb{F}_l}(\mathbb{Q}):=\text{Hom}_{\text{top.gr.}}(G_{\mathbb{Q}},\mathbb{F}_l)$. The map $j_l$ induces an isomorphism $\Gamma_{\mathbb{F}_l}(\mathbb{Q}) \to \Gamma_{\mu_l}(\mathbb{Q})$. We also define $\Gamma_{\mu_l(\mathbb{C})}(\mathbb{Q}):=\text{Hom}_{\text{top.gr.}}(G_{\mathbb{Q}},\mu_l(\mathbb{C}))$, so $h_l$ induces an isomorphism $\Gamma_{\mu_l(\mathbb{C})}(\mathbb{Q}) \to \Gamma_{\mu_l}(\mathbb{Q})$.

If $q$ is either equal to $l$ or to a prime number that is congruent to $1$ modulo $l$, then there exists a unique cyclic degree $l$ extension of $\mathbb{Q}$ that is totally ramified at $q$ and unramified elsewhere. By class field theory, if $q \neq l$, this is the unique cyclic degree $l$ extension contained in $\mathbb{Q}(\zeta_q)/\mathbb{Q}$. If $q = l$ this is the unique cyclic degree $l$ extension contained in $\mathbb{Q}(\zeta_{l^2})/\mathbb{Q}$. Here $\zeta_q$ and $\zeta_{l^2}$ are elements of $\overline{\mathbb{Q}}$ of multiplicative order equal to $q$ and $l^2$ respectively. We denote these extensions by $L_q$, for $q \neq l$, and by $L_{l^2}$ in the case $q = l$. For each $q$ congruent $1$ modulo $l$ we fix a character
$$ 
\chi_q \in \Gamma_{\mu_l}(\mathbb{Q})
$$
such that $\text{ker}(\chi_q) = G_{L_q}$. There is no way to make such a choice in a canonical manner, and we fix one simply for notational purposes. Similarly, we fix a character
$$ 
\chi_l \in \Gamma_{\mu_l}(\mathbb{Q})
$$
such that $\text{ker}(\chi_l) = G_{L_{l^2}}$. All our algebraic results work for a fixed choice of characters, but later on we will have to vary the choice of characters to make our analytic results work.

The set $\{\chi_q\}_{q \equiv 1 \bmod l} \cup \{\chi_{l}\}$ is a \emph{basis} for $\Gamma_{\mu_l}(\mathbb{Q})$. In particular any cyclic degree $l$ extension ramifies only at primes congruent $1$ modulo $l$ or at $l$, see Proposition \ref{only p and 1 mod p} for a generalization of this fact. By the conductor-discriminant formula we see that a positive integer $D$ equals $\Delta_{K_\chi /\mathbb{Q}}$ for some $\chi \in \Gamma_{\mu_l}(\mathbb{Q})$ if and only if 
\[
D=(q_1 \cdot \ldots \cdot q_r)^{l-1}
\]
with $\{q_i\}_{1 \leq i \leq r}$ a set of $r$ distinct elements each a prime $1$ modulo $l$ or equal to $l^2$. In case $D$ admits such a factorization then $D=\Delta_{K_\chi /\mathbb{Q}}$ for $(l - 1)^r$ different choices of $\chi$, which amounts to a total of $(l - 1)^{r-  1}$ different fields. From now on we say that a positive integer $D$ is $l$-\emph{admissible} if it is the discriminant of a cyclic degree $l$ extension of $\mathbb{Q}$. For each $l$-admissible integer $D$ we call an \emph{amalgama for $D$} a map $\epsilon:\{q \mid D\}_{q \ \text{prime}} \to [l - 1]$. For an $l$-admissible integer $D$, the set of characters $\chi \in \Gamma_{\mu_l}(\mathbb{Q})$ such that $\Delta_{K_\chi /\mathbb{Q}}=D$ corresponds bijectively to the set of amalgamas for $D$, via the assignment
$$
\epsilon \mapsto \chi_{\epsilon}(D) := \prod_{q \mid D} \chi_q^{\epsilon(q)}.
$$
We denote by $\chi \mapsto \epsilon_{\chi}$ the inverse assignment. Let $\chi \in \Gamma_{\mu_l}(\mathbb{Q})$ and $q \mid \Delta_{K_\chi /\mathbb{Q}}$. Then there is a unique prime ideal in $O_{K_\chi}$ lying above $q$. We denote such a prime ideal with $\text{Up}_{K_\chi}(q)$. For a positive integer $b$ and for a prime number $q$ dividing $b$, we write 
$$
\epsilon_b(q)
$$
for the unique integer in $\{0, \ldots, l - 1\}$ with $\epsilon_b(q) \equiv v_{\mathbb{Q}_q}(b) \bmod l$. In particular we have that $\prod_{q \mid b}q^{\epsilon_b(q)}$ equals $b$ in $\frac{\mathbb{Q}^\ast}{\mathbb{Q}^{\ast l}}$. We also define $[d] := \{1, \ldots, d\}$ for any integer $d$.

We shall frequently encounter maps from some profinite group $G$ to some finite set $X$. Whenever we encounter such a map, it will be continuous with respect to the discrete topology on $X$.

\begin{lemma}
\label{lFod}
Let $G$ be a profinite group, let $X$ be a discrete topological space and $\phi: G \to X$ a continuous map. There exists a largest (by inclusion) open normal subgroup $N_\phi$ of $G$ such that the map $\phi$ factors through the canonical projection $G \to G/N_{\phi}$.
\end{lemma}

\begin{proof}
This is straightforward.
\end{proof}

\noindent Thanks to Lemma \ref{lFod} we can make the following definition.

\begin{mydef}
Let $\phi: G_\Q \rightarrow X$ be a continuous map, where $X$ is a discrete topological space. The group of definition of $\phi$ is the open normal subgroup $N_\phi$ in Lemma \ref{lFod}. Furthermore, we define $L(\phi)$ to be the fixed field of $N_\phi$, which will be called the field of definition of $\phi$.
\end{mydef}
\section{Ambiguous ideals and genus theory} 
In this section we study $\text{Cl}(K_\chi)[1 - \zeta_l]$ and $\text{Cl}(K_\chi)^\vee[1 - \zeta_l]$. The material collected here is well-known to experts and can be found in various forms in the literature, but we have decided to include it for the sake of completeness.

\subsection{Ambiguous ideals}
\label{invariants}
The material in this subsection is known as the theory of ambiguous ideals. Since in $\mathbb{Z}_l[\zeta_l]$ we have the equality of ideals $(1 - \zeta_l)^{l - 1} = (l)$, we in particular have that $\text{Cl}(K_\chi)[1 - \zeta_l]$ is an $\mathbb{F}_l$-vector space. From the definition of the structure of $\text{Cl}(K_\chi)[l^\infty]$ as a $\mathbb{Z}_l[\zeta_l]$-module, it is clear that $\text{Cl}(K_\chi)[1 - \zeta_l] = \text{Cl}(K_\chi)^{\text{Gal}(K_\chi /\mathbb{Q})}$. Thus we can obtain a description of $\text{Cl}(K_\chi)[1 - \zeta_l]$ by taking Galois invariants of the sequence
$$
1 \to \text{Pr}(K_\chi) \to \mathcal{I}_{K_\chi} \to \text{Cl}(K_\chi) \to 1,
$$ 
where $\mathcal{I}_{K_\chi}$ denotes the group of fractional ideals of $O_{K_\chi}$ and $\text{Pr}(K_\chi)$ denotes the group of principal fractional ideals of $O_{K_\chi}$. To take advantage of this sequence we shall begin with the following simple fact.

\begin{prop}
We have that 
$$
H^1(\emph{Gal}(K_\chi /\mathbb{Q}), \emph{Pr}(K_\chi))=0.
$$
\end{prop}

\begin{proof}
We take the exact sequence
$$
1 \to O_{K_\chi}^{\ast } \to K_\chi^\ast \to \text{Pr}(K_\chi) \to 1.
$$
Thanks to Hilbert $90$, we can canonically identify the $H^1$ in the statement with
$$
\text{ker}\left(H^2\left(\text{Gal}(K_\chi /\mathbb{Q}), O_{K_\chi}^\ast\right) \to H^2\left(\text{Gal}(K_\chi /\mathbb{Q}),K_\chi ^\ast\right)\right).
$$
Hence it is \emph{sufficient} to show that $H^2\left(\text{Gal}(K_\chi /\mathbb{Q}),O_{K_\chi}^\ast\right) = 0$. Since $\text{Gal}(K_\chi /\mathbb{Q})$ is cyclic, it follows from \cite[Section 6.2]{Weibel} that this last $H^2$ is isomorphic to 
$$
\frac{\mathbb{Z}^{\ast }}{N_{K_\chi /\mathbb{Q}}(O_{K_\chi}^{\ast })} = \frac{\langle -1 \rangle}{\langle -1 \rangle}= \{1\},
$$
where in the first equality we have used that $l$ is odd. This concludes the proof.
\end{proof}

\noindent Therefore we have the following corollary.

\begin{corollary} 
\label{invariant ideals surjects on invariant class group}
The natural map $\mathcal{I}_{K_\chi}^{\emph{Gal}(K_\chi /\mathbb{Q})} \to \emph{Cl}(K_\chi)^{\emph{Gal}(K_\chi /\mathbb{Q})}$ induces an isomorphism
$$
\frac{\mathcal{I}_{K_\chi}^{\emph{Gal}(K_\chi /\mathbb{Q})}}{\emph{Pr}(K_\chi)^{\emph{Gal}(K_\chi /\mathbb{Q})}} \simeq \emph{Cl}(K_\chi)^{\emph{Gal}(K_\chi /\mathbb{Q})}.
$$
\end{corollary}

We next focus on the group $\mathcal{I}_{K_\chi}^{\text{Gal}(K_\chi /\mathbb{Q})}$. Recall from Section \ref{conventions} that for a prime $q$ with $q \mid \Delta_{K_\chi /\mathbb{Q}}$ the symbol $\text{Up}_{K_\chi}(q)$ denotes the unique prime ideal of $O_{K_\chi}$ lying above $q$. We have the following fact.

\begin{prop} 
\label{fixed ideals}
For any element $I \in \mathcal{I}_{K_\chi}^{\emph{Gal}(K_\chi /\mathbb{Q})}$ there is a unique pair $(\epsilon, n)$ where $\epsilon$ is a map $\epsilon:\{q \mid \Delta_{K_\chi /\mathbb{Q}}\}_{q \emph{ prime}} \to \{0,\ldots, l - 1\}$ and $n$ is a positive rational number, with the property
$$
I = (n) \prod_{q \mid \Delta_{K_\chi /\mathbb{Q}}}\emph{Up}_{K_\chi}(q)^{\epsilon(q)}.
$$
\end{prop}

\begin{proof}
Let $I$ be in $\mathcal{I}_{K_\chi}^{\text{Gal}(K_\chi /\mathbb{Q})}$ and factor $I$ as a product of prime ideals. Then inert primes can clearly be bunched together into a rational fractional ideal. For split primes, the Galois invariance and unique factorization of ideals imply that every exponent is invariant in each Galois orbit of split primes. Hence also the split primes can be bunched together to give a total contribution that is a rational fractional ideal. 

The remaining primes are exactly the ramified primes. For each ramified prime, we can always pick the largest multiple of $l$ smaller than the exponent, and throw this contribution into a rational fractional ideal. This shows the existence part of the proposition.

For the uniqueness, suppose that the pairs $(\epsilon_1,n_1)$ and $(\epsilon_2,n_2)$ give the same ideal. Observe that when we norm down to $\mathbb{Q}$, we obtain for $q$ a prime not dividing the discriminant that $l \cdot v_{\mathbb{Q}_q}(n_1) = l \cdot v_{\mathbb{Q}_q}(n_2)$ and hence $v_{\mathbb{Q}_q}(n_1) = v_{\mathbb{Q}_q}(n_2)$. Finally, if $q$ is ramified, we obtain 
\[
l \cdot v_{\mathbb{Q}_q}(n_1) + \epsilon_1(q) = l \cdot v_{\mathbb{Q}_q}(n_2) + \epsilon_2(q).
\]
Since $\epsilon_1(q)$ and $\epsilon_2(q)$ are in $\{0, \ldots, l - 1\}$, it must be that $\epsilon_1(q)=\epsilon_2(q)$ and $v_{\mathbb{Q}_q}(n_1) = v_{\mathbb{Q}_q}(n_2)$. So we have that $\epsilon_1=\epsilon_2$ and the two rational numbers $n_1$ and $n_2$ have the same valuation at all finite places and they are both positive, hence they coincide. 
\end{proof}

Since the surjection $\mathcal{I}_{K_\chi}^{\text{Gal}(K_\chi /\mathbb{Q})} \twoheadrightarrow \text{Cl}(K_\chi)[1 - \zeta_l]$ factors through $\mathcal{I}_{\mathbb{Q}}$, it induces a surjective map of $\mathbb{F}_l$-vector spaces
$$
\overline{\text{Cl}}(K_\chi) \twoheadrightarrow \text{Cl}(K_\chi)[1 - \zeta_l],
$$
where $\overline{\text{Cl}}(K_\chi)$ denotes the subgroup consisting of those $\alpha$ in $\frac{\mathbb{Q}^{\ast }}{\mathbb{Q}^{\ast l}}$ such that $v(\alpha)$ is divisible by $l$ for all places $v \in \Omega_{\mathbb{Q}}$ not dividing $\Delta_{K_\chi /\mathbb{Q}}$. It follows from Proposition \ref{fixed ideals} that we have an identification 
\[
\overline{\text{Cl}}(K_\chi) \simeq \frac{\mathcal{I}_{K_\chi}^{\text{Gal}(K_\chi /\mathbb{Q})}}{\mathcal{I}_{\mathbb{Q}}}
\]
via the norm map, which sends every invariant non-zero ideal to its norm in $\frac{\mathbb{Q}^{\ast}}{\mathbb{Q}^{\ast l}}$. By construction, the kernel of the map $\overline{\text{Cl}}(K_\chi) \twoheadrightarrow \text{Cl}(K_\chi)[1 - \zeta_l]$, equals 
$$
\frac{\text{Pr}(K_\chi)^{\text{Gal}(K_\chi /\mathbb{Q})}}{\mathcal{I}_{\mathbb{Q}}}.
$$ 
Due to Hilbert $90$ this group is canonically isomorphic to $H^1\left(\text{Gal}(K_\chi /\mathbb{Q}), O_{K_\chi}^\ast\right)$. 

\begin{prop} 
\label{units are locally free}
The group $H^1\left(\emph{Gal}(K_\chi /\mathbb{Q}),O_{K_\chi}^\ast\right)$ is an $1$-dimensional $\mathbb{F}_l$-vector space. 
\end{prop}

\begin{proof}
Fix a non-trivial element $\sigma$ of $\text{Gal}(K_\chi /\mathbb{Q})$. Then $\sigma$ generates $\text{Gal}(K_\chi /\mathbb{Q})$. Since $\text{Gal}(K_\chi /\mathbb{Q})$ is a cyclic group, an elementary calculation with $1$-cocycles shows that the $H^1$ in the statement is isomorphic to 
\[
\frac{\left\{\alpha \in O_{K_\chi}^\ast : N_{K_\chi /\mathbb{Q}}(\alpha) = 1\right\}}{\left\{\frac{\sigma(\beta)}{\beta} :  \beta \in O_{K_\chi}^\ast\right\}}.
\]
Observe that this is an $\mathbb{F}_l$-vector space, which also follows from the size of the Galois group being $l$. Thus we can compute it also by first completing at $l$, i.e. considering 
\[
\left\{\alpha \in O_{K_\chi}^\ast : N_{K_\chi /\mathbb{Q}}(\alpha) = 1\right\} \otimes_{\mathbb{Z}} \mathbb{Z}_l = O_{K_\chi}^{\ast } \otimes_{\mathbb{Z}} \mathbb{Z}_l.
\]
Through $\chi$, this can be naturally viewed as a $\mathbb{Z}_l[\zeta_l]$-module. But the $\mathbb{Z}_l$-torsion of $O_{K_\chi}^{\ast } \otimes_{\mathbb{Z}} \mathbb{Z}_l$ is trivial, since $l$ is odd and the $\mathbb{Z}$-torsion of $O_{K_\chi}^{\ast }$ is equal to $\langle -1 \rangle$. Hence the $\mathbb{Z}_l[\zeta_l]$-torsion is also trivial. Therefore, using Dirichlet's Unit Theorem and the fact that $l$ is odd, it must be that
$$
O_{K_\chi}^\ast \otimes_{\mathbb{Z}} \mathbb{Z}_l \simeq_{\mathbb{Z}_l[\zeta_l]} \mathbb{Z}_l[\zeta_l].
$$
Therefore the $H^1$ we are after is isomorphic to $\frac{\mathbb{Z}_l[\zeta_l]}{(1 - \zeta_l)} \simeq \mathbb{F}_l$.
\end{proof}

\noindent Combining Corollary \ref{invariant ideals surjects on invariant class group} with Proposition \ref{units are locally free}, we conclude the following.

\begin{corollary} 
\label{main thm on ambiguous ideals}
The map
$$
\overline{\textup{Cl}}(K_\chi) \twoheadrightarrow \emph{Cl}(K_\chi)[1 - \zeta_l]
$$
has $1$-dimensional kernel. Moreover, a generator of the kernel can be obtained by taking $N_{K_\chi /\mathbb{Q}}(\gamma) \in \frac{\mathbb{Q}^\ast}{\mathbb{Q}^{\ast l}}$ for any $\gamma \in K_\chi ^\ast$ such that
\[
\frac{\sigma(\gamma)}{\gamma} \in O_{K_\chi}^\ast - \left\{\frac{\sigma(\beta)}{\beta}: \beta \in O_{K_\chi}^\ast\right\},
\]
where $\sigma$ is any non-trivial element of $\textup{Gal}(K_\chi /\mathbb{Q})$.
\end{corollary}

The second part of Corollary \ref{main thm on ambiguous ideals} suggests that we should not expect a simple description for the relation among the ramified prime ideals as is the case for the $2$-torsion of imaginary quadratic number fields. Instead we should expect it to be a genuine `random' piece of data. It is for this reason that cyclic degree $l$ fields are analogous to real quadratic number fields.

\subsection{Genus theory}
\label{dual invariants}
The material of the previous subsection is sufficient to determine the structure of $\text{Cl}(K_\chi)^{\vee}[1 - \zeta_l]$ right away. We remind the reader that we have fixed an identification $i_l$ between $\langle \zeta_l \rangle$ and $N[1 - \zeta_l]$ in Section \ref{conventions}. Also recall that we have a canonical identification $\text{Cl}(K_\chi) = \text{Gal}(H_{K_\chi}/K_\chi)$ via the Artin map, where $H_{K_\chi}$ is the Hilbert class field of $K_\chi$. Finally, $\chi_p$ denotes a fixed choice of an element in $\Gamma_{\mu_l}(\mathbb{Q})$ of conductor dividing $p^\infty$.

\begin{prop} 
\label{Genus theory}
Let $\chi$ be in $\Gamma_{\mu_l}(\mathbb{Q})$. The set $\{i_l \circ \chi_p|_{\emph{ker}(\chi)}\}_{p \mid \Delta_{K_\chi /\mathbb{Q}}}$ is a generating set for $\emph{Cl}(K_\chi)^{\vee}[1 - \zeta_l]$. Moreover, any relation among these characters is a multiple of the trivial relation $\chi|_{G_{K_\chi}}=1$.
\end{prop}

\begin{proof}
It is clear that the set $\{i_l \circ \chi_p|_{\text{ker}(\chi)}\}_{p \mid \Delta_{K_\chi /\mathbb{Q}}}$ belongs to $\text{Cl}(K_\chi)^{\vee}[1 - \zeta_l]$, with any relation a multiple of the trivial one. This combined with Proposition \ref{A isomorphic to its dual} and Corollary \ref{main thm on ambiguous ideals} gives the conclusion by counting. 
\end{proof}

It is possible to give a more direct proof of Proposition \ref{Genus theory} by completely different considerations. This relies on the following fundamental fact that will anyway play a crucial role for us.

\begin{prop} 
\label{when abelian by cyclic splits}
Let $F/E$ be a cyclic degree $l$ extension of number fields. Let $v \in \Omega_E$ be a place of $E$ ramifying in $F/E$. Suppose moreover that $L/F$ is an abelian extension of $F$, Galois over $E$ and unramified at the unique place in $\Omega_F$ lying above $v$. Then the exact sequence
$$1 \to \emph{Gal}(L/F) \to \emph{Gal}(L/E) \to \emph{Gal}(F/E) \to 1
$$
splits, i.e. the surjection $\emph{Gal}(L/E) \twoheadrightarrow \emph{Gal}(F/E)$ admits a section.
\end{prop}

\begin{proof}
For a number field $F$, we let $\Omega_F$ be the set of places of $F$. Since $L/F$ is unramified at the unique place in $\Omega_F$ lying above $v$, and since $v$ is totally ramified in $F/E$, we deduce that for every place $\tilde{v} \in \Omega_{L}$ lying above $v$ in $L/E$, the inertia group $I_{\tilde{v}/v}$ has size exactly $l$. Therefore it is either fully contained in $\text{Gal}(L/F)$ or it intersects $\text{Gal}(L/F)$ trivially providing the claimed section. 

We assume that $I_{\tilde{v}/v}$ is fully contained in $\text{Gal}(L/F)$ and derive a contradiction. Since $\text{Gal}(L/F)$ is a normal subgroup, it follows that all conjugates of $I_{\tilde{v}/v}$ are also contained in $\text{Gal}(L/F)$. Then we conclude that
\[
\text{Gal}(L/F) \supseteq \prod_{w \in \Omega_L: w \mid v}I_{w/v},
\]
which is equivalent to 
\[
F \subseteq \bigcap_{w \in \Omega_L: w \mid v}L^{I_{w/v}}.
\]
This implies that $v$ is unramified in $F/E$, which is the desired contradiction.
\end{proof}

\emph{Alternative proof of Proposition \ref{Genus theory}}: Let $L/K_\chi $ be a degree $l$ extension coming from a character in $\text{Cl}(K_\chi)^{\vee}[1 - \zeta_l]$. Note that $L$ is a Galois extension of $\Q$ with degree $l^2$. By Proposition \ref{when abelian by cyclic splits} we conclude that 
\[
\text{Gal}(L/\mathbb{Q}) \simeq \mathbb{F}_l \times \mathbb{F}_l.
\]
Hence the extension is of the shape $KK_\chi/K_\chi$, where $K$ is a cyclic degree $l$ extension of $\mathbb{Q}$. Let $\chi' \in \Gamma_{\mu_l}(\mathbb{Q})$ be a character with $K_{\chi'} = K$. If $\chi'$ ramifies at any prime $q$ not dividing $\Delta_{K_\chi /\mathbb{Q}}$, then the resulting extension of $K_\chi $ will ramify at the primes of $K_\chi $ lying above $q$. Hence $\chi'$ must be one of the characters listed in Proposition \ref{Genus theory}, which clearly satisfy only the trivial relation stated there.

\section{Central extensions}
\label{central extensions}
In this section we prove several important facts about central $\mathbb{F}_l$-extensions that we will extensively use in the coming sections to deal with the first and higher Artin pairings. Let $E$ be a field and fix a separable closure $E^{\text{sep}}$ of $E$. We extend the notation from Subsection \ref{conventions} in the natural way to our more general setting; in particular we will use $G_E$ and $\Gamma_{\mathbb{F}_l}(E)$ without further introduction. Let $F \subseteq E^{\text{sep}}$ be a finite Galois extension of $E$. The most important object in this section is the set
$$
\text{Cent}_{\mathbb{F}_l}(F/E)
$$
consisting of degree dividing $l$ extensions $\tilde{F}/F$ in $E^{\text{sep}}$ that are Galois over $E$ and such that $\text{Gal}(\tilde{F}/F)$ is a central subgroup of $\text{Gal}(\tilde{F}/E)$. For $\tilde{F}_1, \tilde{F}_2$ in $\text{Cent}_{\mathbb{F}_l}(F/E)$ we say that $\tilde{F}_1$ is equivalent to $\tilde{F}_2$ if at least one of the following two statements is true
\begin{itemize}
\item we have $\tilde{F}_1 = \tilde{F}_2 = F$;
\item the compositum of $\tilde{F}_1\tilde{F}_2$ contains a non-trivial cyclic degree $l$ extension of $F$ that is obtained as $\tilde{E}F$, where $\tilde{E}$ is a cyclic degree $l$ extension of $E$.
\end{itemize}
One can easily see that this is an equivalence relation, and we use the symbol $\tilde{F}_1 \sim \tilde{F}_2$ to express the fact that $\tilde{F}_1$ is equivalent to $\tilde{F}_2$. 

Every cyclic extension of $F$ with degree dividing $l$ can be obtained as $(E^{\text{sep}})^{\text{ker}(\chi)}$ for some $\chi \in \Gamma_{\mathbb{F}_l}(F)$. This can be done only with the trivial character if the extension is trivial, and in all the other cases with precisely the $l - 1$ non-zero multiples of a non-trivial character. It can be easily seen that for this extension to be in $\text{Cent}_{\mathbb{F}_l}(F/E)$ it is necessary and sufficient that $\chi$ is fixed by the natural action of $\text{Gal}(F/E)$ on $\Gamma_{\mathbb{F}_l}(F)$. Therefore we have a natural surjective map
$$
\Gamma_{\mathbb{F}_l}(F)^{\text{Gal}(F/E)} \twoheadrightarrow \text{Cent}_{\mathbb{F}_l}(F/E),
$$
which descends to a map
$$
\frac{\Gamma_{\mathbb{F}_l}(F)^{\text{Gal}(F/E)}}{\Gamma_{\mathbb{F}_l}(E)} \twoheadrightarrow \text{Cent}_{\mathbb{F}_l}(F/E)/\sim.
$$
The map attains the class of $F/E$ precisely on the trivial element of $\frac{\Gamma_{\mathbb{F}_l}(F)^{\text{Gal}(F/E)}}{\Gamma_{\mathbb{F}_l}(E)}$ and on the remaining points is a $(l-1):1$ assignment.

Given a class $\chi \in \frac{\Gamma_{\mathbb{F}_l}(F)^{\text{Gal}(F/E)}}{\Gamma_{\mathbb{F}_l}(E)}$ we can naturally attach a class 
$$ 
r_1(\chi) \in H^2(\text{Gal}(F/E), \mathbb{F}_l),
$$
where the implicit action is declared to be trivial. This uses the group-theoretic interpretation of $H^2(\text{Gal}(F/E), \mathbb{F}_l)$, and goes as follows. Using $\chi$ we have an identification 
$$
\text{Gal}\left((E^{\text{sep}})^{\text{ker}(\chi)}/F\right) \simeq \mathbb{F}_l.
$$
Therefore this transforms the sequence
$$
1 \to \text{Gal}\left((E^{\text{sep}})^{\text{ker}(\chi)}/F\right) \to \text{Gal}\left((E^{\text{sep}})^{\text{ker}(\chi)}/E\right) \to \text{Gal}(F/E) \to 1
$$ 
into a sequence 
$$
0 \to \mathbb{F}_l \to \text{Gal}\left((E^{\text{sep}})^{\text{ker}(\chi)}/E\right) \to \text{Gal}(F/E) \to 1,
$$ 
which naturally provides us with a class $r_1(\chi) \in H^2(\text{Gal}(F/E), \mathbb{F}_l)$. It is not hard to show that the resulting map
$$
r_1 : \frac{\Gamma_{\mathbb{F}_l}(F)^{\text{Gal}(F/E)}}{\Gamma_{\mathbb{F}_l}(E)} \to H^2(\text{Gal}(F/E), \mathbb{F}_l)
$$
is an injective group homomorphism. Hence we can use $r_1$ to identify the group $\frac{\Gamma_{\mathbb{F}_l}(F)^{\text{Gal}(F/E)}}{\Gamma_{\mathbb{F}_l}(E)}$ with its image in $H^2(\text{Gal}(F/E), \mathbb{F}_l)$. We define 
$$
\widetilde{\text{Cent}}_{\mathbb{F}_l}(F/E) := \text{Im}(r_1).
$$
Our next goal is to characterize the elements of $H^2(\text{Gal}(F/E), \mathbb{F}_l)$ that belong to $\widetilde{\text{Cent}}_{\mathbb{F}_l}(F/E)$. To do so we write another natural map
$$
r_2: \frac{\Gamma_{\mathbb{F}_l}(F)^{\text{Gal}(F/E)}}{\Gamma_{\mathbb{F}_l}(E)} \to H^2(\text{Gal}(F/E), \mathbb{F}_l),
$$
this time using Galois cohomology in the following manner. Firstly observe that $\Gamma_{\mathbb{F}_l}(F) = H^1(G_F, \mathbb{F}_l)$ and $\Gamma_{\mathbb{F}_l}(E) = H^1(G_E, \mathbb{F}_l)$, since the action of $G_E$ on $\mathbb{F}_l$ has been declared to be trivial. In this manner the natural map coming from restriction $\Gamma_{\mathbb{F}_l}(E) \rightarrow \Gamma_{\mathbb{F}_l}(F)^{\text{Gal}(F/E)}$ becomes the natural restriction homomorphism
$$
\text{Res}:H^1(G_E, \mathbb{F}_l) \to H^1(G_F, \mathbb{F}_l)^{\text{Gal}(F/E)}.
$$
Therefore the generalized restriction-inflation long exact sequence gives us a connecting homomorphism $r_2$, which provides a canonical isomorphism
$$
r_2:\frac{H^1(G_F, \mathbb{F}_l)^{\text{Gal}(F/E)}}{H^1(G_E, \mathbb{F}_l)} \simeq \text{ker}\left(\text{Inf} : H^2(\text{Gal}(F/E), \mathbb{F}_l) \to H^2(G_E, \mathbb{F}_l)\right).
$$
The map $\text{Inf}$ maps a $2$-cocycle for $\text{Gal}(F/E)$ to a $2$-cocycle for $G_E$ simply by precomposing the $2$-cocycle with the projection of $G_E$ onto $\text{Gal}(F/E)$. So the kernel consists of the $2$-cocycles $\theta \in H^2(\text{Gal}(F/E), \mathbb{F}_l)$ for which there exists a continuous $1$-cochain $\phi:G_E \to \mathbb{F}_l$ with $d(\phi)=\theta$. 

We stress that this does not imply that $\theta$ is trivial as an element of $H^2(\text{Gal}(F/E), \mathbb{F}_l)$, since the field of definition of $\phi$ need not be a subfield of $F$. We claim that $r_1$ and $r_2$ are actually the same map. Indeed, by the general formula for the connecting homomorphism for the restriction-inflation exact sequence, one finds that the map $r_2$ can be written as follows. For each element $\sigma$ of $\text{Gal}(F/E)$ fix a lift $\tilde{\sigma} \in G_E$ and take $\chi \in \Gamma_{\mathbb{F}_l}(F)^{\text{Gal}(F/E)}$. Then $r_2(\chi)$ is represented in $H^2(\text{Gal}(F/E),\mathbb{F}_l)$ by the cocycle
$$
(\sigma_1,\sigma_2) \mapsto \chi\left(\widetilde{\sigma_1 \sigma_2} \widetilde{\sigma_1}^{-1} \widetilde{\sigma_2}^{-1}\right).
$$  
It is a pleasant exercise to show directly, using that $\chi \in \Gamma_{\mathbb{F}_l}(F)^{\text{Gal}(F/E)}$, that the choice of the lift only changes the expression by a coboundary. On the other hand one readily sees that this is the same class as $r_1(\chi)$. Therefore we conclude the following fundamental fact.

\begin{prop} 
\label{abstract criterion for realizing central extensions}
We have
$$
\widetilde{\emph{Cent}}_{\mathbb{F}_l}(F/E) = \emph{ker}\left(H^2(\emph{Gal}(F/E), \mathbb{F}_l) \to H^2(G_E, \mathbb{F}_l)\right).
$$
In other words, a class $\theta \in H^2(\emph{Gal}(F/E), \mathbb{F}_l)$ is equal to $r_1(\chi)$ for some $\chi \in \Gamma_{\mathbb{F}_l}(F)^{\emph{Gal}(F/E)}$ if and only if $\theta$ is trivial when viewed as a class in $ H^2(G_E, \mathbb{F}_l)$. In this case, the set ${r_1}^{-1}(\theta)$ consists precisely of a single coset for the group $\Gamma_{\mathbb{F}_l}(E)$.
\end{prop}

Loosely speaking the above criterion tells us that the group theoretic $\mathbb{F}_l$-extensions of $\text{Gal}(F/E)$ that can be realized with a field extension are precisely equal to the classes of $H^2(\text{Gal}(F/E), \mathbb{F}_l)$ that become trivial in $H^2(G_E, \mathbb{F}_l)$. This criterion takes an even simpler form if we assume that $E$ contains an element $\zeta_l \in E^{\text{sep}}$ of multiplicative order exactly $l$. 

Recall from Section \ref{conventions} that we identified $\mathbb{F}_l$ and $\langle\zeta_l \rangle$ with the isomorphism $j_l(a)=\zeta_l^a$ for each $a \in \mathbb{F}_l$. Observe that if $\zeta_l \in E$, then $j_l$ is an identification of $G_E$-modules. In particular $j_l$ induces an isomorphism
$$
j_l : H^2(G_E, \mathbb{F}_l) \to \text{Br}_{E}[l].
$$
Also the group $\frac{\Gamma_{\mathbb{F}_l}(F)^{\text{Gal}(F/E)}}{\Gamma_{\mathbb{F}_l}(E)}$ can be identified via $j_l$ with 
$$
\frac{\left(\frac{F^\ast}{F^{\ast l}}\right)^{\text{Gal}(F/E)}}{E^\ast},
$$
and one can quickly re-obtain in this special case a proof of Proposition \ref{abstract criterion for realizing central extensions} by using Kummer sequences, which provide a natural identification
$$
\frac{\left(\frac{F^\ast}{F^{\ast l}}\right)^{\text{Gal}(F/E)}}{E^\ast} = \text{ker}(H^2(\text{Gal}(F/E), \mathbb{F}_l) \to \text{Br}_{E}[l]).
$$
It turns out that as soon as $\text{char}(E) \neq l$ we can always verify the realizability of a cohomology class $\theta$ in terms of the vanishing of a class attached to $\theta$ in a Brauer group.

\begin{prop} 
\label{criterion with Brauer groups}
Suppose $F/E$ is a finite Galois extension with $\emph{char}(E) \neq l$. Then the group $\widetilde{\emph{Cent}}_{\mathbb{F}_l}(F/E)$ consists of those classes $\theta \in H^2(\emph{Gal}(F/E), \mathbb{F}_l)$ such that $j_l \circ \emph{Res}_{G_{E(\zeta_l)}}(\theta)$ is trivial in $\emph{Br}_{E(\zeta_l)}$. 
\end{prop}

\begin{proof}
By Proposition \ref{abstract criterion for realizing central extensions} we know that $\theta$ is realizable if and only if $\theta$ becomes trivial in $H^2(G_E, \mathbb{F}_l)$. There is a map $\text{Co-Res} : H^2(G_{E(\zeta_l)}, \mathbb{F}_l) \to H^2(G_E, \mathbb{F}_l)$ with the property 
\[
\text{Co-Res} \circ \text{Res}(\theta)=[E(\zeta_l):E]\cdot \theta.
\]
But $l \cdot \theta = 0$ and $[E(\zeta_l) : E]$ divides $\left|\text{Aut}_{\text{gr}}(\langle \zeta_l \rangle)\right| = l - 1$. Hence $\theta$ becomes trivial in $H^2(G_E, \mathbb{F}_l)$ if and only if $\text{Res}_{G_{E(\zeta_l)}}(\theta)$ is trivial in $H^2(G_{E(\zeta_l)}, \mathbb{F}_l)$. Finally, since $j_l$ is an isomorphism, the triviality of $\theta$ in $H^2(G_{E(\zeta_l)}, \mathbb{F}_l)$ is equivalent to the triviality of $j_l \circ \theta$ in $\text{Br}_{E(\zeta_l)}$.
\end{proof}

In particular for $E = \mathbb{Q}$ we derive the following criterion. The group $\mathbb{F}_l$ will be implicitly considered as a trivial $G$-module whenever a symbol suggests an action of a group $G$ on $\mathbb{F}_l$.

\begin{prop} 
\label{criterion for Q}
Let $F/\mathbb{Q}$ be a finite Galois extension. Then the following are equivalent for $\theta \in H^2(\emph{Gal}(F/\mathbb{Q}), \mathbb{F}_l)$
\begin{enumerate}
\item[(1)] we have that $\theta \in \widetilde{\emph{Cent}}_{\mathbb{F}_l}(F/\mathbb{Q})$, i.e. $\theta = r_1(\chi)$ for some $\chi \in \Gamma_{\mathbb{F}_l}(F)^{\emph{Gal}(F/\mathbb{Q})}$;
\item[(2)] there exists a continuous $1$-cochain 
$$
\phi : G_{\mathbb{Q}} \to \mathbb{F}_l,
$$ 
such that $d(\phi) = \theta$;
\item[(3)] for every $v \in \Omega_{\mathbb{Q}(\zeta_l)}$ we have that $\emph{inv}_v(j_l \circ \theta) = 0$. 
\end{enumerate}

\begin{proof}
This is an immediate consequence of Proposition \ref{abstract criterion for realizing central extensions}, Proposition \ref{criterion with Brauer groups} and the fact that $\text{Br}_{\mathbb{Q}(\zeta_l)}$ embeds in the direct sum $\bigoplus_{v \in \Omega_{\mathbb{Q}(\zeta_l)}}\text{Br}_{\mathbb{Q}(\zeta_l)_v}$. 
\end{proof}
\end{prop}

The following fact will help us to cut down in practice the set of places one needs to check in part $(3)$ of Proposition \ref{criterion for Q}. The proposition could easily be derived from local class field theory. Instead we give an elementary proof here based on Proposition \ref{criterion with Brauer groups}. 

\begin{prop}
\label{unramified extensions trivial inv}
Let $E$ be a local field and $F/E$ a finite unramified extension. Then $\widetilde{\emph{Cent}}_{\mathbb{F}_l}(F/E)=H^2(\emph{Gal}(F/E), \mathbb{F}_l)$, i.e. for each $\theta \in H^2(\emph{Gal}(F/E), \mathbb{F}_l)$, we have that $\theta$ is trivial in $H^2(G_E, \mathbb{F}_l)$. In particular, if $\mu_l(E) \neq \{1\}$, we have $\emph{inv}_{E}(j_l \circ \theta) = 0$ for each $\theta \in H^2(\emph{Gal}(F/E), \mathbb{F}_l)$.
\end{prop}

\begin{proof}
The group $\text{Gal}(F/E)$ is cyclic, hence the group $H^2(\text{Gal}(F/E), \mathbb{F}_l)$ equals the group $\text{Ext}(\text{Gal}(F/E), \mathbb{F}_l)$. Indeed, if $G$ is a group with a central subgroup $H$ and cyclic quotient $G/H$, then $G$ is abelian. We can assume that $l$ divides $[F:E]$ otherwise the $H^2$ collapses and the statement becomes a triviality. In this case $\text{Ext}(\text{Gal}(F/E), \mathbb{F}_l)$ is cyclic of order $l$. Therefore it is enough to provide one non-trivial element of $\text{Cent}_{\mathbb{F}_l}(F/E)$ for which we take the unramified degree $l$ extension of $F$. The other statements now follow from Proposition \ref{abstract criterion for realizing central extensions} and Proposition \ref{criterion with Brauer groups}.
\end{proof}

Thus, in practice, when we use Proposition \ref{criterion for Q}, it is enough to check at the places $v \in \Omega_{\mathbb{Q}(\zeta_l)}$ that ramify in $F(\zeta_l)/\mathbb{Q}(\zeta_l)$. The following general fact explains why we need only deal with elements of $\Omega_{\mathbb{Q}(\zeta_l)}$ with residue field degree $1$. 

\begin{prop} 
\label{only p and 1 mod p}
Let $F/\mathbb{Q}$ a finite Galois extension of degree a power of $l$. Then if a prime $q$ divides $\Delta_{F/\mathbb{Q}}$ then $q = l$ or $q \equiv 1 \bmod l$. 
\end{prop}

\begin{proof}[First proof] Let $q$ be a prime different from $l$ that ramifies in $F/\mathbb{Q}$ and choose some inertia subgroup $I_q \leq \text{Gal}(F/\mathbb{Q})$ at $q$. Observe that, since $q \neq l$ it must be that $\gcd(\left|I_q\right|, q) = 1$. Hence $\left|I_q\right|$ divides $q^{f_q(F/\mathbb{Q})} - 1$. On the other hand $f_q(F/\mathbb{Q})$ divides $[F : \mathbb{Q}]$ which is a power of $l$. We conclude that $q^{f_q(F/\mathbb{Q})} - 1 \equiv q - 1 \bmod l$. Because $\left|I_q\right|$ is a non-trivial power of $l$, this implies that $q$ is $1$ modulo $l$. 
\end{proof}

\begin{proof}[Second proof] Filter $F/\mathbb{Q}$ with a sequence $F=F_n \supseteq F_{n-1} \supseteq\ldots\supseteq F_0=\mathbb{Q}$ such that each $F_i/\mathbb{Q}$ has degree $l^i$ and is Galois. Take an $i$ such that there exists a prime $\mathfrak{q}$ of $F_i$ above $q$ that ramifies in $F_{i + 1}$; such an $i$ exists otherwise the extension $F/\mathbb{Q}$ would be unramified at $q$. 

Let $\tilde{\mathfrak{q}}$ be the unique prime above $\mathfrak{q}$ in $F_{i+1}$. The local field extension ${F_{i+1}}_{\tilde{\mathfrak{q}}}/{F_i}_{\mathfrak{q}}$ is totally ramified cyclic of degree $l$. The local Artin reciprocity law implies that the group $O_{{F_i}_{\mathfrak{q}}}^{\ast }$ has a cyclic quotient of size $l$. If $q \neq l$, the group $U_1({F_i}_{\mathfrak{q}})$ is a pro-$q$ group with trivial image in any such quotient. Hence 
the group $\left(O_{{F_i}_{\mathfrak{q}}}/\mathfrak{q}\right)^\ast$ has a cyclic quotient of size $l$. On the other hand
$$
\left|\left(\frac{O_{{F_i}_{\mathfrak{q}}}}{\mathfrak{q}}\right)^\ast\right| = \left|\mathbb{F}_{q^{l^j}}^\ast\right|
$$
for some $j \leq i$. But we have
\[
\left|\mathbb{F}_{q^{l^j}}^\ast\right| = q^{l^j} - 1 \equiv q - 1 \bmod l,
\]
thereby concluding our proof.
\end{proof}

Recall that if we have two conjugate subgroups $H_1$, $H_2$ of $G_{\mathbb{Q}}$ and a class $\theta \in H^2(G_{\mathbb{Q}},\mathbb{F}_l)$ then $\text{Res}_{H_1}(\theta)=0$ if and only if $\text{Res}_{H_2}(\theta)=0$. Thus the following definition makes sense.

\begin{mydef} 
\label{locally trivial classes}
Let $v \in \Omega_{\mathbb{Q}}$ be a place and $\theta$ be a class in $H^2(G_{\mathbb{Q}},\mathbb{F}_l)$. We say that $\theta$ is \emph{locally trivial} at $v$ in case $\textup{Res}_{i^{\ast }(G_{\mathbb{Q}_v})}(\theta)=0$ for a (equivalently any) choice of an embedding $i : \overline{\mathbb{Q}} \to \overline{\mathbb{Q}_v}$.
\end{mydef}

\begin{prop}
\label{local-global stuff}
Let $F/\mathbb{Q}$ be a finite Galois extension. 
Let $\theta$ be a class in $H^2(\emph{Gal}(F/\mathbb{Q}), \mathbb{F}_l)$. Then the following are equivalent
\begin{enumerate}
\item[(1)] we have that $\theta \in \widetilde{\emph{Cent}}_{\mathbb{F}_l}(F/\mathbb{Q})$;
\item[(2)] the inflation of $\theta$ to $G_{\mathbb{Q}}$ is locally trivial at all places $v \in \Omega_{\mathbb{Q}}$;
\item[(3)] the inflation of $\theta$ to $G_{\mathbb{Q}}$ is locally trivial at all places $v \in \Omega_{\mathbb{Q}}$ that ramify in $F/\mathbb{Q}$;
\item[(4)] for any $v \in \Omega_{\mathbb{Q}}$ and $\tilde{v} \in \Omega_F$ lying above $v$, we have that $\theta \in \widetilde{\emph{Cent}}_{\mathbb{F}_l}(F_{\tilde{v}}/\mathbb{Q}_v)$.
\end{enumerate}
\end{prop}

\begin{proof}
Part (2) and part (4) are equivalent in view of Proposition \ref{abstract criterion for realizing central extensions}. Thanks to the same proposition, we have that part (1) certainly implies part (2). Furthermore, part (2) trivially implies part (3). On the other hand, thanks to Proposition \ref{unramified extensions trivial inv}, we see that part (3) implies part (2) as well. It remains to show that part (2) implies part (1). But part (2) of this proposition implies that part (3) of Proposition \ref{criterion for Q} holds. Now use Proposition \ref{criterion for Q}.
\end{proof}

The equivalence between $(1)$ and $(4)$ in Proposition \ref{local-global stuff} tells us that a class $\theta$ is realizable if and only if it is realizable locally everywhere, and moreover it is sufficient to check that it is realizable locally at the places of $\mathbb{Q}$ that are ramified in $F$.

In our applications, we will not merely be interested in writing the relevant $\theta$ as $r_1(\chi)$ for some $\chi \in \Gamma_{\mathbb{F}_l}(F)^{\text{Gal}(F/\mathbb{Q})}$, but it will also be important for us to find a representative $\chi$ in the $\Gamma_{\mathbb{F}_l}(\mathbb{Q})$-coset $r_1^{-1}(\theta)$ such that $\overline{\mathbb{Q}}^{\text{ker}(\chi)}/\mathbb{Q}$ has as little ramification as possible.  

\begin{prop} 
\label{little ramification}
Let $F/\mathbb{Q}$ a finite Galois extension and let $\theta \in \widetilde{\emph{Cent}}_{\mathbb{F}_l}(F/\mathbb{Q})$. Then there exists $\chi \in \Gamma_{\mathbb{F}_l}(F)^{\emph{Gal}(F/\mathbb{Q})}$ such that $r_1(\chi) = \theta$ and $\overline{\mathbb{Q}}^{\emph{ker}(\chi)}/\mathbb{Q}$ is unramified at all primes $q$ not dividing $\Delta_{F/\mathbb{Q}}$.  
\end{prop}

\begin{proof}
Take any $\chi \in \Gamma_{\mathbb{F}_l}(F)^{\text{Gal}(F/\mathbb{Q})}$ with $r_1(\chi) = \theta$. If $\overline{\mathbb{Q}}^{\text{ker}(\chi)}/\mathbb{Q}$ is unramified at all primes $q$ not dividing $\Delta_{F/\mathbb{Q}}$, we are done. So suppose that there is a rational prime $q$ that does not ramify in $F$ but does ramify in $\overline{\mathbb{Q}}^{\text{ker}(\chi)}$. Then there is a prime $\mathfrak{q}$ of $O_F$ above $q$ that ramifies in $\overline{\mathbb{Q}}^{\text{ker}(\chi)}$. 

As observed in the proof of Proposition \ref{unramified extensions trivial inv}, the group $H^2(\text{Gal}(F_{\mathfrak{q}}/\mathbb{Q}_q), \mathbb{F}_l)$ is cyclic of order $l$, generated by an unramified character of $G_{F_{\mathfrak{q}}}$ of order $l$. This means that we can always find $\chi' \in G_{\mathbb{Q}_q}$ of order $l$ such that $\chi + \chi'$ is an unramified character for $G_{F_\mathfrak{q}}$. Moreover, we can take $\chi'$ to be a multiple of $\chi_q$, see Subsection \ref{conventions} for the notation. So we can find $\chi' \in \Gamma_{\mathbb{F}_l}(\mathbb{Q})$ such that $\mathfrak{q}$ does not ramify in $\overline{\mathbb{Q}}^{\text{ker}(\chi+\chi')}$. 

We claim that this implies that $\chi + \chi'$ does not ramify at \emph{any} prime above $q$. Indeed we certainly have that for each $\sigma \in G_{\mathbb{Q}}$, the character $\sigma \cdot (\chi+\chi')$ is unramified at $\sigma(\mathfrak{q})$. On the other hand $\sigma \cdot (\chi+\chi')=\chi+\chi'$, since by assumption $\chi \in \Gamma_{\mathbb{F}_l}(F)^{\text{Gal}(F/\mathbb{Q})}$. This proves our claim. Finally observe that $\overline{\mathbb{Q}}^{\text{ker}(\chi+\chi')}/F$ does not ramify at any new prime, since $\chi_q$ ramifies only at $q$. Hence continuing in this manner we get rid of all such $q$ and we have proved the proposition.
\end{proof}

A stronger control on the ramification can be achieved at the cost of having a stronger notion of local triviality, which will be given in the next definition. Recall again that if $H_1,H_2$ are conjugate subgroups of a finite group $G$ and if $\theta \in H^2(G,\mathbb{F}_l)$, then $\text{Res}_{H_1}(\theta)=0$ if and only if $\text{Res}_{H_2}(\theta)=0$. This shows that the following definition makes sense. 

\begin{mydef} 
\label{locally split classes}
Let $F/\mathbb{Q}$ be a finite Galois extension, $\theta \in H^2(\emph{Gal}(F/\mathbb{Q}), \mathbb{F}_l)$ and $q$ a prime number. We say that $\theta$ is \emph{locally split at} $q$ if the restriction of $\theta$ to one (equivalently any) subgroup $D_{\mathfrak{q}/q}$ is trivial, where $\mathfrak{q}$ is a prime above $q$ in $F$ and  $D_{\mathfrak{q}/q}$ is the corresponding decomposition group. Moreover we say that $\theta$ is \emph{locally split at inertia at} $q$ in case the restriction of $\theta$ to one (equivalently any) subgroup $I_{\mathfrak{q}/q}$ is trivial, where $I_{\mathfrak{q}/q}$ denotes the inertia subgroup relative to $\mathfrak{q}$. 
\end{mydef}

\begin{prop}
\label{no ramification}
$\emph{(1)}$ Let $F/\mathbb{Q}$ be a finite Galois extension. If $\theta \in H^2(\emph{Gal}(F/\mathbb{Q}), \mathbb{F}_l)$ is locally split at all primes dividing $\Delta_{F/\mathbb{Q}}$, then we have $\theta \in \widetilde{\emph{Cent}}_{\mathbb{F}_l}(F/\mathbb{Q})$. Moreover, there exists $\chi \in \Gamma_{\mathbb{F}_l}(F)^{\emph{Gal}(F/\mathbb{Q})}$ such that $r_1(\chi)=\theta$ and $\overline{\mathbb{Q}}^{\emph{ker}(\chi)}/F$ is unramified.

\noindent $\emph{(2)}$ Suppose that all the primes dividing $\Delta_{F/\mathbb{Q}}$ are $1$ modulo $l$. Then the same conclusion as in $\emph{(a)}$ can be reached assuming only that $\theta$ is locally trivial and locally split at inertia at all primes dividing $\Delta_{F/\mathbb{Q}}$. 

\noindent $\emph{(3)}$ Suppose that $[F:\mathbb{Q}]$ is of degree a power of $l$ and does not ramify at $l$. Then the same conclusion as in $\emph{(a)}$ can be reached assuming only that $\theta$ is locally trivial and locally split at inertia at all primes dividing $\Delta_{F/\mathbb{Q}}$.
\end{prop}
 
\begin{proof}[Proof of (1)] Observe that if $\theta$ is locally split at a prime $q$, then it is certainly also locally trivial at $q$. Hence by Proposition \ref{local-global stuff} we deduce that $\theta \in \widetilde{\text{Cent}}_{\mathbb{F}_l}(F/\mathbb{Q})$. Due to Proposition \ref{little ramification} there is $\chi \in \Gamma_{\mathbb{F}_l}(F)^{\text{Gal}(F/\mathbb{Q})}$ with $\theta = r_1(\chi)$ such that $\overline{\mathbb{Q}}^{\text{ker}(\chi)}/F$ is unramified at all primes $\mathfrak{q}$ in $O_F$ that lie above primes $q$ in $\mathbb{Z}$ not dividing $\Delta_{F/\mathbb{Q}}$. 

Now take a prime $q$ dividing $\Delta_{F/\mathbb{Q}}$ and let $\mathfrak{q}$ in $O_F$ be a prime above $q$ that ramifies in $\overline{\mathbb{Q}}^{\text{ker}(\chi)}$. By assumption $\theta$ is locally split at each such prime $q$. Hence locally at each such $q$ the character $\chi$ is a character from $G_{\mathbb{Q}_q}$. Following the logic of the proof of Proposition \ref{little ramification}, we may employ multiples of $\chi_q$ to get rid of this additional ramification whenever that is required.
\end{proof}

\begin{proof}[Proof of (2)] The assumption that $\theta$ is locally trivial at all primes dividing $\Delta_{F/\mathbb{Q}}$ guarantees that $\theta \in \widetilde{\text{Cent}}_{\mathbb{F}_l}(F/\mathbb{Q})$ due to Proposition \ref{local-global stuff}. Again, by Proposition \ref{little ramification}, write $\theta = r_1(\chi)$ with $\overline{\mathbb{Q}}^{\text{ker}(\chi)}/F$ unramified at all primes of $O_F$ above a rational prime not dividing the discriminant. 

Let $q$ be a prime divisor of $\Delta_{F/\mathbb{Q}}$ and let $\mathfrak{q}$ be a prime above it in $O_F$. Let $F_{\mathfrak{q}}^{\text{unr}}$ be the largest unramified extension inside $F_{\mathfrak{q}}/\mathbb{Q}_q$. This is precisely the field fixed by the inertia subgroup $I_{\mathfrak{q}/q}$. The assumption that $\theta$ is locally split at inertia at $q$ guarantees precisely that $\chi$ restricted to any copy of $G_{F_{\mathfrak{q}}}$ in $G_F$ equals the restriction of a character coming from $G_{F_{\mathfrak{q}}^{\text{unr}}}$. Since $q$ is $1$ modulo $l$, any such character equals the product of a multiple of $\chi_q$ and an unramified character. Hence we can use the same logic of part $\text{(1)}$.
\end{proof}

\begin{proof}[Proof of (3)] This follows from part (2) and Proposition \ref{only p and 1 mod p}.
\end{proof}

\subsection{The Heisenberg group}
Let $l$ be an odd prime. Recall that if $G$ is a non-abelian group of order $l^3$, then the center of $G$ must be equal to its commutator subgroup and has order $l$. Therefore such non-abelian groups of order $l^3$ are completely described by $H^2(\mathbb{F}_l \oplus \mathbb{F}_l,\mathbb{F}_l)$. Recall that taking commutators gives an exact sequence
$$
0 \to \text{Ext}(\mathbb{F}_l \oplus \mathbb{F}_l, \mathbb{F}_l) \to H^2(\mathbb{F}_l \oplus \mathbb{F}_l, \mathbb{F}_l) \to \text{Hom}_{\mathbb{F}_l}(\wedge^2(\mathbb{F}_l \oplus \mathbb{F}_l), \mathbb{F}_l) \to 0,
$$
where $\wedge^2$ is the second exterior power. Therefore the dimension of $H^2(\mathbb{F}_l \oplus \mathbb{F}_l, \mathbb{F}_l)$ over $\mathbb{F}_l$ is $3$, and the $2$-dimensional subspace $\text{Ext}(\mathbb{F}_l \oplus \mathbb{F}_l, \mathbb{F}_l)$ describes abelian groups of order $l^3$ and exponent at most $l^2$. Observe that given any $h \in  H^2(\mathbb{F}_l \oplus \mathbb{F}_l, \mathbb{F}_l)$, we can attach a set theoretic map $\Phi_l(h):\mathbb{F}_l \oplus \mathbb{F}_l \to \mathbb{F}_l$ obtained by lifting and $l$-powering in the central extension. The result is independent on the lift, since the central subgroup is of exponent $l$. We remark that the assignment $h \mapsto \Phi_l(h)$ is linear in $h$. 

Moreover, a simple calculation, using that $\binom{l}{2}$ is divisible by $l$ if $l$ is odd, shows that for odd $l$ the map $\Phi_l(h)$ is a group homomorphism from $\mathbb{F}_l \oplus \mathbb{F}_l$ to $\mathbb{F}_l$. In other words, one has a group homomorphism for all odd $l$
$$
\Phi_l: H^2(\mathbb{F}_l \oplus \mathbb{F}_l, \mathbb{F}_l) \to \text{Hom}_{\mathbb{F}_l}(\mathbb{F}_l \oplus \mathbb{F}_l, \mathbb{F}_l).
$$
When restricted to $\text{Ext}(\mathbb{F}_l \oplus \mathbb{F}_l,\mathbb{F}_l)$ the map $\Phi_l$ gives the classical connecting homomorphism obtained by taking $l$-torsion of an exact sequence. Hence in this case $\Phi_l$ restricts to an isomorphism
$$
\Phi_l: \text{Ext}(\mathbb{F}_l \oplus \mathbb{F}_l, \mathbb{F}_l) \simeq  \text{Hom}_{\mathbb{F}_l}(\mathbb{F}_l \oplus \mathbb{F}_l, \mathbb{F}_l).
$$
Hence we can naturally split the sequence of $\text{GL}_2(\mathbb{F}_l) \times \text{GL}_1(\mathbb{F}_l)$-modules
$$
0 \to \text{Ext}(\mathbb{F}_l \oplus \mathbb{F}_l, \mathbb{F}_l) \to H^2(\mathbb{F}_l \oplus \mathbb{F}_l, \mathbb{F}_l) \to \text{Hom}_{\mathbb{F}_l}(\wedge^{2}(\mathbb{F}_l \oplus \mathbb{F}_l), \mathbb{F}_l) \to 0
$$
with
$$
H^2(\mathbb{F}_l \oplus \mathbb{F}_l, \mathbb{F}_l)=\text{Ext}(\mathbb{F}_l \oplus \mathbb{F}_l,\mathbb{F}_l) \oplus \text{ker}(\Phi_l).
$$
From the way it has been defined, it is clear that the group $\text{ker}(\Phi_l)$ is stable under the action of $\text{GL}_{2}(\mathbb{F}_l) \times \text{GL}_1(\mathbb{F}_l)$, so the above decomposition is a decomposition of $H^2(\mathbb{F}_l \oplus \mathbb{F}_l, \mathbb{F}_l)$ as a $\text{GL}_{2}(\mathbb{F}_l) \times \text{GL}_1(\mathbb{F}_l)$-module. Moreover, by definition $\text{ker}(\Phi_l)$ consists precisely of the classes giving groups of exponent $l$. Observe again that it is crucial that $l$ is odd, since for $l = 2$ we would get in this manner only the trivial class. 

Let $\chi_1,\chi_2$ be the projections with respect to the first and second coordinate, in the standard basis, of $\mathbb{F}_l \oplus \mathbb{F}_l$. Identify $\mathbb{F}_l \otimes \mathbb{F}_l$ with $\mathbb{F}_l$, via the map $a \otimes b \mapsto a\cdot b$ where the product is with respect to the field structure of $\mathbb{F}_l$. In this way $\chi_1 \cup \chi_2 \in H^2(\mathbb{F}_l \oplus \mathbb{F}_l, \mathbb{F}_l \otimes \mathbb{F}_l)$ gives the $2$-cocycle 
$$
\{(v, w) \mapsto \chi_1(v) \cdot \chi_2(w)\}_{v, w \in \mathbb{F}_l \oplus \mathbb{F}_l} \in H^2(\mathbb{F}_l \oplus \mathbb{F}_l, \mathbb{F}_l),
$$
that we will still denote by $\chi_1 \cup \chi_2$ by abuse of language. It is an immediate verification that the groups structure induced in this way coincide with the matrix group law on $N_3(\mathbb{F}_l)$, which is by definition the $l$-Sylow of $\text{GL}_3(\mathbb{F}_l)$. This group is also known as the \emph{Heisenberg group}, which is a group of exponent $l$. Using once more that $l$ is odd, we see that a change of basis  of $\mathbb{F}_l \oplus \mathbb{F}_l$ multiplies the cohomology class of $\chi_1 \cup \chi_2$ by the determinant of the transformation. In other words we have
$$
\text{ker}(\Phi_l)=\mathbb{F}_l \cdot (\chi_1 \cup \chi_2).
$$
Furthermore, the non-trivial multiples of $\chi_1 \cup \chi_2$ consists of a single orbit under $\text{GL}_2(\mathbb{F}_l)$. In particular all non-abelian groups of size $l^3$ and exponent $l$ are isomorphic, a fact that can also be established directly by an argument with generators and relations. Another way to write down such a group is 
$$
\frac{\mathbb{Z}_l[\zeta_l]}{(1 - \zeta_l)^2} \rtimes \langle \zeta_l \rangle.
$$
A similar description can be done for non-abelian groups of size $l^3$ and exponent $l^2$. Namely one considers the group
$$\frac{\mathbb{Z}_l}{l^2} \rtimes \langle 1 + l\rangle.
$$
Our decomposition of $H^2$ implies quite easily that this is the unique non-abelian group of order $l^3$ and exponent $l^2$. Indeed it is even true that all the classes in $H^2$ giving such a group of order $l^3$ are conjugate under the action of the product of the two automorphism groups. Summing up we have established the following. 

\begin{prop} 
\label{Heisenberg group is a line in the H^2 given by the cup}
Let $l$ be an odd prime. Then the $l$-map 
$$
\Phi_l : H^2(\mathbb{F}_l \oplus \mathbb{F}_l, \mathbb{F}_l) \to \emph{Hom}_{\mathbb{F}_l}(\mathbb{F}_l \oplus \mathbb{F}_l, \mathbb{F}_l)
$$ 
induces a splitting of $\emph{GL}_2(\mathbb{F}_l) \times \emph{GL}_1(\mathbb{F}_l)$-modules
$$
H^2(\mathbb{F}_l \oplus \mathbb{F}_l, \mathbb{F}_l)=\emph{Ext}(\mathbb{F}_l \oplus \mathbb{F}_l, \mathbb{F}_l) \oplus \emph{ker}(\Phi_l).
$$
The group $\emph{ker}(\Phi_l)$ gives precisely the set of groups of order $l^3$ of exponent $l$. Moreover,
$$
\emph{ker}(\Phi_l) = \mathbb{F}_l \cdot (\chi_1 \cup \chi_2)
$$
for any linearly independent characters $\chi_1,\chi_2$ in $\emph{Hom}_{\mathbb{F}_l}(\mathbb{F}_l \oplus \mathbb{F}_l, \mathbb{F}_l)$. The group $\emph{GL}_2(\mathbb{F}_l)$ acts by determinant on $\emph{ker}(\Phi_l)$, whose non-zero elements form a single $\emph{GL}_2(\mathbb{F}_l)$-orbit, all giving the group $\frac{\mathbb{Z}_l[\zeta_l]}{(1 - \zeta_l)^2} \rtimes \langle \zeta_l \rangle$.

All the extensions outside $\emph{Ext}(\mathbb{F}_l \oplus \mathbb{F}_l, \mathbb{F}_l) \cup \emph{ker}(\Phi_l)$ consists of a single orbit under the action of $\emph{GL}_2(\mathbb{F}_l) \times \emph{GL}_1(\mathbb{F}_l)$, and all give the group $\frac{\mathbb{Z}_l}{l^2} \rtimes \langle 1 + l \rangle$, which is the unique non-abelian group of order $l^3$ and exponent $l^2$. In each non-trivial coset of 
$$
\emph{Ext}(\mathbb{F}_l \oplus \mathbb{F}_l, \mathbb{F}_l)
$$ 
there are precisely $l - 1$ classes which give the unique non-abelian group of exponent $l^2$. The remaining class in the coset is a unique non-zero multiple of $\chi_1 \cup \chi_2$. 
\end{prop} 

In particular we deduce the following fact. Fix a field $K$ and a separable closure $K^{\text{sep}}$. Denote by $G_K$ the group of $K$-algebra automorphisms of $K^{\text{sep}}$. For a continuous character $\chi:G_K \to \mathbb{F}_l$, let $K(\chi)$ be the corresponding field extension of $K$. 

\begin{corollary} 
\label{Only need to check if it is Heisenberg}
Let $\chi_1, \chi_2$ be two independent continuous character from $G_K$ to $\mathbb{F}_l$. Then $\chi_1 \cup \chi_2$ is in $\widetilde{\emph{Cent}}_{\mathbb{F}_l}(K(\chi_1)K(\chi_2)/K)$ if and only if there exists a Galois extension $L/K$ containing $K(\chi_1)K(\chi_2)$ such that 
$$
\emph{Gal}(L/K) \simeq_{\emph{gr.}} \frac{\mathbb{Z}_l[\zeta_l]}{(1 - \zeta_l)^2} \rtimes \langle \zeta_l \rangle.
$$
\end{corollary}

\begin{remark}
Observe the contrast with the case $l = 2$. In that case it is far from true that putting a biquadratic extension inside a $D_8$-extension is equivalent to the realizability of all the cups of the characters of the biquadratics. It is in that case only one of the cups that is realizable. Correspondingly the various subgroups of index $2$ of $D_8$ are not all conjugate under the automorphism group of $D_8$, while for odd $l$ they are. Also the various cups in that case do not form a single $\mathbb{F}_2$-line as in Proposition \ref{Heisenberg group is a line in the H^2 given by the cup}. In that case they even generate the full $H^2$! 
\end{remark}

We end this section by mentioning the following fact concerning extensions having Galois group $N_3(\mathbb{F}_l)$ that can be used to prove Corollary \ref{if and only if it is a norm} for fields of characteristic different from $l$. One way to prove this fact is to use the material in this section. The interested reader can also look at the reference \cite[Theorem 3.1]{Mi}. 

\begin{prop} 
\label{Heisenberg group and norms}
Suppose $K$ has a primitive $l$-th root of unity. If $b \in K^\ast$, we define $\chi_b : G_K \rightarrow \mathbb{F}_l$ to be the unique character such that for each $\beta \in K^{\sep}$ with $\beta^l = b$ we have
$$
\sigma(\beta) = \left(j_l \circ \chi_b(\sigma)\right) \beta.
$$
Let $b_1, b_2 \in K^\ast$ be two elements that are independent in $K^{\ast }/K^{\ast l}$. Then we have
$$
\chi_{b_1} \cup \chi_{b_2} \in \widetilde{\emph{Cent}}_{\mathbb{F}_l}(K(\chi_{b_1})K(\chi_{b_2})/K) \Longleftrightarrow \exists \omega \in K(\chi_{b_1})^\ast : b_2 = N_{K_{\chi_{b_1}}/K}(\omega).
$$ 
In that case the image of $\chi_{b_1} \cup \chi_{b_2}$ in $\emph{Cent}_{\mathbb{F}_l}(K(\chi_{b_1})K(\chi_{b_2})/K)$ is obtained by taking the extension $K(\chi_{b_1})K(\chi_{b_2})(\sqrt[l]{\alpha})$ with $\alpha := \prod_{i = 0}^{l - 2} \sigma^i(\omega^{l - 1 - i})$, where $\sigma$ is a generator of $\emph{Gal}(K(\chi_{b_1})/K)$. 
\end{prop}

\section{The first Artin pairing} 
\label{Redei matrices}
In this section we study the first Artin pairing. Let $\chi$ be in $\Gamma_{\mu_l}(\mathbb{Q})$ and $b \in \overline{\text{Cl}}(K_\chi)$. We extend the notation introduced in Section \ref{conventions} by defining $\text{Up}_{K_\chi}(b)$ to be the unique product of ramified prime ideals of $K_\chi$ whose norm is precisely $b$. We will sometimes attribute properties of $\text{Up}_{K_\chi}(b)$ to $b \in \overline{\text{Cl}}(K_\chi)$. For example, we shall often say $b$ is in $(1 - \zeta_l)^k \text{Cl}(K_\chi)$, for some positive integer $k$, which means that $\text{Up}_{K_\chi}(b)$ is in $(1 - \zeta_l)^k \text{Cl}(K_\chi)$. 

From the description of $\text{Cl}(K_\chi)[1 - \zeta_l]$ and $\text{Cl}(K_\chi)^{\vee}[1 - \zeta_l]$ combined with Proposition \ref{abstract Artin pairing}, we readily obtain the following description of $(1 - \zeta_l)\text{Cl}(K_\chi)[(1 - \zeta_l)^2]$ and of $(1 - \zeta_l)\text{Cl}(K_\chi)^{\vee}[(1 - \zeta_l)^2]$. 

\begin{prop} 
\label{First description of Redei matrices}
\emph{(1)} An element $b \in \overline{\textup{Cl}}(K_\chi)$ is in $(1 - \zeta_l)\emph{Cl}(K_\chi)[(1 - \zeta_l)^2]$ if and only if for every prime $q$ dividing $\Delta_{K_\chi/\mathbb{Q}}$, we have that the Artin symbol
$$
\left[\frac{K_{\chi_q}K_\chi/K_\chi}{\emph{Up}_{K_\chi}(b)}\right]
$$
is the identity. \\
\emph{(2)} A character $\chi' \in \emph{Cl}(K_\chi)^{\vee}[1 - \zeta_l]$ is in $(1 - \zeta_l)\emph{Cl}(K_\chi)^{\vee}[(1 - \zeta_l)^2]$ if and only if for every prime $q$ dividing $\Delta_{K_\chi/\mathbb{Q}}$, we have that 
$$
\chi'\left(\emph{Frob}_{\emph{Up}_{K_\chi}(q)}\right) = 1.
$$
\end{prop}

\begin{proof}
This follows immediately from the material in Subsection \ref{invariants}, Subsection \ref{dual invariants} and Proposition \ref{abstract Artin pairing}. 
\end{proof}

\noindent The following fact will be useful for us.

\begin{prop} 
\label{descending artin symbols}
Let $\chi$ be in $\Gamma_{\mu_l}(\mathbb{Q})$. Suppose that $q$ and $q'$ are two distinct primes dividing $\Delta_{K_\chi/\mathbb{Q}}$. Then
$$
\chi_q\left(\emph{Frob}_{\emph{Up}_{K_\chi}(q')}\right) = \chi_q\left(\emph{Frob}_{q'}\right).
$$
\end{prop}

\begin{proof}
Observe that $q'$ splits in $K_{\chi_{q}}$ if and only if $\text{Up}_{K_\chi}(q')$ splits in $K_\chi K_{\chi_q}$. Hence we can safely assume that they are both not split. In that case, since $q \neq q'$, they must be both inert. It follows from the defining property of Frobenius that restricting $\text{Frob}_{\text{Up}_{K_\chi}(q')}$ to $K_{\chi_q}/\mathbb{Q}$ gives $\text{Frob}_{q'}$. Hence we obtain the desired conclusion. 
\end{proof}

Recall that if $\chi$ is in $\Gamma_{\mu_l}(\mathbb{Q})$, then $\epsilon_\chi$ denotes the unique amalgama for $\Delta_{K_\chi /\mathbb{Q}}$ with the property
$$
\chi = \prod_{q \mid  \Delta_{K_\chi /\mathbb{Q}}} \chi_q^{\epsilon_{\chi}(q)}.
$$ 
We will now define an $\omega(\Delta_{K_\chi /\mathbb{Q}}) \times \omega(\Delta_{K_\chi /\mathbb{Q}})$ matrix with coefficients in $\mathbb{F}_l$. We index the matrix by the product set $\{q \text{ prime} : q \mid \Delta_{K_\chi/\mathbb{Q}}\} \times \{q \text{ prime} : q \mid \Delta_{K_\chi/\mathbb{Q}}\}$, where in the row $\{q\} \times \{q' \text{ prime} : q' \mid \Delta_{K_\chi/\mathbb{Q}}\}$ we put for each $q \neq q'$ the element
$$
j_l^{-1} \circ \chi_{q'}^{\epsilon_\chi(q')}\left(\text{Frob}_{q}\right),
$$
and we impose that the sum on each row is $0$. This determines uniquely the so-called Redei matrix that we denote as 
$$
\text{Redei}(K_\chi) \in \mathbb{F}_l ^{\{q \text{ prime } : \ q \mid \Delta_{K_\chi/\mathbb{Q}}\} \times \{q \text{ prime } : \ q \mid \Delta_{K_\chi/\mathbb{Q}}\}}.
$$
In what follows exponentiation by an element $v$ of $\mathbb{F}_l$ has to be read as the conventional powering with the only integer between $\{0, \ldots , l-1\}$ that is congruent to $v$ modulo $l$. Then we have the following important conclusion.

\begin{corollary} 
\label{More on Redei matrices}
\emph{(1)} The elements $b$ of $\overline{\textup{Cl}}(K_\chi)$ that are in $(1-  \zeta_l)\emph{Cl}(K_\chi)[(1 - \zeta_l)^2]$ are precisely those $b$ such that
$$
\emph{Up}_{K_\chi}(b) = \prod_{q \mid \Delta_{K_\chi/\mathbb{Q}}}\emph{Up}_{K_\chi}(q)^{v_q},
$$
for $(v_q)_{q \mid \Delta_{K_\chi /\mathbb{Q}}}$ an element of the left kernel of $\emph{Redei}(K_\chi)$.  \\
\emph{(2)} The elements $\chi'$ of $\emph{Cl}(K_\chi)^{\vee}[(1 - \zeta_l)]$ that are in $(1 - \zeta_l)\emph{Cl}(K_\chi)^{\vee}[(1 - \zeta_l)^2]$ are precisely those $\chi'$ such that 
$$
\chi' = \prod_{q \mid \Delta_{K_\chi/\mathbb{Q}}} \chi_q^{w_q\epsilon_\chi(q)},
$$
where $(w_q)_{q \mid \Delta_{K_\chi /\mathbb{Q}}}$ is in the right kernel of $\emph{Redei}(K_\chi)$. 
\end{corollary}

\begin{proof}
This follows upon combining Proposition \ref{First description of Redei matrices} and Proposition \ref{descending artin symbols}.
\end{proof}

In Subsection \ref{1-zeta-square are norms} and in Subsection \ref{1-zeta-square are realizable cups} we investigate more closely the structure of respectively the left and the right kernel of the Redei matrix. The resulting characterizations are contained in Corollary \ref{being a norm in Kchi} and Corollary \ref{characterizing 1-zetap square character via central extensions}. In these subsections we additionally provide alternative, and more direct, proofs of Corollary \ref{being a norm in Kchi} and Corollary \ref{characterizing 1-zetap square character via central extensions}. The material in Subsection \ref{1-zeta-square are realizable cups} relies on the material in Section \ref{central extensions} about central $\mathbb{F}_l$-extensions.

\subsection{\texorpdfstring{$(1 - \zeta_l)\text{Cl}(K_\chi)[(1 - \zeta_l)^2]$}{The left kernel}} 
\label{1-zeta-square are norms}
Let $\chi \in \Gamma_{\mu_l}(\mathbb{Q})$. We begin by rewriting $\text{Redei}(K_\chi)$ as a matrix of symbols coming from cyclic algebras. See the appendix \ref{general facts about cyclic algebras} for the notation and the basic facts used from the theory of such algebras over local and global fields. We use the convention that $A(i,j)$ denotes the element on the $i$-th row and $j$-th column of a matrix $A$.

\begin{prop}
\label{Redei as matrix of invs}
For all primes $q$, $q'$ dividing $\Delta_{K_\chi /\mathbb{Q}}$ we have
$$
\emph{Redei}(K_\chi)(q, q') = j_l^{-1} \circ h_l \circ \eta_{\mathbb{Q}_{q}}\left(\left\{\chi_{q'}^{\epsilon_{\chi}(q')}, q\right\}\right).
$$
\end{prop}

\begin{proof}
This follows immediately from Proposition \ref{inv and Artin symbols} combined with the definition of the Redei matrix $\text{Redei}(K_\chi)$ and the bilinearity of $(\chi, \theta) \mapsto \{\chi, \theta\}$.
\end{proof}

The following important corollary furnishes an interpretation of the left kernel that will be crucial in handling the higher pairings as we shall see in the later sections.

\begin{corollary} 
\label{being a norm in Kchi}
An element $b$ in $\overline{\textup{Cl}}(K_\chi)$ is in $(1 - \zeta_l)\emph{Cl}(K_\chi)[(1 - \zeta_l)^2]$ if and only if $b$ is a norm in $K_\chi$. 
\end{corollary}

\begin{proof}
Fix a prime divisor $q$ of $\Delta_{K_\chi/\mathbb{Q}}$. For now assume that $q$ does not divide $b$. Then the pairing of $\chi_q$ and $b$ is trivial if and only if
$$
\prod_{p \mid b} \eta_{\mathbb{Q}_p}\left(\left\{\chi_q, b\right\}\right) = 1.
$$
Indeed, this follows from Proposition \ref{Redei as matrix of invs} and the fact that the result in Proposition \ref{inv and Artin symbols} is independent of the choice of uniformizer; the latter observation also follows from a combination of Proposition \ref{unramified extensions trivial inv} and Proposition \ref{cup products and cyclic algebras}. From Proposition \ref{unramified extensions trivial inv} and Proposition \ref{cup products and cyclic algebras}, we see that the only other place in $\Omega_{\mathbb{Q}}$ where the cyclic algebra $\{\chi_q, b\}$ could possibly be non-trivial is $q$. Therefore by Proposition \ref{Hilbert reciprocity} (Hilbert reciprocity), we learn that 
$$
\eta_{\mathbb{Q}_q}\left(\left\{\chi_q, b\right\}\right) = 1.
$$
This implies that
$$
\eta_{\mathbb{Q}_q}\left(\left\{\chi, b\right\}\right) = 1
$$
for each $q$ that does not divide $b$ by Proposition \ref{unramified extensions trivial inv}. Next assume that $q$ divides $b$. Denote
$$
b' := \frac{b}{q^{v_{\mathbb{Q}_q}(b)}}. 
$$
From the definition of the pairing we see that $\chi_q$ and $b$ have trivial pairing if and only if $\chi_q^{\epsilon_{\chi}(q)}$ and $b$ have trivial pairing. This is equivalent to
$$
\prod_{p \mid b'} \eta_{\mathbb{Q}_p}\left(\left\{\chi_q^{\epsilon_{\chi}(q)}, b\right\}\right) \cdot \left(\prod_{\substack{p \mid \Delta_{K_\chi /\mathbb{Q}} \\ p \neq q}} \eta_{\mathbb{Q}_q}\left(\left\{\chi_p^{\epsilon_{\chi}(p)}, b\right\}\right)\right)^{-1} = 1.
$$
Applying Proposition \ref{Hilbert reciprocity} (Hilbert reciprocity) to the first factor, this happens if and only if
$$
\eta_{\mathbb{Q}_q}\left(\left\{\chi, b\right\}\right) = 1.
$$
Hence the statement that $b$ pairs trivially with all the $\chi_q$ is equivalent to
$$
\eta_{\mathbb{Q}_q}\left(\left\{\chi, b\right\}\right) = 1
$$
for each prime divisor $q$ of $\Delta_{K_\chi/\mathbb{Q}}$, which is in turn equivalent to
$$
\eta_{\mathbb{Q}_v}\left(\left\{\chi, b\right\}\right) = 1
$$ 
for all $v \in \Omega_{\mathbb{Q}}$. Therefore, from Proposition \ref{Hilbert reciprocity} again, we see that $\text{Up}_{K_\chi}(b)$ is a multiple of $1 - \zeta_l$ in the class group if and only if the cyclic algebra $\{\chi, b\}$ is trivial. Thanks to Corollary \ref{if and only if it is a norm}, this is equivalent to $b$ being a norm in $K_\chi$, which is precisely the desired statement. 
\end{proof}

Corollary \ref{being a norm in Kchi} can be proved directly without the detour through Redei matrices and class field theory. We devote the rest of this section to explain this different proof. 

\begin{proof}[Alternative proof of Corollary \ref{being a norm in Kchi}] 
Let $b$ in $\overline{\text{Cl}}(K_\chi)$ be such that $\text{Up}_{K_\chi}(b) \in (1 - \zeta_l)\text{Cl}(K_\chi)$. That means that there exists an ideal $I$ in $O_{K_\chi}$ and an element $\alpha$ in $K_\chi^\ast$ such that the following equality of fractional ideals holds
$$
(\alpha)(1 - \sigma)(I) = \text{Up}_{K_\chi}(b),
$$
where $\sigma$ is a non-trivial element of $\text{Gal}(K_\chi/\mathbb{Q})$. Taking the norm, as ideals, we find that
$$
\left(N_{K_\chi/\mathbb{Q}}(\alpha)\right) = (b)
$$
as equality of fractional ideals in $\mathbb{Q}$. Since $l$ is odd, we may change the sign of $\alpha$, if necessary, to find $\alpha' \in K_\chi$ satisfying
$$
N_{K_\chi/\mathbb{Q}}(\alpha') = b.
$$
For the converse, suppose that there exists an $\alpha \in K_\chi^\ast$ with $N_{K_\chi/\mathbb{Q}}(\alpha) = b$. By changing $\alpha$ we may assume that $b$ is free of $l$-th powers. We call an element of $O_{K_\chi}$ primitive if it is not divisible by any integer greater than $1$. Then we claim that there exists a primitive element $\widetilde{\alpha} \in O_{K_\chi}$ such that
\[
N_{K_\chi/\mathbb{Q}}(\widetilde{\alpha}) = bt^l
\]
for some non-zero integer $t$. Indeed, by scaling $\alpha$ we can ensure that $\alpha \in O_{K_\chi}$. If the resulting $\alpha$ is not primitive, we can take out common factors of $\alpha$ and $t$, since $b$ is free of $l$-th powers. This establishes the existence of $\widetilde{\alpha}$.

Now one sees that the factorization of $\widetilde{\alpha}$ has the form $\text{Up}_{K_\chi}(b)I$, with $I$ an integral ideal of norm $t^l$. We can discard the inert primes from $I$, because they are principal and their norm is also a $l$-th power. So we obtain an ideal $I'$ that is a product of split primes and whose norm is equal to ${t'}^l$ for a non-zero integer $t'$. 

Let $\sigma_1$ be the linear automorphism of $\mathbb{F}_l^l$ that sends $e_i$ to $e_{i + 1 \bmod l}$, where the vectors $e_i$ are the standard basis vectors. Then there is a representation of $\mathbb{F}_l$ on $\mathbb{F}_l^l$ given by sending $a$ to $\sigma_1^a$. In this way $\mathbb{F}_l^l$ becomes an $\mathbb{F}_l[\mathbb{F}_l]$-module, and is even isomorphic to $\mathbb{F}_l[\mathbb{F}_l]$ as an $\mathbb{F}_l[\mathbb{F}_l]$-module. We see that the sum zero vectors of $\mathbb{F}_l^l$ are precisely the image of the augmentation ideal under this isomorphism.

But note that the group ring $\mathbb{F}_l[\mathbb{F}_l]$ is isomorphic to the ring $\mathbb{F}_l[x]/x^l$. In the latter ring it is clear that the norm operator is $x^{l - 1}$ and hence the elements killed by the norm are precisely the multiples of $x$, which is exactly the augmentation ideal. Hence we conclude that $I'$ can be written as $\frac{\sigma(J)}{J} {J'}^l$. Projecting this onto the class group and remembering that $l$ is a multiple of $(1 - \zeta_l)^{l - 1}$, we find that $\text{Up}_{K_\chi}(b)$ is in $(1 - \zeta_l)\text{Cl}(K_\chi)[(1 - \zeta_l)^2]$.
\end{proof}

\subsection{\texorpdfstring{$(1 - \zeta_l)\text{Cl}(K_\chi)^{\vee}[(1 - \zeta_l)^2]$}{The right kernel}} 
\label{1-zeta-square are realizable cups}
Let $\chi \in \Gamma_{\mu_l}(\mathbb{Q})$. The following characterization of the right kernel of $\text{Redei}(K_\chi)$ follows easily from Proposition \ref{Redei as matrix of invs}; it can be proved with an argument analogous to the one given in the proof of Proposition \ref{being a norm in Kchi}. Instead, in the rest of this section, we shall opt for a different argument relying on central $\mathbb{F}_l$-extensions. Recall that we identify $\mathbb{F}_l \otimes \mathbb{F}_l$ with $\mathbb{F}_l$ through the map $a \otimes b \mapsto a \cdot b$, where the product is in $\mathbb{F}_l$. This allows us to view the cup of two $1$-cocycles in $\mathbb{F}_l$ as a $2$-cocycle with values in $\mathbb{F}_l$.

\begin{corollary} 
\label{characterizing 1-zetap square character via central extensions}
A character $\chi' \in \emph{Cl}(K_\chi)^{\vee}[1 - \zeta_l]$ is in $(1 - \zeta_l)\emph{Cl}(K_\chi)^{\vee}[(1 - \zeta_l)^2]$ if and only if $((i_l \circ j_l)^{-1} \circ \chi') \cup ({j_l}^{-1} \circ \chi)$ is trivial in $H^2(G_{\mathbb{Q}}, \mathbb{F}_l)$. 
\end{corollary}

\begin{proof}
Observe that, thanks to Proposition \ref{abstract criterion for realizing central extensions}, we have that 
$$
((i_l \circ j_l)^{-1} \circ \chi') \cup (j_l^{-1} \circ \chi) = 0
$$
in $H^2(G_{\mathbb{Q}}, \mathbb{F}_l)$ if and only if
$$
\left[\left(\left(i_l \circ j_l\right)^{-1} \circ \chi'\right) \cup (j_l^{-1} \circ \chi)\right]_{H^2(\text{Gal}(K_\chi K_{\chi'}/\mathbb{Q}), \mathbb{F}_l)} \in \widetilde{\text{Cent}}_{\mathbb{F}_l}\left(K_\chi K_{\chi'}/\mathbb{Q}\right).
$$
A straightforward local computation shows that for the primes  $q$ dividing $\Delta_{K_\chi /\mathbb{Q}}$ the $2$-cocycle 
\[
h(\chi, \chi') := \left(\left(i_l \circ j_l\right)^{-1} \circ \chi'\right) \cup (j_l^{-1} \circ \chi)
\]
is locally trivial if and only if the two characters $(i_l \circ j_l)^{-1} \circ \chi'$ and $j_l^{-1} \circ \chi$ are locally linearly dependent at each such $q$. In other words, $h(\chi, \chi')$ is locally trivial if and only if it is locally split at all primes dividing the discriminant. Therefore, we conclude by Proposition \ref{no ramification} that 
$$
\left[h(\chi, \chi')\right]_{H^2(G_{\mathbb{Q}}, \mathbb{F}_l)} = 0
$$
implies that $\chi' \in (1 - \zeta_l)\text{Cl}(K_\chi)[(1 - \zeta_l)^2]$. The converse follows immediately from a combination of Proposition \ref{Only need to check if it is Heisenberg} and Proposition \ref{when abelian by cyclic splits}.
\end{proof}

\section{Raw cocycles}
Let $A$ be a finite $\mathbb{Z}_l[\zeta_l]$-module and let the pair $(G, \chi)$ consist respectively of a finite cyclic group of size $l$ and of an isomorphism $\chi : G \to \langle \zeta_l \rangle$. Using $\chi$ we turn $A$ into a $G$-module killed by the norm operator $N_G := \sum_{g \in G} g \in \mathbb{Z}_l[G]$. We write $A \rtimes G$ for the semidirect product with respect to this action. Similarly, through $\chi$, we can view $N$, introduced in Section \ref{set up}, as a $G$-module and we use the symbol
$$
N(\chi)
$$
to denote the implicit $G$-module structure, which will play a central role for us. 

Whenever we have a quotient map $\widetilde{G} \to G$, we will talk by abuse of language of $N(\chi)$ as a $\widetilde{G}$-module with the induced action. For any profinite group $H$ and any discrete $H$-module $B$ we denote by
$$
\text{Cocy}(H, B)
$$
the group of continuous $1$-cocycles from $H$ to $B$. Recall that if $K \leq H$ is a subgroup and $\psi : H \to B$ is an element of $\text{Cocy}(H, B)$, then we call the element $\psi|_{K} \in \text{Cocy}(K,B)$ the \emph{restriction} of $\psi$ to $K$. If we have a surjective homomorphism $\pi : \widetilde{H} \twoheadrightarrow H$, then we call the element $\psi \circ \pi \in \text{Cocy}(\widetilde{H}, B)$ the \emph{inflation} of $\psi$ to $\widetilde{H}$. 

\begin{prop} 
\label{exact sequence of cocycles}
Let $(G, \chi)$ and $A$ be as above and $k$ a positive integer. Inflation and restriction of $1$-cocycles induce a split exact sequence of $\mathbb{Z}_l[\zeta_l]$-modules
$$
0 \to \emph{Cocy}(G, N(\chi))[(1 - \zeta_l)^k] \to \emph{Cocy}(A \rtimes G, N(\chi))[(1 - \zeta_l)^k] \to A^{\vee}[(1 - \zeta_l)^k] \to 0.
$$
Moreover, $\emph{Cocy}(G, N(\chi))[(1 - \zeta_l)^k] \simeq_{\mathbb{Z}_l[\zeta_l]} N[(1 - \zeta_l)^k]$, with such an isomorphism arising from the evaluation map $\psi \mapsto \psi(\sigma)$ for $\sigma$ a fixed non-trivial element of $G$.
\end{prop}

\begin{proof}
Let $\sigma$ denote a generator of $G$. Since $G$ is cyclic, we have that the map from $\text{Cocy}(G, N(\chi))[(1 - \zeta_l)^k]$ to $N[(1 - \zeta_l)^k]$ sending $\psi$ to $\psi(\sigma)$ induces an isomorphism between $\text{Cocy}(G, N(\chi))[(1 - \zeta_l)^k]$ and the kernel of the norm $\text{Norm}_G$ operating on $N(\chi)[(1 - \zeta_l)^k]$, which is the full $N(\chi)[(1 - \zeta_l)^k]$. Hence the evaluation map $\psi \mapsto \psi(\sigma)$ induces an isomorphism of $\mathbb{Z}_l[\zeta_l]$-modules between $\text{Cocy}(G, N(\chi))[(1 - \zeta_l)^k]$ and $N(\chi)[(1 - \zeta_l)^k]$ as claimed. 

This implies that the natural inclusion 
$$
0 \to \text{Cocy}(G, N(\chi))[(1 - \zeta_l)^k] \to \text{Cocy}(A \rtimes G, N(\chi))[(1 - \zeta_l)^k],
$$
induced by inflation of cocycles, is split, since $\text{Cocy}(A \rtimes G, N(\chi))[(1 - \zeta_l)^k]$ is finite and certainly killed by $(1 - \zeta_l)^k$. Therefore we are left with showing that the natural map, induced by restriction of cocycles,
$$
\text{Cocy}(A \rtimes G, N(\chi))[(1 - \zeta_l)^k] \to A^{\vee}[(1 - \zeta_l)^k]
$$
is surjective. Let $\phi \in A^{\vee}[(1 - \zeta_l)^k]$. Consider the map
$$
\psi : A \rtimes G \to N(\chi)[(1 - \zeta_l)^k],
$$
which sends $(a, g)$ to $\phi(a)$. By construction $\psi$ restricts to $\phi$. Hence we are left with checking that $\psi$ is a $1$-cocycle. We have by definition of semidirect product and of the map $\psi$
$$
\psi((a_1, g_1) (a_2, g_2)) = \psi(a_1 + \chi(g_1)a_2, g_1g_2) = \phi(a_1) + \chi(g_1)\phi(a_2).
$$
On the other hand, in order for $\psi$ to be a $1$-cocycle for the action of $G$ on $A$ it must satisfy
$$
\psi((a_1, g_1) (a_2,g_2)) = \chi(g_1) \psi(a_2, g_2) + \psi(a_1, g_1) = \chi(g_1)\phi(a_2) + \phi(a_1),
$$
precisely the same equation as above. This concludes our proof. 
\end{proof}

The next proposition provides, roughly speaking, a converse to Proposition \ref{exact sequence of cocycles}. This will be useful later: it will tell us that the Galois group of the field of definition of a $1$-cocycle in $N(\chi)$ always splits as a semidirect product.

\begin{prop} 
\label{Galois groups of cocycles are always semidirect}
Let $(G, \chi)$ be as above. Let $\pi : \widetilde{G} \twoheadrightarrow G$ be a surjective homomorphism and let moreover $\psi \in \emph{Cocy}(\widetilde{G}, N(\chi))[(1 - \zeta_l)^{k}]$ be a cocycle such that the image of the character $\psi|_{\emph{ker}(\pi)}$ is $N(\chi)[(1 - \zeta_l)^{k}]$. Set $H := \{g \in \emph{ker}(\pi) : \psi(g) = 0\}$. We have the following facts.\\
$(1)$ The set $H$ is a normal subgroup of $\widetilde{G}$. \\ 
$(2)$ The assignment 
$$
f : \frac{\widetilde{G}}{H} \to N(\chi)[(1 - \zeta_l)^{k}] \rtimes \langle \zeta_l \rangle,
$$
defined by the formula $f(g) = (\psi(g), \chi(g))$, is a group isomorphism. 
\end{prop}

\begin{proof}
Let us begin by verifying $(1)$. The restriction of $\psi$ to $\text{ker}(\pi)$ is a group homomorphism. Moreover, for each $g \in \widetilde{G}$ and $h \in \text{ker}(\pi)$ we have that 
\begin{align*}
\psi(ghg^{-1}) &= \chi(g)\psi(hg^{-1}) + \psi(g) = \chi(g)\psi(g^{-1}) + \chi(g)\psi(h) + \psi(g) \\
&= \psi(gg^{-1}) + \chi(g)\psi(h) = \chi(g)\psi(h).
\end{align*} 
This immediately implies $(1)$. Next take $g_1,g_2 \in \widetilde{G}$, then
\begin{align*}
(\psi(g_1), \chi(g_1)) (\psi(g_2), \chi(g_2)) &= (\psi(g_1), 1) (0, \chi(g_1)) (\psi(g_2), 1) (0, \chi(g_1)^{-1}) (0, \chi(g_1g_2)) \\
&= (\psi(g_1) + \chi(g_1)\psi(g_2), \chi(g_1g_2)) = (\psi(g_1g_2), \chi(g_1g_2)).
\end{align*} 
Hence $f$ is a group homomorphism and it is zero precisely for the $g$ with $\psi(g) = 0$ and $\chi(g) = 1$. This is the definition of $H$. 
\end{proof}

Let now $\chi \in \Gamma_{\mu_l}(\mathbb{Q})$. We set $A := \text{Cl}(K_\chi)[(1 - \zeta_l)^{\infty}]$ and $G := \text{Gal}(K_\chi/\mathbb{Q})$. Denote by $H_{\chi, l}$ the largest subextension of the Hilbert class field of $K_\chi$, within $\overline{\mathbb{Q}}$, having degree a power of $l$. The Artin map gives a canonical identification $\text{Cl}(K_\chi)[(1 - \zeta_l)^{\infty}] = \text{Gal}(H_{\chi, l}/K_\chi)$. 

\begin{prop} 
\label{Hilbert class field splits}
The surjection $\emph{Gal}(H_{\chi, l}/\mathbb{Q}) \twoheadrightarrow \emph{Gal}(K_\chi/\mathbb{Q})$ admits a section, yielding an isomorphism
$$
\emph{Gal}(H_{\chi, l}/\mathbb{Q}) \simeq_{\emph{gr}} \emph{Cl}(K_\chi)[(1 - \zeta_l)^{\infty}] \rtimes \emph{Gal}(K_\chi/\mathbb{Q})
$$
inducing the Artin identification between $\emph{Gal}(H_{\chi, l}/K_\chi)$ and $\emph{Cl}(K_\chi)[(1 - \zeta_l)^{\infty}]$ when restricted to $\emph{Gal}(H_{\chi, l}/K_\chi)$.
\end{prop}

\begin{proof}
This follows at once from Proposition \ref{when abelian by cyclic splits}. 
\end{proof}

\noindent Hence we deduce the following.

\begin{prop} 
\label{direct sum of cocycle groups}
Let $k$ be a positive integer. The natural map
$$
\emph{Cocy}(\emph{Gal}(H_{\chi, l}/\mathbb{Q}), N(\chi))[(1 - \zeta_l)^k] \to \emph{Cl}(K_\chi)^{\vee}[(1 - \zeta_l)^k]
$$
is a split surjection of $\mathbb{Z}_l[\zeta_l]$-modules yielding an isomorphism
$$
\emph{Cocy}(\emph{Gal}(H_{\chi, l}/\mathbb{Q}), N(\chi))[(1 - \zeta_l)^k] \simeq \emph{Cl}(K_\chi)^{\vee}[(1 - \zeta_l)^k] \oplus N[(1 - \zeta_l)^k].
$$
\end{prop}

\begin{proof}
This follows immediately upon combining Proposition \ref{Hilbert class field splits} and Proposition \ref{exact sequence of cocycles}.
\end{proof}

\noindent In particular we conclude the following important fact.

\begin{corollary} 
\label{lifting character is the same as lifting unramified cocycles}
Let $k$ be a positive integer and let $\psi \in \emph{Cl}(K_\chi)^{\vee}[(1 - \zeta_l)^k]$. The following are equivalent
\begin{enumerate}
\item[(1)] we have that $\psi \in (1 - \zeta_l)\emph{Cl}(K_\chi)^{\vee}[(1 - \zeta_l)^{k + 1}]$;
\item[(2)] there is a $\widetilde{\psi} \in (1 - \zeta_l)\emph{Cocy}(\emph{Gal}(H_{\chi, l}/\mathbb{Q}), N(\chi))[(1 - \zeta_l)^{k + 1}]$ restricting to $\psi$;
\item[(3)] for any $\widetilde{\psi} \in \emph{Cocy}(\emph{Gal}(H_{\chi, l}/\mathbb{Q}), N(\chi))[(1 - \zeta_l)^k]$ such that $\widetilde{\psi}$ restricts to $\psi$  we have that $\widetilde{\psi} \in (1 - \zeta_l)\emph{Cocy}(\emph{Gal}(H_{\chi, l}/\mathbb{Q}), N(\chi))[(1 - \zeta_l)^{k + 1}]$.
\end{enumerate}
\end{corollary}

\begin{proof}
This is a trivial consequence of Proposition \ref{direct sum of cocycle groups}.
\end{proof}

Corollary \ref{lifting character is the same as lifting unramified cocycles} tells us that finding an $(1 - \zeta_l)$-lift for an unramified character is the same as finding an $(1 - \zeta_l)$-lift for an unramified cocycle representing our character. In turn we now show that finding an $(1 - \zeta_l)$-lift for an unramified cocycle is often the same as finding an $(1 - \zeta_l)$-lift of the cocycle inflated to the absolute Galois group, providing in total a very convenient criterion for the existence of an $(1 - \zeta_l)$-lift of a character in terms of cocycles from $G_{\mathbb{Q}}$ to $N(\chi)$. This is given in the following important proposition of which we will give two completely different proofs; one based on the facts about central extension established in Section \ref{central extensions}, while the other proof is based on class field theory. 

\begin{prop} 
\label{liftable unramified cocycle can always be liftable unramified}
Let $\psi \in \emph{Cocy}(\emph{Gal}(H_{\chi, l}/\mathbb{Q}), N(\chi))[(1 - \zeta_l)^k]$ for some integer $k \geq 1$. In case $l$ divides $\Delta_{K_\chi/\mathbb{Q}}$ we assume that $\emph{Up}_{K_\chi}(l)$ splits completely in the extension $L(\psi)K_\chi/K_\chi$, where we recall that $L(\psi)$ denotes the field of definition of $\psi$. Then the following are equivalent
\begin{enumerate}
\item[(1)] $\psi \in (1 - \zeta_l)\emph{Cocy}(\emph{Gal}(H_{\chi, l}/\mathbb{Q}), N(\chi))[(1 - \zeta_l)^{k + 1}]$;
\item[(2)] $\psi \in (1 - \zeta_l)\emph{Cocy}(G_{\mathbb{Q}}, N(\chi))[(1 - \zeta_l)^{k + 1}]$.
\end{enumerate}
\end{prop}

\begin{proof}[Proof using central $\mathbb{F}_l$-extensions]
It is a triviality that $(1)$ implies $(2)$. We now show that $(2)$ implies $(1)$. We can assume without loss of generality that $\psi$ restricted to $G_{K_\chi}$ surjects onto $N[(1 - \zeta_l)^k]$. Let $L := L(\psi) \cdot K_\chi$ and let $\widetilde{\psi}$ be a lift of $\psi$. Thanks to Proposition \ref{Galois groups of cocycles are always semidirect} we find that the map
\begin{align}
\label{ePsiIsom}
\text{Gal}\left(L(\widetilde{\psi})/\mathbb{Q}\right) \to N[(1 - \zeta_l)^{k + 1}] \rtimes \langle \zeta_l \rangle
\end{align}
defined by the formula $g \mapsto (\widetilde{\psi}(g), \chi(g))$ is an isomorphism. We see that this map sends $\text{Gal}(L(\widetilde{\psi})/L)$ into $N[1 - \zeta_l]$. Therefore we deduce 
$$
j_l^{-1} \circ i_l^{-1} \circ \widetilde{\psi} \in \Gamma_{\mathbb{F}_l}(L)^{\text{Gal}(L/\mathbb{Q})}.
$$
The only primes that ramify in the extension $L/\mathbb{Q}$ are those dividing $\Delta_{K_\chi /\mathbb{Q}}$, since by assumption the extension $L/K_\chi $ is contained in $H_{K_\chi}$ and hence is unramified. For each $q$ different from $l$ dividing $\Delta_{K_\chi /\mathbb{Q}}$ we have that $\text{Up}_{K_\chi}(q)$ splits completely up to $K_\chi L((1 - \zeta_l)\psi)$. 

As we next explain, from this we conclude that if we restrict $f := (\widetilde{\psi}, \chi)$ to $I_q$, an inertia subgroup of $q$, the image of (\ref{ePsiIsom}) is a group of the form $\langle N[1 - \zeta_l], \sigma \rangle$, where $\chi(\sigma) \neq 1$. Indeed, since $\chi(I_q) \neq \{1\}$, we have that $f(I_q) \cap N$ is a $\mathbb{Z}_l[\zeta_l]$-submodule; we have $g \in I_q$ with $\chi(g) = \zeta_l$ and conjugation by $g$ on $f(I_q) \cap N$ equals precisely multiplication by $\zeta_l$. If $f(I_q) \cap N$ is trivial we are certainly done. On the other hand it can not have size more than $l$, since $\chi(I_q) \neq \{1\}$ and $I_q$ is of size at most $l^2$. The only non-trivial $\mathbb{Z}_l[\zeta_l]$-submodule of size $l$ is $N[1 - \zeta_l]$. This shows our claim. 

Recall that all $\sigma \in N[(1 - \zeta_l)^{k+1}] \rtimes \langle \zeta_l \rangle$ with $\chi(\sigma) \neq 1$ have order $l$. From this, we conclude that the extension is locally split at inertia at $q$. In particular, in case $l$ divides $\Delta_{K_\chi /\mathbb{Q}}$, we have by assumption that inertia at $l$ equals the decomposition group, hence the extension is locally split at $l$. Hence we conclude, by Proposition \ref{no ramification}, that we can find $\chi' \in \Gamma_{\mathbb{F}_l}(\mathbb{Q})$ such that $L(\widetilde{\psi}+\chi')/L$ is unramified. This ends the proof.
\end{proof}

\begin{remark}
If $k \geq 2$, we claim that $L = L(\psi)$. This fact will be important throughout the paper. Indeed, if $n$ denotes an element of $G_{\mathbb{Q}}$ with $(1 - \zeta_l)\psi(n) \neq 0$, then for each $g \in G_{\psi}$, the group of definition of $\psi$, we have that 
$$
\psi(g \cdot n) = \chi(g) \cdot n + \psi(g) \quad \text{  and  } \quad \psi(g) = \psi(\textup{id}) = 0,
$$
where each equality is justified by the fact that $\psi$ is a cocycle. Therefore we find that $\chi(g) = 1$, thanks to the fact that $(1 - \zeta_l)\psi(n) \neq 0$. That means that $G_{\psi} \subseteq G_{K_\chi}$, which is equivalent to $L(\psi) \supseteq K_\chi$.
\end{remark}

\begin{proof}[Proof using class field theory]
Let $\widetilde{\psi}$ be a lift of $\psi$. Observe that, using the Artin map, $\widetilde{\psi}|_{G_{K_\chi}}$ can be naturally identified with an element of 
$$
\left(\frac{\text{Cl}(K_\chi,c)}{\text{Norm}_{\text{Gal}(K_\chi /\mathbb{Q})} \text{Cl}(K_\chi, c)}\right)^{\vee}
$$ 
for some positive integer $c$, where $\text{Cl}(K_\chi, c)$ denotes the ray class group of modulus $c$. Here 
$$
\frac{\text{Cl}(K_\chi, c)}{\text{Norm}_{\text{Gal}(K_\chi/\mathbb{Q})} \text{Cl}(K_\chi, c)}
$$
is the largest quotient of $\text{Cl}(K, c)$ where the action of $\mathbb{Z}[\text{Gal}(K_\chi/\mathbb{Q})]$ factors through $\mathbb{Z}[\zeta_l]$. Moreover, thanks to Proposition \ref{little ramification} if $l$ does not divide $\Delta_{K_\chi /\mathbb{Q}}$ and thanks to Proposition \ref{no ramification} if $l$ divides $\Delta_{K_\chi /\mathbb{Q}}$, we can reduce to the case that $c$ is coprime to $l$. 

Our next step is to reduce to the case that $c$ is even coprime to $\Delta_{K_{\chi}/\mathbb{Q}}$. Indeed, from Proposition \ref{Galois groups of cocycles are always semidirect} we have an isomorphism
$$
\text{Gal}(L(\widetilde{\psi})/\mathbb{Q}) \simeq N[(1-\zeta_l)^{k+1}] \rtimes \langle \zeta_l \rangle
$$
given by the map $(\widetilde{\psi}, \chi)$. Hence if the extension $L(\widetilde{\psi})/K_{\chi}$ were to ramify at $\text{Up}_{K_{\chi}}(q)$ for some $q \mid \Delta_{K_{\chi}/\mathbb{Q}}$ different from $l$, then we would obtain a totally ramified $\mathbb{F}_l \times \mathbb{F}_l$-extension of $\mathbb{Q}_q$. This is not possible for $q \neq l$. This completes our reduction step.

Thanks to Corollary \ref{lifting character is the same as lifting unramified cocycles} we see that the statement would certainly follow if we prove that the inclusion
$$
\text{Cl}(K_\chi)^{\vee} \subseteq \left(\frac{\text{Cl}(K_\chi, c)}{\text{Norm}_{\text{Gal}(K_\chi /\mathbb{Q})} \text{Cl}(K_\chi, c)}\right)^{\vee}
$$
remains injective after tensoring with $\frac{\mathbb{Z}_l[\zeta_l]}{(1 - \zeta_l)}$. We claim that
\begin{multline}
\label{eClKcDim}
\text{dim}_{\mathbb{F}_l} \left(\left(\frac{\text{Cl}(K_\chi, c)}{\text{Norm}_{\text{Gal}(K_\chi /\mathbb{Q})} \text{Cl}(K_\chi, c)}\right)^{\vee}[1 - \zeta_l]\right) = \\
\text{dim}_{\mathbb{F}_l} (\text{Cl}(K_\chi)^{\vee}[1 - \zeta_l]) + \text{dim}_{\mathbb{F}_l} \left({(O_{K_\chi}/c)^\ast}^{\text{Gal}(K_\chi/\mathbb{Q})}[l]\right).
\end{multline}
To prove equation (\ref{eClKcDim}), it is enough to show the inequality 
\begin{multline*}
\text{dim}_{\mathbb{F}_l} \left(\left(\frac{\text{Cl}(K_\chi, c)}{\text{Norm}_{\text{Gal}(K_\chi /\mathbb{Q})} \text{Cl}(K_\chi, c)}\right)^{\vee}[1 - \zeta_l]\right) \geq \\
\text{dim}_{\mathbb{F}_l} (\text{Cl}(K_\chi)^{\vee}[1 - \zeta_l]) + \text{dim}_{\mathbb{F}_l} \left({(O_{K_\chi}/c)^\ast}^{\text{Gal}(K_\chi/\mathbb{Q})}[l]\right),
\end{multline*}
since the other inequality is trivially true. To prove the above inequality just consider the characters $\chi_q$ for $q \mid c \Delta_{K_\chi/\mathbb{Q}}$. 
\end{proof}

\begin{remark}
An important feature of the second proof is that the image of the restriction map from $\textup{Cocy}(G_{\mathbb{Q}}, N(\chi))$ to $\textup{Cl}(K_{\chi}, c)^\vee$ are characters with two additional properties. Firstly, they must be killed by the norm. Secondly, Proposition \ref{Galois groups of cocycles are always semidirect} implies that the resulting Galois group is semidirect over $\mathbb{Q}$. Both properties are crucially used in the proof. 

The second property is necessary to get rid of the $c$ not coprime to $\Delta_{K_{\chi}/\mathbb{Q}}$. Only then one proceeds to show that the sequence of $(1 - \zeta_l)$-torsion remains exact. We remark that this is not true in general. Indeed, it should be possible to construct an example where the sequence does not remain exact after taking $(1-\zeta_l)$-torsion. In such an example one would need characters that are $(1-\zeta_l)$-liftable only in the dual ray class group and not in the dual class group. However, the Galois group of such a character will never be semidirect over $\mathbb{Q}$. Hence they will never arise as restrictions of cocycles from $N(\chi)$. 
\end{remark}

\noindent We next define raw cocycles for $\chi$. 

\begin{mydef} 
\label{raw cocycles}
A raw cocycle for $\chi$ is a finite sequence 
$$
\{\psi_i\}_{0 \leq i \leq j}
$$
with $\psi_i \in \emph{Cocy}(\emph{Gal}(H_{\chi, l}/\mathbb{Q}), N(\chi))[(1 - \zeta_l)^i]$, with $(1 - \zeta_l)\psi_i = \psi_{i - 1}$ for each $1 \leq i \leq j$. The integer $j$ is called the rank of the raw cocycle.
\end{mydef}

\section{A reflection principle}
\label{sRel}
\subsection{The differential of sums of cocycles} 
\label{differential of the sum}
In this section if $M$ denotes a multiset, then we let $\text{Set}(M)$ be the corresponding set and call them the elements of $M$. If $x_0 \in \text{Set}(M)$, we extend the usual operations $-$, $=$, $\neq$ and $|\cdot|$ of sets to multisets in the natural way. Let $m$ be a positive integer and let
$$
Y := \prod_{i = 1}^m Y_i \times \{d\},
$$
be a product multiset, where each $Y_i$ is a multiset of $l$ primes that are all $1$ modulo $l$. We further assume that the $Y_i$ are all distinct multisets as $i$ varies in $[m]$, and that the elements in each $Y_i$ are coprime to the integer $d$, which is also a product of distinct primes $q$ with $q = l$ or $q$ equal to $1$ modulo $l$.

We will identify points in this product space $x \in Y$ with the integer $D_x := \left(\prod_{i = 1}^m \pi_i(x)\right)d$ whenever convenient. A subset $C \subseteq Y$ is called a sub-cube in case it is the inverse image of a singleton under the projection of $Y$ on $\prod_{i \in T} Y_i$ for a subset $T$ of $\{1, \ldots, m\}$. The non-negative integer $m - |T|$ is called the \emph{dimension} of $C$ and denoted by $\text{dim}(C)$. One has that $l^{\text{dim}(C)} = \left|C\right|$. 

Call an \emph{amalgama} for $Y$ a map 
$$
\epsilon_Y : \bigcup_{i = 1}^m \text{Set}(Y_i) \cup \{q \mid d\}_{q \text{ prime}} \to [l - 1].
$$
If $\epsilon_Y$ is an amalgama for $Y$ we obtain for each $D \in Y$ an amalgama $\epsilon_Y(D)$ for $D$. Conversely, giving an amalgama for $Y$ is the same as assigning an amalgama at each $D \in Y$ in a consistent manner, i.e. taking the same value at a prime whenever that prime divides two different point of $Y$. In this manner given $\epsilon_Y$, an amalgama for $Y$, we have for each $D \in Y$ a character
$$
\chi_{\epsilon_Y}(D) := \chi_{\epsilon_Y(D)}(D) = \prod_{p \mid D} \chi_p^{\epsilon_Y(p)}.
$$
Next, for $\epsilon_Y$ an amalgama for $Y$, a raw cocycle for $(Y, \epsilon_Y)$ is an assignment that gives to each point $D \in Y$ a raw cocycle for $\chi_{\epsilon_Y}(D)$ (see Definition \ref{raw cocycles}) with each point having rank at least $m - 1$. We will write $\psi_k(\mathfrak{R}, \chi_{\epsilon_Y}(D))$ for the $\psi_k$ in the raw cocycle of $\chi_{\epsilon_Y}(D)$.

Let $x_0 \in \text{Set}(Y)$. We say that a raw cocycle $\mathfrak{R}$ for $Y$ is \emph{promising} with respect to $x_0$, if $\mathfrak{R}(\chi_{\epsilon_Y}(x))$ has rank at least $m$ at each $x \in Y - \{x_0\}$ and 
$$
\sum_{x \in H} \psi_j(\mathfrak{R}, \chi_{\epsilon_Y}(x)) \in N[1 - \zeta_l]
$$
for each proper sub-cube of $Y$ having dimension $j$ not containing $x_0$. The rest of this subsection is devoted to computing for each $\sigma, \tau \in G_{\mathbb{Q}}$ the value of
$$
d_{x_0}\left(\sum_{x \in Y: x \neq x_0} \psi_m(\mathfrak{R}, \chi_{\epsilon_Y}(x))\right)(\sigma, \tau),
$$
where $\mathfrak{R}$ is a promising raw cocycle with respect to $x_0$ and the differential $d_{x_0}$ is taken by considering the sum as a $1$-cochain in the $G_{\mathbb{Q}}$-module $N(\chi_{\epsilon_Y}(x_0))$. We remind the reader that $Y$ is a multiset, so that the above sum has to be computed with the corresponding multiplicities. To state the result of this calculation, we need some additional definitions. 

Firstly, we assume that each $Y_i$ is written as $\{p_{i0}, \ldots, p_{i(l - 1)}\}$ with the convention that the coordinates of $x_0$ occupy the first $s(i)$ indexes for each $i$. In this manner the multiset $Y$ corresponds bijectively with the set 
$$
\mathcal{F} := \text{Map}([m], \{0, \ldots, l - 1\})
$$
with the points of $Y$ giving $x_0$ in $\text{Set}(Y)$ corresponding to the functions $f$ with $f(i) \leq s(i) - 1$ for each $i$ in $[m]$. We call $\mathcal{F}_{x_0}$ the set of such functions.

For each $j \in [m], i \in \{0, \ldots, l - 1\}$ we define the character
$$
\chi_{j, i, \epsilon_Y} := -j_l^{-1} \circ \chi_{p_{j0}}^{\epsilon_Y(p_{j0})} + j_l^{-1} \circ \chi_{p_{ji}}^{\epsilon_Y(p_{ji})}.
$$
The following function attached to each $\tilde{f} \in \mathcal{F} - \{0\}$ will play an important role
$$
\chi_{\tilde{f}, \epsilon_Y}(\sigma) := \prod_{j \in [m] : \tilde{f}(j) \neq 0} \chi_{j, \tilde{f}(j), \epsilon_Y}(\sigma).
$$
Here the product takes place in $\mathbb{F}_l$ and the map is a continuous $1$-cochain from $G_{\mathbb{Q}}$ to $\mathbb{F}_l$. For each $\tilde{f} \in \mathcal{F}- \mathcal{F}_{x_0}$ we denote by $H_{\tilde{f}}$ the set of $f \in \mathcal{F}$ such that for all $j \in [m]$, we have that $\tilde{f}(j) \neq 0$ implies $\tilde{f}(j) = f(j)$. We can now state the following fact.

\begin{prop} 
\label{combinatorial formula quite in general}
Let $Y$, $\epsilon_Y$, $x_0$ be as above. Let $\mathfrak{R}$ be a raw cocycle for $(Y, \epsilon_Y)$ that is promising for $x_0$. Then
\begin{multline*}
d_{x_0}\left(\sum_{x \in Y: x \neq x_0} \psi_m(\mathfrak{R}, \chi_{\epsilon_Y}(x))\right) = \\ \sum_{\tilde{f} \in \mathcal{F} - \{0\}} (-1)^{\left|\{j : \tilde{f}(j) \neq 0\}\right| + 1} \chi_{\tilde{f}}(\sigma) \left(\sum_{f \in H_{\tilde{f}}} \psi_{m - \left|\{j: \tilde{f}(j) \neq 0\}\right|}(f)(\tau)\right).
\end{multline*}
\end{prop}

\begin{proof}
We start with a simple computation
\begin{align*}
d_{x_0}\left(\psi_m(x)\right) &= \chi_{\epsilon_Y}(x_0)\psi_m(x)(\tau) + \psi_m(x)(\sigma) - \psi_m(x)(\sigma \tau) \\
&= -\chi_{\epsilon_Y}(x)(\sigma) \psi_m(x)(\tau) + \chi_{\epsilon_Y}(x_0)(\sigma)\psi_m(x)(\tau) \\
&= \chi_{\epsilon_Y}(x_0)(\sigma) \left(1 - \frac{\chi_{\epsilon_Y}(x)(\sigma)}{\chi_{\epsilon_Y}(x_0)(\sigma)} \right) \psi_m(x)(\tau).
\end{align*}
Put $\chi := \chi_{\epsilon_Y}(x_0)$. Summing over all the $\psi_m$ yields
\begin{align}
\label{eSumPsi}
\chi(\sigma) \sum_{f \in \mathcal{F} - \mathcal{F}_{x_0}} \left(1 - \frac{\chi_{\epsilon_Y}(f)(\sigma)}{\chi(\sigma)}\right) \psi_{m}(f)(\tau),
\end{align}
where $\chi_{\epsilon_Y}(f)$ is defined to be $\chi_{\epsilon_Y}(x)$ for the unique $x$ that corresponds to $f$ under our bijection. Note that this is not the same as $\chi_{f, \epsilon_Y}$. For $f \in \mathcal{F}, \sigma \in G_{\mathbb{Q}}$ define 
$$
\Sigma(f, \sigma) := \sum_{j \in [m]} \chi_{j, f(j), \epsilon_Y}(\sigma).
$$
Observe that for $f \in \mathcal{F}_{x_0}$ the symbol $\psi_m(f)$ is not defined; however we use the convention that $0 \cdot \psi_m(f)$ is always defined and it is equal to $0$, moreover we also assume that the symbol $(1 - \zeta_l)^i \psi_m(f)$ is defined to be $\psi_{m - i}(f)$ for each integer $0 \leq i \leq m$. This will simplify the notation in the coming calculations. With these conventions, equation (\ref{eSumPsi}) is equal to
\begin{align}
\label{eDx01}
\chi(\sigma) \sum_{f \in \mathcal{F}} \left(1 - \zeta_l^{\Sigma(f,\sigma)}\right) \psi_m(f)(\tau).
\end{align}
Define subsets
$$
T_{\sigma, f}^i := \left\{j \in [m]: \chi_{j, f(j), \epsilon_Y}(\sigma) = i\right\}.
$$
We obviously have $T_{\sigma, f}^i \cap T_{\sigma, f}^j = \emptyset$ for $i \neq j$. We can now rewrite equation (\ref{eDx01}) as
\begin{multline}
\label{eDx02}
\chi(\sigma) \sum_{f \in \mathcal{F}} \left(1 - \prod_{i = 1}^{l - 1} \zeta_l^{i \left|T_{\sigma,f}^i\right|}\right) \psi_m(f)(\tau) = \\
\chi(\sigma) \sum_{f \in \mathcal{F}} \left(1 - \prod_{i = 1}^{l - 1} (1 + (\zeta_l^i - 1))^{\left|T_{\sigma,f}^i\right|}\right) \psi_m(f)(\tau) = \\
\chi(\sigma) \sum_{f \in \mathcal{F}} a(f, \sigma)\psi_m(f)(\tau),
\end{multline}
where
$$
a(f, \sigma) := 1 - \sum_{(T_1, \ldots, T_{l - 1}) : T_i \subseteq T_{\sigma,f}^i} (-1)^{\sum_{i = 1}^{l - 1} \left|T_i\right|} \prod_{i = 1}^{l - 1}\left(\sum_{j = 0}^{i - 1} \zeta_l^j\right)^{\left|T_i\right|} (1 - \zeta_l)^{\sum_{i = 1}^{l - 1} \left|T_i\right|}.
$$
We can expand equation (\ref{eDx02}) as
\begin{multline}
\label{eDx03}
-\chi(\sigma) \sum_{f \in \mathcal{F}} \sum_{\substack{(T_1, \ldots, T_{l - 1}) \neq (\emptyset, \ldots, \emptyset) \\ T_i \subseteq T_{\sigma,f}^i}} (-1)^{\sum_{i = 1}^{l - 1} \left|T_i\right|} \prod_{i = 1}^{l - 1} \left(\sum_{j = 0}^{i - 1} \zeta_l^j\right)^{\left|T_i\right|} \psi_{m - \sum_{i = 1}^{l - 1}\left|T_i\right|}(f)(\tau) = \\
-\chi(\sigma) \sum_{\substack{(T_1, \ldots, T_{l - 1}) \neq (\emptyset, \ldots, \emptyset) \\ \text{all disjoint}}} (-1)^{\sum_{i = 1}^{l - 1}\left|T_i\right|} \prod_{i = 1}^{l - 1} \left(\sum_{j = 0}^{i - 1} \zeta_l^j\right)^{\left|T_i\right|} \cdot \\
\sum_{f \in \mathcal{F}} \mathbf{1}_{\forall i : T_i \subseteq T_{\sigma,f}^i}(\sigma) \psi_{m - \sum_{i = 1}^{l - 1}\left|T_i\right|}(f)(\tau).
\end{multline}
We next make a definition. For $T_\bullet := (T_1, \ldots, T_{l - 1})$ we say that $\tilde{f}$ is $(T_\bullet, \sigma)$-good in case we have the following two conditions
\begin{itemize}
\item for each $i \in [l - 1]$ we have that $T_i \subseteq T_{\sigma,\tilde{f}}^i$;
\item for each $j \not \in \bigcup_{i = 1}^{l - 1} T_i$ we have that $\tilde{f}(j) = 0$.  
\end{itemize}
The following remark will be useful. 

\begin{remark} 
\label{silly remark}
Given $T_\bullet$ there can be \emph{many} functions $\tilde{f}$ that are $(T_\bullet, \sigma)$-good. But for each $\tilde{f} \in \mathcal{F}$ there is at most one ordered choice of $l - 1$ disjoint sets $T_\bullet$ such that $\tilde{f}$ is $(T_\bullet, \sigma)$-good; this partition exists if and only if $\chi_{\widetilde{f}, \epsilon_Y}(\sigma) \neq 0$. It is simply obtained by declaring that $j \in T_i$ if and only if $\chi_{j, \tilde{f}(j), \epsilon_Y}(\sigma) = i$. 
\end{remark}

We now use the definition of $(T_\bullet, \sigma)$-good to rearrange the sum in equation (\ref{eDx03}) as a sum over sub-cubes
\begin{multline}
\label{eDx04}
-\chi(\sigma) \sum_{\substack{(T_1, \ldots, T_{l - 1}) \neq (\emptyset, \ldots, \emptyset) \\ \text{all disjoint}}} (-1)^{\sum_{i = 1}^{l - 1} \left|T_i\right|} \prod_{i = 1}^{l - 1} \left(\sum_{j = 0}^{i - 1} \zeta_l^j\right)^{\left|T_i\right|} 
\cdot \\
\sum_{\tilde{f} \in \mathcal{F} - \mathcal{F}_{x_0}} \mathbf{1}_{\tilde{f} \text{ is } (T_\bullet, \sigma)-\text{good}} 
\cdot \sum_{f \in H_{\tilde{f}}}  \psi_{m-\sum_{i=1}^{l - 1}\left|T_i\right|}(f)(\tau).
\end{multline}
Here $H_{\tilde{f}}$ denotes the set of $f \in \mathcal{F}$ such that $\tilde{f}(j) \neq 0$ implies $\tilde{f}(j) = f(j)$ for all $j \in [m]$. Thanks to the fact that the raw cocycle is promising we can rewrite equation (\ref{eDx04}) simply as
$$
-\sum_{\substack{(T_1, \ldots, T_{l - 1}) \neq (\emptyset, \ldots, \emptyset) \\ \text{all disjoint}}}\sum_{\tilde{f} \in \mathcal{F} - \{0\}} (-1)^{\left|\{j : \tilde{f}(j) \neq 0 \}\right|} \chi_{\widetilde{f}, \epsilon_Y}(\sigma) \mathbf{1}_{\tilde{f} \text{ is } (T_\bullet, \sigma)-\text{good}} \cdot \sum_{f \in H_{\tilde{f}}} \psi_{m - \left|\{j : \tilde{f}(j) \neq 0\}\right|}(f)(\tau).
$$
We now swap the first two sums and apply Remark \ref{silly remark} to obtain
$$ 
\sum_{\tilde{f} \in \mathcal{F} - \{0\}} (-1)^{\left|\{j : \tilde{f}(j) \neq 0\}\right| +1 } \chi_{\widetilde{f}, \epsilon_Y}(\sigma) \left(\sum_{f \in H_{\tilde{f}}} \psi_{m - \left\{j : \tilde{f}(j) \neq 0\}\right|}(f)(\tau)\right),
$$
which is the desired expression.
\end{proof}

We can further simplify the result of Proposition \ref{combinatorial formula quite in general} under the following more restrictive assumption. We say that a raw cocycle $\mathfrak{R}$ for $(Y, \epsilon_Y)$ is \emph{very promising} if it is promising and moreover 
$$
\sum_{f \in H_{\tilde{f}_1}} \psi_{m - \left|\{j : \tilde{f}_1(j) \neq 0\}\right|}(f)(\tau) = \sum_{f \in H_{\tilde{f}_2}} \psi_{m - \left|\{j : \tilde{f}_2(j) \neq 0\}\right|}(f)(\tau)
$$
whenever $\tilde{f}_1$ and $\tilde{f}_2$ share the same zero set. In this case for each subset $T \subseteq [m]$ we put 
$$
\chi_{T, \epsilon_Y} := \prod_{j \in T} j_l^{-1} \left(\prod_{p \in Y_j} \chi_p^{\epsilon_Y(p)}\right)
$$ 
and 
$$
\psi(\mathfrak{R}([m] - T))(\tau) := \sum_{f \in H_{\tilde{f}}} \psi_{m - \left|\{j : \tilde{f}(j) \neq 0\}\right|}(f)(\tau)
$$
for any $\tilde{f} \in \mathcal{F}$ such that the set of $j$ with $\tilde{f}(j) = 0$ is the set $T$. The inner product is a product of characters, while the outer product takes place in $\mathbb{F}_l$ yielding a continuous $1$-cochain from $G_{\mathbb{Q}}$ to $\mathbb{F}_l$. We have the following fact.

\begin{prop} 
\label{combinatorial formula in more special case}
Let $Y$, $\epsilon_Y$, $x_0$ be as above. Let $\mathfrak{R}$ be a raw cocycle for $(Y, \epsilon_Y)$ that is very promising for $x_0$. Then
$$
d_{x_0} \left(\sum_{x \in Y: x \neq x_0} \psi_m(\mathfrak{R}, \chi_{\epsilon_Y}(x))(\sigma,\tau)\right) = \sum_{\emptyset \neq T \subseteq [m]} (-1)^{\left|T\right| + 1} \chi_{T, \epsilon_Y}(\sigma) \psi(\mathfrak{R}([m] - T))(\tau).
$$
\end{prop}

\subsection{Expansion maps} 
\label{governing maps}
Let $m$, $Y$ be as in the previous subsection. Let $\epsilon_Y$ be an amalgama for $Y$. The next step is to construct a collection of $1$-cochains that have precisely the same recursive formula as in Proposition \ref{combinatorial formula in more special case}. This leads to the following definition. We call a \emph{pre-expansion} for $(Y, \epsilon_Y)$ a sequence parametrized by the proper subsets of $[m]$
$$
\{\phi_T(Y, \epsilon_Y)\}_{T \subsetneq [m]}
$$
consisting of continuous $1$-cochains from $G_{\mathbb{Q}}$ to $\mathbb{F}_l$ satisfying the following equation
$$
\left(d\phi_T(Y, \epsilon_Y)\right)(\sigma,\tau) = \sum_{\emptyset \neq U \subseteq T} (-1)^{\left|U\right| + 1} \chi_{U, \epsilon_Y}(\sigma) \phi_{T - U}(Y, \epsilon_Y)(\tau)
$$
for each $\sigma, \tau \in G_{\mathbb{Q}}$ and each proper subset $T$ of $[m]$. Here we consider $\mathbb{F}_l$ as a $G_{\mathbb{Q}}$-module with the trivial action. We will assume that $\phi_{\emptyset}$ is linearly independent from the space of characters spanned by the set $\left\{\chi_{\{i\}, \epsilon_Y}\right\}_{i \in [m]}$.

A pre-expansion is said to be \emph{promising} if for every $i \in [m]$, one has that every prime $p \in Y_i$ splits completely in the field of definition of $\phi_{[m] - \{i\}}$. Next we define an \emph{expansion} for $Y$ to be
$$
\{\phi_T(Y, \epsilon_Y)\}_{T \subsetneq [m]} \cup \{\phi_{[m]}(Y, \epsilon_Y)\},
$$
where $\{\phi_T(Y, \epsilon_Y)\}_{T \subsetneq [m]}$ is a pre-expansion for $(Y, \epsilon_Y)$ and $\phi_{[m]}(Y, \epsilon_Y)$ is a continuous $1$-cochain from $G_{\mathbb{Q}}$ to $\mathbb{F}_l$ satisfying
$$
\left(d\phi_{[m]}(Y, \epsilon_Y)\right)(\sigma,\tau) = \sum_{\emptyset \neq U \subseteq [m]} (-1)^{\left|U\right| + 1} \chi_{U, \epsilon_Y}(\sigma) \phi_{[m] - U}(Y, \epsilon_Y)(\tau).
$$ 
The maps composing a pre-expansion or an expansion are said to be \emph{good} if their field of definition is unramified above the maximal elementary abelian $\mathbb{F}_l$-extension, which is the field of definition of the map $\chi_{[m], \epsilon_Y} \cdot \phi_{\emptyset}(Y, \epsilon_Y)$. A pre-expansion or an expansion are said themselves to be good if all their maps are good. The pair $(Y, \epsilon_Y)$ is said to be \emph{cooperative} if for each distinct $i, j \in [m]$ we have that the character $\chi_{\{i\}, \epsilon_Y}$ is locally trivial at each prime appearing in $Y_j$ and at $l$.
 
\begin{prop} 
\label{existence of expansions}
Suppose that $(Y, \epsilon_Y)$ is cooperative. If $\{\phi_T(Y, \epsilon_Y)\}_{T \subsetneq [m]}$ is a promising good pre-expansion for $(Y, \epsilon_Y)$, then it can be completed to a good expansion for $(Y, \epsilon_Y)$.
\end{prop} 

\begin{proof}
Consider the map
$$
\theta : G_{\mathbb{Q}} \times G_{\mathbb{Q}} \to \mathbb{F}_l,
$$
defined by the formula
$$
\theta(\sigma,\tau) := \sum_{\emptyset \neq U \subseteq [m]} (-1)^{\left|U\right| + 1} \chi_{U, \epsilon_Y}(\sigma) \phi_{[m] - U}(Y, \epsilon_Y)(\tau).
$$
Simply from the assumption that $\{\phi_T(Y, \epsilon_Y)\}_{T \subsetneq [m]}$ is a pre-expansion, it follows that $\theta$ is a $2$-cocycle. We will also write $\theta$ for the resulting class in $H^2(G_\Q, \mathbb{F}_l)$. Observe that $\theta$ factors through 
$$
M := \prod_{T \subsetneq [m]} L(\phi_T(Y, \epsilon_Y))
$$ 
and hence defines an element of $H^2(\text{Gal}(M/\Q), \mathbb{F}_l)$. We next show that the $\mathbb{F}_l$-extension of the group $\text{Gal}(M/\mathbb{Q})$ given by the class of $\theta$ is actually in $\widetilde{\text{Cent}}_{\mathbb{F}_l}(M/\mathbb{Q})$ and that it can be realized by an \emph{unramified} $\mathbb{F}_l$-extension. We do so by applying Proposition \ref{no ramification} part $(1)$. 

Since the expansion is good, we only have to check the primes ramifying in the field of definition of $\chi_{[m], \epsilon_Y} \cdot \phi_{\emptyset}(Y, \epsilon_Y)$. Locally at these primes the expression defining $\theta$ becomes identically zero, because the pre-expansion is good and $(Y,\epsilon_Y)$ is cooperative. Hence, by Proposition \ref{no ramification}, we conclude that indeed $\theta \in \widetilde{\text{Cent}}_{\mathbb{F}_l}(M/\mathbb{Q})$ and that we can realize it as $r_1(\chi)$ for some $\chi \in \Gamma_{\mathbb{F}_l}(M)$ with $L(\chi)/M$ unramified. Thanks to Proposition \ref{criterion for Q} we also find that we can write
$$
\theta = d(\phi)
$$
for some $1$-cochain $\phi:G_{\mathbb{Q}} \to \mathbb{F}_l$. One can show with a direct calculation that $\phi$ restricted to $G_M$ is an element of $\Gamma_{\mathbb{F}_l}(M)^{\text{Gal}(M/\mathbb{Q})}$ with $r_1(\phi) = \theta$. After twisting with an element of $\Gamma_{\mathbb{F}_l}(\mathbb{Q})$ we can assume that $\phi$ restricts to $\chi$. Finally, a simple calculation (based on the formula of $\theta$) shows that the field of definition $L(\phi)$ is actually equal to $L(\chi)$. Setting $\phi_{[m]}(Y,\epsilon_Y) := \phi$ establishes the proposition.
\end{proof}

Our next goal is to determine structural information about the Galois group
$$
\text{Gal}(L(\phi_{[m]}(Y, \epsilon_Y))/\mathbb{Q}).
$$
To simplify the notation and to enhance generality, in the remainder of this section we assume to have a collection of characters
$$
\{\chi_{\{i\}}\}_{i \in [m]} \cup \{\chi_{\emptyset}\}
$$
living in $\Gamma_{\mathbb{F}_l}(\mathbb{Q})$ and forming altogether an $\mathbb{F}_l$-linearly independent set. For each non-empty subset $T \subseteq [m]$ we denote by $\chi_T$ the continuous $1$-cochain from $G_{\mathbb{Q}}$ to $\mathbb{F}_l$ defined by the formula
$$
\chi_T := \prod_{i \in T}\chi_{\{i\}},
$$
where the product is just multiplication in $\mathbb{F}_l$. We assume moreover to have a collection of maps $\{\phi_T\}_{T \subseteq [m]}$ satisfying the above equation of expansion maps
$$
(d\phi_T)(\sigma,\tau) = \sum_{\emptyset \neq U \subseteq T} (-1)^{\left|U\right| + 1} \chi_{U}(\sigma) \phi_{T - U}(\tau)
$$
for each $T \subseteq [m]$, with furthermore $\phi_{\emptyset}=\chi_{\emptyset}$. We proceed to determine the structure of
$$
\text{Gal}(L(\phi_{[m]})/\mathbb{Q}).
$$ 
For each non-empty $T \subseteq [m]$ we put $K_T := L(\chi_T)$, i.e. the $\mathbb{F}_l$-elementary extension obtained by adding all the characters $\chi_{\{i\}}$ with $i \in T$. We denote by
$$
\{\sigma_i\}_{i \in [m]}
$$
the dual basis of $\chi_{\{i\}}$ in $\text{Gal}(K_{[m]}/\mathbb{Q})$. Recall that $\text{Gal}(K_{[m]}/\mathbb{Q})$ acts on $\Gamma_{\mathbb{F}_l}(K_{[m]})$ by conjugation. Observe that one clearly has $\phi_T|_{G_{K_T}} \in \Gamma_{\mathbb{F}_l}(K_T)$. Take $i \in [m]$ and $T \subseteq [m]$. Let by abuse of notation $\sigma_i$ denote also any lift of $\sigma_i$ to $G_{\mathbb{Q}}$; the choice is relevant only to write symbolically meaningful formulas, but, since we are going to examine the effect on conjugation on a character, it will be irrelevant for the end result. Observe that for any $\tau \in G_{K_{[m]}}$ we have
\begin{align}
\label{ePhiConj1}
\phi_T(\sigma_i \tau \sigma_{i}^{-1}) &= \phi_T([\sigma_i, \tau] \tau) = \phi_T([\sigma_i, \tau]) + \phi_T(\tau) - (d\phi_T)([\sigma_i, \tau], \tau) \nonumber \\
&= \phi_T([\sigma_i, \tau]) + \phi_T(\tau).
\end{align}
Since
$$
\phi_T([\sigma_i, \tau]) + \phi_T(\tau \sigma_i) - \phi_T(\sigma_i \tau) = (d\phi_T)([\sigma_i, \tau], \tau \sigma_i) = 0,
$$
one finds
\begin{align}
\label{ePhiConj2}
\phi_T([\sigma_i, \tau]) = \phi_T(\sigma_i \tau) - \phi_T(\tau \sigma_i).
\end{align}
This can in turn be rewritten as
\begin{align}
\label{ePhiConj3}
\phi_T(\sigma_i \tau) - \phi_T(\tau \sigma_i) = (d\phi_T)(\tau, \sigma_i) - (d\phi_T)(\sigma_i, \tau) = -(d\phi_T)(\sigma_i, \tau),
\end{align}
where in the last identity we made use of the fact that $\tau \in G_{K_{[m]}}$. From equation (\ref{ePhiConj1}), (\ref{ePhiConj2}) and (\ref{ePhiConj3}) we find that
$$
\phi_T(\sigma_i\tau\sigma_i^{-1}) =
\left\{
	\begin{array}{ll}
		\phi_T(\tau) - \phi_{T - \{i\}}(\tau)  & \mbox{if } i \in T \\
		\phi_T(\tau) & \mbox{if } i \not \in T.
	\end{array}
\right.
$$
Therefore the action is given by the formula
$$
\sigma_i \cdot \phi_T =
\left\{
	\begin{array}{ll}
		\phi_T - \phi_{T - \{i\}}  & \mbox{if } i \in T \\
		\phi_T & \mbox{if } i \not \in T.
	\end{array}
\right.
$$
From this formula it follows immediately that $\text{Gal}(L(\phi_{[m]})/K_{[m]})$ is an $\mathbb{F}_l$-vector space whose dual is generated by all the maps $\phi_T$. Moreover, we deduce from the above formula that the natural action of $\mathbb{F}_l[\text{Gal}(K_{[m]}/\mathbb{Q})]$ factors through the ideal generated by $\{(\sigma_i-1)^2\}_{i \in [m]}$. The change of variables $t_i := \sigma_i-1$ shows that the group ring $\mathbb{F}_l[\text{Gal}(K_{[m]}/\mathbb{Q})]$ is isomorphic to the polynomial ring $\frac{\mathbb{F}_l[t_1, \ldots ,t_m]}{(t_1^l, \ldots ,t_m^l)}$. Hence we conclude that $\text{Gal}(L(\phi_{[m]})/K_{[m]})$ is a module over the ring
$$
R := \frac{\mathbb{F}_l[t_1, \ldots ,t_m]}{(t_1^2, \ldots ,t_m^2)}.
$$
We next prove that the dual $\text{Gal}(L(\phi_{[m]})/K_{[m]})^{\vee}$ is a free module of rank $1$ over $R$. It is clear that $\phi_{[m]}$ is a generator. Hence we only need to show that the annihilator ideal of $\phi_{[m]}$ is the zero ideal. Since $R$ is Gorenstein, we see that if the annihilator is not the zero ideal, then it must contain $t_1 \cdot \ldots \cdot t_m$. But we have, still thanks to the formula, that $t_1 \cdot \ldots \cdot t_m$ sends $\phi_{[m]}$ to $\chi_{\emptyset}$. Since the character $\chi_{\emptyset}$ is independent of the characters $\chi_{\{i\}}$, we obtain the desired conclusion.

Denote by $G$ the group
$$\frac{\mathbb{F}_l[t_1, \ldots ,t_m]}{(t_1^2, \ldots ,t_m^2)} \rtimes \text{Gal}(K_{[m]}/\mathbb{Q}),
$$
where the implicit action is the natural action of $\text{Gal}(K_{[m]}/\mathbb{Q})$ on $\frac{\mathbb{F}_l[t_1, \ldots ,t_m]}{(t_1^2, \ldots ,t_m^2)} $. 
With some little extra effort one can show that $\text{Gal}(L(\phi_{[m]})/\mathbb{Q})$ is actually isomorphic to $G$. 

We next examine the map
$$
\beta_{m + 1}(\phi_{[m]}) : G_{\mathbb{Q}}^{m+1} \to \mathbb{F}_l,
$$
that sends a vector $(\tau_1, \ldots, \tau_{m+1})$ to
$$
\phi_{[m]}([\tau_1, [\tau_2, [\ldots [\tau_m, \tau_{m+1}] \ldots]]]).
$$
Using the structure of the group $\text{Gal}(L(\phi_{[m]})/\mathbb{Q})$ one can establish quite easily the formula 
\begin{multline*}
\beta_{m + 1}(\phi_{[m]})(\tau_1, \ldots, \tau_{m + 1}) = \sum_{\rho \in \text{Sym}\{1, \ldots, m\}}\chi_{\emptyset}(\tau_{m + 1}) \prod_{1 \leq i \leq m} \chi_{\rho(i)}(\tau_i) - \\
\sum_{\rho \in \text{Sym}(i \in \{1, \ldots, m + 1\} - \{m\})} \chi_{\emptyset}(\tau_m) \prod_{1 \leq i \leq m + 1, i \neq m} \chi_{\rho(i)}(\tau_i).
\end{multline*} 
For an alternative way to arrive at the same formula one can use the identity
$$
\phi_{[m]}([\sigma, \tau]) = \phi_{[m]}(\sigma \tau) - \phi_{[m]}(\tau \sigma) = (d\phi_{[m]})(\tau, \sigma) - (d\phi_{[m]})(\sigma, \tau)
$$
and proceed by induction as explained in \cite[p.\ 12]{Smith}. The formula for $\beta_{m + 1}$ will be of utmost importance in Section \ref{sGov}, since it reduces the task of finding all relations among certain collections of expansion maps, called \emph{governing expansions}, to the task of finding all relations among (suitable functions of) the characters at the base of the expansion. 

\subsection{Creating unramified cocycles}
Let $Y$ and $\epsilon_Y$ be as in the previous two subsections and let $x_0 \in \text{Set}(Y)$. In this subsection we build on the previous two subsections to prove that under suitable assumptions $\psi_1(\mathfrak{R}, \chi_{\epsilon_Y}(x_0)) \in (1 - \zeta_l)^{m - 1}\text{Cl}(K_{\chi_{\epsilon_Y}(x_0)})^{\vee}[(1 - \zeta_l)^m]$. These assumptions come in two flavors, and we devote a subsection for each.

\subsubsection{Minimality}
Let $\mathfrak{R}$ be a raw cocycle on $(Y, \epsilon_Y)$ that is promising at $x_0$. Moreover, we assume that $Y$ is non-degenerate; there are no $Y_i$ containing all equal entries. We say that $\mathfrak{R}$ is \emph{minimal} with respect to $x_0$ if for any sub-cube $H$ of $Y$ not containing $x_0$ we have that
$$
\sum_{x \in H} \psi_{\text{dim}(H)}(\mathfrak{R}, \chi_{\epsilon_Y}(x)) = 0.
$$
We have the following fact.

\begin{prop} 
\label{obtaining x_0 with minim.}
Let $\mathfrak{R}$ be a promising minimal raw cocycle at $(Y, \epsilon_Y)$. Suppose that $\psi_1(\mathfrak{R}, \chi_{\epsilon_Y}(x))$ is constant as $x$ varies in $Y$. Then there exists $\chi' : G_\Q \rightarrow N[1 - \zeta_l]$ such that
$$
\psi := -\sum_{x \in Y: x \neq x_0} \psi_m(\mathfrak{R}, \chi_{\epsilon_Y}(x)) + \chi' \in \emph{Cocy}(\emph{Gal}(H_{\chi, l}/\mathbb{Q}), N(\chi_{\epsilon_Y}(x_0))).
$$
One has that $(1 - \zeta_l)^{m - 1} \psi = \left|\{x \in Y : x = x_0\}\right| \cdot \psi_1(\mathfrak{R}, \chi_{\epsilon_Y}(x_0))$ yielding in particular that
$$
\psi_1(\mathfrak{R}, \chi_{\epsilon_Y}(x_0)) \in (1 - \zeta_l)^{m - 1}\emph{Cl}(K_{\chi_{\epsilon_Y}(x_0)})^{\vee}[(1 - \zeta_l)^m].
$$
\end{prop}

\begin{proof}
We surely have that
$$
\tilde{\psi} := \sum_{x \in Y : x \neq x_0} \psi_m(\mathfrak{R}, \chi_{\epsilon_Y}(x)) \in \text{Cocy}(G_{\mathbb{Q}}, N(\chi_{\epsilon_Y}(x_0)))
$$ 
thanks to our minimality assumption and Proposition \ref{combinatorial formula quite in general}. Next observe that the proposition is trivially true if $m = 1$. So we can safely assume $m \geq 2$. We claim that 
$$
(1 - \zeta_l) \tilde{\psi} \in \text{Cocy}(\text{Gal}(H_{\chi, l}/\mathbb{Q}), N(\chi_{\epsilon_Y}(x_0))).
$$
Once we show this, the existence of $\chi'$ follows from Proposition \ref{liftable unramified cocycle can always be liftable unramified}. Indeed, if $l$ divides $x$ for some $x \in Y$, we have that $l$ divides $d$. In this case we see that $\text{Up}_{K_{\chi_{\epsilon_Y}(x_0)}}(l)$ splits completely in $L((1-\zeta_l)\tilde{\psi})K_{\chi_{\epsilon_Y}(x_0)}$, which places us in the position to use Proposition \ref{liftable unramified cocycle can always be liftable unramified}. We are now going to prove the claim. Define
\[
L := \prod_{x \in Y : x \neq x_0} L(\psi_m(\mathfrak{R}, \chi_{\epsilon_Y}(x))),
\]
where we recall that $L(\psi_m)$ is the field of definition of $\psi_m$. From the fact that $m \geq 2$ we see that $L$ contains $K_{\chi_{\epsilon_Y}(x_0)}$ and ramifies with inertia degree at most $l$ at each prime. Moreover, its ramification locus is contained in the set of primes appearing as coordinates of $Y$, since all the $\psi_m$ are unramified above their corresponding degree $l$ cyclic extension.

This already implies that the primes dividing $\Delta_{K_{\chi_{\epsilon_Y}(x_0)}/\mathbb{Q}}$ can not ramify in $L/K_{\chi_{\epsilon_Y}(x_0)}$, since the ramification is already eaten up by $K_{\chi_{\epsilon_Y}(x_0)}/\mathbb{Q}$. We are left with the other primes. For such a prime $q$, observe that thanks to the minimality assumption, we can always rewrite $(1 - \zeta_l)\tilde{\psi}$ as a sum only over $x$ where the prime $q$ is never used. Therefore the ramification locus of $L((1 - \zeta_l)\tilde{\psi})$ over $\mathbb{Q}$ does not contain any of those $q$ as well. Hence we conclude that $(1 - \zeta_l)\tilde{\psi}$ has indeed unramified field of definition above $K_{\chi_{\epsilon_Y}(x_0)}$. 

The claim gives $\chi'$ satisfying
$$
-\tilde{\psi} + \chi' = \psi \in \text{Cocy}(\text{Gal}(H_{\chi, l}/\mathbb{Q}), N(\chi_{\epsilon_Y}(x_0))).
$$ 
Observe that 
\begin{align*}
(1 - \zeta_l)^{m - 1} \psi &= \left(l^m - \left|\{x \in Y : x = x_0\}\right|\right) \cdot -\psi_1(\mathfrak{R}, \chi_{\epsilon_Y}(x_0)) \\
&= \left|\{x \in Y : x = x_0\}\right| \cdot \psi_1(\mathfrak{R}, \chi_{\epsilon_Y}(x_0)),
\end{align*}
where in the first equality we make use of the fact that all the $\psi_1$ are identical, and in the second we make use of the fact that $1 - \zeta_l$ kills each $\psi_1$ and thus the same is true for $l$. 
\end{proof}

\begin{remark}
If we have the stronger assumption that the sum over \emph{any} proper sub-cube is trivial for $\mathfrak{R}$, then we can directly conclude that the rank of $\mathfrak{R}$ at $x_0$ is at least $m$, since the expression in Proposition \ref{obtaining x_0 with minim.} would be a lift of $\psi_{m - 1}(\mathfrak{R}, \chi_{\epsilon_Y}(x_0))$. That said, in our application this will be irrelevant, since raw cocycles are merely a tool to access the pairing $\emph{Art}_{m}(\emph{Cl}(K_{\chi_{\epsilon_Y}(x_0)}))$.
\end{remark}

\subsubsection{Agreement} 
Let $i_a$ be in $[m]$. We will refer to $i_a$ as the \emph{index of agreement} in $[m]$. Moreover from now on we shall use the notation
$$Y-\{i_a\}:=\prod_{i \neq i_a}Y_{i}.
$$
Next we assume to have a good expansion $\{\phi_T(Y - \{i_a\}, \epsilon_Y)\}_{T \subseteq [m] - \{i_a\}}$ for $(Y - \{i_a\}, \epsilon_Y)$, where, by abuse of notation, $\epsilon_Y$ denotes also the restriction of $\epsilon_Y$ to $Y - \{i_a\}$. 

Let $\mathfrak{R}$ be a raw cocycle on $(Y, \epsilon_Y)$ that is very promising at $x_0$. Moreover, we assume that $Y$ is not degenerate, i.e. no $Y_i$ consists of all equal entries. Recall in this case that for $T \subsetneq [m]$ we introduced the notation $\psi(\mathfrak{R}(T))$ at the end of Section \ref{differential of the sum}. We say that $\mathfrak{R}$ \emph{agrees} with a good expansion $\{\phi_T(Y - \{i_a\},\epsilon_Y)\}_{T \subseteq [m] - \{i_a\}}$ if
\begin{itemize}
\item for each $T \subsetneq [m]$ containing $i_a$ one has 
$$
\psi(\mathfrak{R}(T)) = i_l \circ j_l \circ \phi_{T - \{i_a\}}(Y - \{i_a\}, \epsilon_Y);
$$
\item for each $T \subsetneq [m]$ not containing $i_a$ one has
$$
\psi(\mathfrak{R}(T)) = 0.
$$
\end{itemize}

\noindent We have the following fact.

\begin{prop} 
\label{obtaining x_0 under agreement}
Let $\mathfrak{R}$ be a promising raw cocycle at $(Y, \epsilon_Y)$. We assume that there exists a character $\chi$ such that
$$
\psi_1(\mathfrak{R}, \chi_{\epsilon_Y}(x)) = \chi + \chi_{\pi_{i_a}(x)}^{\epsilon_Y(\pi_{i_a}(x))}
$$
for all $x \in Y$. Let $\{\phi_T(Y - \{i_a\}, \epsilon_Y)\}_{T \subseteq [m] - \{i_a\}}$ be a good expansion for $(Y, \epsilon_Y)$. Suppose that $\mathfrak{R}$ agrees with this expansion. Then there exists $\chi' : G_\Q \rightarrow N[1 - \zeta_l]$ such that
$$ 
\psi := -\sum_{x \in Y : x \neq x_0} \psi_m(\mathfrak{R}, \chi_{\epsilon}(x)) + \chi' + i_l \circ j_l \circ \phi_{[m] - \{i_a\}}(Y - \{i_a\},\epsilon_Y)
$$ 
is in
$$ 
\emph{Cocy}(\emph{Gal}(H_{\chi, l}/\mathbb{Q}), N(\chi_{\epsilon_Y}(x_0))).
$$
One has that $(1 - \zeta_l)^{m - 1} \psi = \left|\{x \in Y : x = x_0\}\right| \cdot \psi_1(\mathfrak{R}, \chi_{\epsilon_Y}(x_0))$ yielding in particular that
$$
\psi_1(\mathfrak{R}, \chi_{\epsilon_Y}(x_0)) \in (1 - \zeta_l)^{m - 1} \emph{Cl}(K_{\chi_{\epsilon_Y}(x_0)})^{\vee}[(1 - \zeta_l)^m].
$$
\end{prop}

\begin{proof}
The proof is identical to the proof of Proposition \ref{obtaining x_0 with minim.}.
\end{proof}

\subsection{Sum of Artin pairings} 
\label{sum of artin pairings}
In this subsection we upgrade the two results of the previous subsection, showing that, under some additional assumptions, one can also control the sum of the Artin pairings over the cube. We keep the parallel with the previous section, dividing in two the discussion according to the two cases consisting of minimality and agreement. However, in both cases the crucial additional assumption is the following. 

\begin{mydef} 
We say that $Y$ is consistent if for all $i$ in $[m]$ and for each prime $q$ dividing $d$ we have that the characters $\chi_{q'}$ with $q' \in Y_i$ are all the same locally at $q$. 
\end{mydef}

\subsubsection{Minimality}
Let $\mathfrak{R}$ be a raw cocycle on $(Y, \epsilon_Y)$ that is promising at $x_0$.  Suppose also that $m \geq 2$ and that $Y$ is non-degenerate, i.e. there are no $Y_i$ consisting of all equal entries. We now show that under the assumption of Proposition \ref{obtaining x_0 with minim.}, and some additional assumptions, we can also obtain a relation among the $m$-th Artin pairings of the cube. 

\begin{theorem} 
\label{tMin}
Let $\mathfrak{R}$ be a minimal raw cocycle on $(Y, \epsilon_Y)$ that is promising at $x_0$ and is consistent. Suppose that $\psi_1(\mathfrak{R}, \chi_{\epsilon_Y}(x))$ is constant as $x$ varies in $Y$. 

Let $b$ be a positive integer whose prime divisors are all divisors of $d$. We assume that for all $x \in Y$, we have that $b$, viewed as an element of $\overline{\emph{Cl}}(K_{\chi_{\epsilon_Y}(x)})$, maps to an element of $(1 - \zeta_l)^{m - 1}\emph{Cl}(K_{\chi_{\epsilon_Y}(x)})[(1 - \zeta_l)^m]$. Then 
$$
\psi_1(\mathfrak{R}, \chi_{\epsilon_Y}(x_0)) \in (1 - \zeta_l)^{m - 1}\emph{Cl}(K_{\chi_{\epsilon_Y}})[(1 - \zeta_l)^m]
$$
and furthermore
$$
\sum_{x \in Y} \emph{Art}_{m}(\emph{Cl}(K_{\chi_{\epsilon_Y}(x)}))(b, \psi_1(\mathfrak{R}, \chi_{\epsilon_Y}(x))) = 0.
$$
\end{theorem} 

\begin{proof}
We know from Proposition \ref{obtaining x_0 with minim.} that indeed 
$$
\psi_1(\mathfrak{R}, \chi_{\epsilon_Y}(x_0)) \in (1 - \zeta_l)^{m - 1}\text{Cl}(K_{\chi_{\epsilon_Y}})[(1 - \zeta_l)^m],
$$
and moreover we can find an unramified $(1 - \zeta_l)^{m - 1}$-cocycle lift of the form
$$
\psi = -\frac{1}{\left|\{x \in Y : x = x_0\}\right|} \cdot \left(\sum_{x \in Y: x \neq x_0} \psi_m(\mathfrak{R}, \chi_{\epsilon_Y}(x)) + \chi'\right).
$$
It is clear from that proof that the character $\chi'$ can ramify at most at the primes occurring as coordinates of a point of $Y$. Indeed, in that proof, the cocycle $\tilde{\psi}$ has ramification over $\mathbb{Q}$ already contained only at most in such primes, since its field of definition is contained in the compositum of all the $\psi_m$, for which this last claim is evidently true by their defining property. Recall that in the proof of Proposition \ref{obtaining x_0 with minim.} the character $\chi'$ comes from applying Proposition \ref{liftable unramified cocycle can always be liftable unramified} and hence Proposition \ref{little ramification}, where one proceeds eliminating one by one all the eventual ramifying primes with a character supported precisely in that prime. This substantiates our claim on the shape of $\chi'$.

For each prime divisor $q$ of $b$, fix an embedding $G_{\mathbb{Q}_q} \subseteq G_{\mathbb{Q}}$ coming from a given fixed embedding $\overline{\mathbb{Q}} \to \overline{\mathbb{Q}_q}$. Thanks to our consistency assumption we see in particular the following crucial fact. For each prime divisor $q$ of $b$ and for each $x \in Y$ we have that
$$
\text{ker}(\chi_{\epsilon_Y}(x)) \cap G_{\mathbb{Q}_q}
$$
is constantly the same index $l$ subgroup. Call $K_q/\mathbb{Q}_q$ the corresponding field extension, totally ramified of degree $l$. Denote by $\tilde{K}_q/K_q$ the unique unramified extension of $K_q$ of degree $l$. Recall that this comes with the canonical generator given by $\text{Frob}_{K_q}$, the Frobenius automorphism. Observe that by definition of the Artin pairing we have that
\begin{multline*}
\sum_{x \in Y} \text{Art}_m(\text{Cl}(K_{\chi_{\epsilon_Y}(x)}))(b, \psi_1(\mathfrak{R}, \chi_{\epsilon_Y}(x))) = \\
\sum_{q \mid b} \psi\left(\text{Frob}_{K_q}^{\epsilon_b(q)}\right) + \sum_{q \mid b} \sum_{x \in Y - \{x_0\}} \psi_m(\mathfrak{R}, \chi_{\epsilon_Y}(x))\left(\text{Frob}_{K_q}^{\epsilon_b(q)}\right),
\end{multline*}
which is simply equal to
$$
\sum_{q \mid b} \chi'\left(\text{Frob}_{K_q}^{\epsilon_b(q)}\right)
$$
due to the definition of $\psi$. 

Next, we can decompose $\chi'$ as a product of not necessarily distinct characters having conductor a power of a prime that is a coordinate of $Y$ (the power will be precisely $1$ if the prime is different from $l$, and it will be $2$ if the prime is equal to $l$); this has been established above in this proof. If we can show that for such a $\chi_{q'}$ we have that
$$
\prod_{q \mid b} \chi_{q'}\left(\text{Frob}_{K_q}^{\epsilon_b(q)}\right) = 1,
$$
then we are clearly done. To see this pick a point $x$ with $K_{\chi_{\epsilon_Y}(x)}$ ramifying at $q'$. By definition of the Artin symbol, we have that
$$
\prod_{q \mid b} \chi_{q'}\left(\text{Frob}_{K_q}^{\epsilon_b(q)}\right) = \chi_{q'}\left(\text{Up}_{K_{\chi_{\epsilon_Y}(x)}}(b)\right) = 1,
$$
where the last equality follows directly from the assumption on $b$ and Proposition \ref{First description of Redei matrices}. 
\end{proof}

\subsubsection{Agreement}
Let $\mathfrak{R}$ be a raw cocycle on $(Y, \epsilon_Y)$ that is promising at $x_0$. Suppose also that $m \geq 2$. We also assume that $Y$ is non-degenerate, i.e. there are no $Y_i$ consisting of all equal entries. We now show that under the assumption of Proposition \ref{obtaining x_0 under agreement}, and some additional assumptions, we can also obtain a relation among the $m$-th Artin pairings of the cube. Define 
\[
M\left(\phi_{[m] - \{i_a\}}(Y - \{i_a\}, \epsilon_Y)\right) := \prod_{T \subsetneq [m] - \{i_a\}} L\left(\phi_T(Y - \{i_a\}, \epsilon_Y)\right).
\]

\begin{theorem}
\label{tAgree}
Let $\mathfrak{R}$ be a raw cocycle for $(Y, \epsilon_Y)$ that is very promising at $x_0$. We assume that there exists a character $\chi$ such that
$$
\psi_1(\mathfrak{R}, \chi_{\epsilon_Y}(x)) = \chi + \chi_{\pi_{i_a}(x)}^{\epsilon_Y(\pi_{i_a}(x))}
$$
for all $x \in Y$. Let $\{\phi_T(Y - \{i_a\}, \epsilon_Y)\}_{T \subseteq [m] - \{i_a\}}$ be a good expansion for $(Y, \epsilon_Y)$. Suppose that $\mathfrak{R}$ agrees with this expansion. Let $b$ be a positive integer whose prime divisor are all divisors of $d$. We assume that for all $x \in Y$, we have that $b$, viewed as an element of $\overline{\emph{Cl}}(K_{\chi_{\epsilon_Y}(x)})$, maps to an element of $(1 - \zeta_l)^{m - 1}\emph{Cl}(K_{\chi_{\epsilon_Y}(x)})[(1 - \zeta_l)^m]$. Then 
$$
\psi_1(\mathfrak{R}, \chi_{\epsilon_Y}(x_0)) \in (1 - \zeta_l)^{m - 1}\emph{Cl}(K_{\chi_{\epsilon_Y}})[(1 - \zeta_l)^m].
$$
Moreover, the Frobenius class of each prime $q$ dividing $d$ in $L(\phi_{[m] - \{i_a\}})$ consists of a central element in $\emph{Gal}(L(\phi_{[m] - \{i_a\}}(Y - \{i_a\},\epsilon_Y))/M(\phi_{[m] - \{i_a\}}(Y - \{i_a\}, \epsilon_Y)))$ and
\begin{multline*}
\sum_{x \in Y} \emph{Art}_{m}(\emph{Cl}(K_{\chi_{\epsilon_Y}(x)}))(b, \psi_1(\mathfrak{R}, \chi_{\epsilon_Y}(x))) = \\
\sum_{q \mid b}\epsilon_b(q) \left(i_l \circ j_l \circ \phi_{[m] - \{i_a\}}\right)(Y - \{i_a\}, \epsilon_Y)(\emph{Frob}_q). 
\end{multline*}
\end{theorem}

\begin{proof}
The proof is identical to the proof of Theorem \ref{tMin}.
\end{proof}

\section{Additive systems}
\label{sAdd}
Recall that $[d]$ denotes the set $\{1, \ldots, d\}$, where $d$ is any integer. In the previous subsections we dealt with just one cube $Y$. To prove our main theorems, we will use Theorem \ref{tMin} and Theorem \ref{tAgree} many times in various cubes $Y$. To facilitate our analysis, we need a flexible notation that can deal with different cubes at the same time. For this reason we now introduce the following notation depending on $l$, which is similar to the notation in Smith \cite[p.\ 8]{Smith}.

\begin{itemize}
\item $X$ will always denote a product set
\[
X = X_1 \times X_2 \times \ldots \times X_d
\]
with $X_i$ disjoint finite sets consisting of primes all equal to $0$ or $1$ modulo $l$.
\item For $S \subseteq [d]$, we define
\[
\overline{X}_S = \left(\prod_{i \in S} X_i^l\right) \times \prod_{i \in [d] - S} X_i.
\]
Furthermore, write $\pi_i$ for the projection map from $\overline{X}_S$ to $X_i^l$ if $i \in S$, and the projection map from $\overline{X}_S$ to $X_i$ if $i \in [d] - S$. If $\bar{x} \in \overline{X}_S$, we will sometimes call $\bar{x}$ a cube.
\item The natural projection maps from $X_i^l$ to $X_i$ are denoted by $\pr_1, \ldots, \pr_l$.
\item For $S, S_0 \subseteq [d]$, we let $\pi_{S, S_0}$ be the projection map from $\overline{X}_S$ to
\[
\left(\prod_{i \in S \cap S_0} X_i^l\right) \times \prod_{i \in ([d] - S) \cap S_0} X_i
\]
given by $\pi_i$ on each $i \in S_0$. When the set $S$ is clear from context, we will simply write $\pi_{S_0}$ instead of $\pi_{S, S_0}$.
\item Let $\bar{x} \in \overline{X}_S$, $T \subseteq S \subseteq [d]$ and put $U := S - T$. Then we define $\bar{x}(T)$ to be the multiset with underlying set given by
\[
\left\{\bar{y} \in \overline{X}_T: \pi_{[d] - U}(\bar{y}) = \pi_{[d] - U}(\bar{x}) \text{ and } \forall i \in U \ \exists j \in [l]: \pi_i(\bar{y}) = \pr_j(\pi_i(\bar{x}))\right\}.
\]
We define the multiplicity of $\bar{y}$ in $\bar{x}(T)$ to be
\[
\prod_{i \in U} \left|\left\{j \in [l]: \pi_i(\bar{y}) = \pr_j(\pi_i(\bar{x}))\right\}\right|.
\]
\end{itemize}

With this notation one element $\bar{x} \in \overline{X}_S$ corresponds to a cube $Y$ in the previous subsections. This is extremely convenient. For example, with this new notation we can rephrase Theorem \ref{tAgree} as
\[
\sum_{x \in \bar{z}(\emptyset)} F(x, b) = \sum_{q \mid b} \epsilon_b(q) \phi_{S, \bar{z}}(\text{Frob}_q),
\]
where $\bar{z} \in \overline{X}_S$ and $F(x, b)$ is some Artin pairing. Here we suppress the dependence on the amalgama. We will frequently suppress this dependence in the remainder of the paper, in particular for $\phi_{S, \bar{z}}$.

In the previous sections we defined expansion maps and raw cocycles. Both these objects are rather complicated to work with directly. Instead, we abstract their most important properties in the following combinatorial structure, which we call an $l$-additive system. The material in this section is an adaptation of Section 3 and Section 4 of Smith \cite{Smith}. 

\begin{mydef}
\label{dAdd}
An $l$-additive system $\mathfrak{A}$ on $X = X_1 \times \ldots \times X_d$ is a tuple
\[
\left(\overline{Y}_S, \overline{Y}_S^\circ, F_S, A_S\right)_{S \subseteq [d]}
\]
indexed by the subsets $S$ of $[d]$ with the following properties
\begin{itemize}
\item for all $S \subseteq [d]$, $A_S$ is a finite $\mathbb{F}_l$-vector space and $\overline{Y}_S^\circ  \subseteq \overline{Y}_S \subseteq \overline{X}_S$;
\item for all $S \subseteq [d]$ with $S \neq \emptyset$, we have
\[
\overline{Y}_S = \left\{\bar{x} \in \overline{X}_S: \bar{x}(T) \subseteq \overline{Y}_T^\circ \text{ for all } T \subsetneq S\right\};
\]
\item for all $S \subseteq [d]$, $F_S: \overline{Y}_S \rightarrow A_S$ is a function and
\[
\overline{Y}_S^\circ = \left\{\bar{x} \in \overline{Y}_S: F_S(\bar{x}) = 0\right\};
\]
\item (Additivity) let $s \in S$ and $\bar{x}_1, \ldots, \bar{x}_{l + 1} \in \overline{Y}_S$. Suppose that
\[
\pi_{[d] - \{s\}}(\bar{x}_1) = \ldots = \pi_{[d] - \{s\}}(\bar{x}_{l + 1})
\]
and suppose that there exist $p_1, \ldots, p_{l - 1}, q_1, \ldots, q_l \in X_s$ such that
\[
\pi_s(\bar{x}_1) = (p_1, \ldots, p_{l - 1}, q_1), \ldots, \pi_s(\bar{x}_l) = (p_1, \ldots, p_{l - 1}, q_l)
\]
and $\pi_s(\bar{x}_{l + 1}) = (q_1, \ldots, q_l)$. Then we have
\[
F_S(\bar{x}_1) + \ldots + F_S(\bar{x}_l) = F_S(\bar{x}_{l + 1}).
\]
\end{itemize}
We will sometimes write $\overline{Y}_S(\mathfrak{A})$, $\overline{Y}_S^\circ(\mathfrak{A})$, $F_S(\mathfrak{A})$ and $A_S(\mathfrak{A})$ for the data associated to an $l$-additive system $\mathfrak{A}$.
\end{mydef}

We remark that the condition $\bar{x}_{l + 1} \in \overline{Y}_S$ may be dropped, since it follows from the other conditions. The minimality and agreement conditions from Theorem \ref{tMin} and Theorem \ref{tAgree} can naturally be encoded in an $l$-additive system. Although we could already do this now, we postpone this task until Lemma \ref{lAddCon}. Similarly, the existence of the maps $\phi_{S, \bar{x}}$ can be encoded in an $l$-additive system. 

For our analytic techniques to work, it is essential that we can apply Theorem \ref{tMin} and Theorem \ref{tAgree} to many different cubes $\bar{x}$; the more cubes to which we can apply these theorems the better. It is for this reason that we need to give a lower bound for the density of $\overline{Y}_S^\circ$ in $\overline{X}_S$ for an $l$-additive system. Since $l$-additive systems are purely combinatorial objects, we will state our theorems for general finite sets, not just sets of primes.

\begin{prop}
\label{pAS}
Let $X = X_1 \times \ldots \times X_d$ be a product of finite sets and let $\mathfrak{A}$ be an $l$-additive system on $X$. Let $\delta$ be the density of $\overline{Y}_\emptyset^\circ$ in $X$ and put
\[
a := \max_{S \subseteq [d]} |A_S|.
\]
Then the density of $\overline{Y}_S^\circ$ in $\overline{X}_S$ is lower bounded by $\delta^{l^{|S|}} a^{-(l + 1)^{|S| + 1}}$ for all subsets $S$ of $[d]$.
\end{prop}

\begin{proof}
For $S = \emptyset$ this is clear, so from now on we assume that $S \neq \emptyset$. Fix a choice of $s \in S$ for the remainder of the proof. Define for $\bar{x}_0 \in \overline{X}_{S - \{s\}}$
\[
V(\bar{x}_0) := \left\{\bar{y} \in \overline{Y}_{S - \{s\}}^\circ: \pi_{[d] - \{s\}}(\bar{y}) = \pi_{[d] - \{s\}}(\bar{x}_0)\right\}
\]
and
\[
W(\bar{x}_0) := \left\{\bar{y} \in \overline{Y}_S^\circ: \pi_{[d] - \{s\}}(\bar{y}) = \pi_{[d] - \{s\}}(\bar{x}_0)\right\}.
\]
There are natural injective maps from $V(\bar{x}_0)$ to both $X_s$ and $\overline{X}_{S - \{s\}}$. The former map is given by sending $\bar{y}$ to $\pi_{\{s\}}(\bar{y})$, while the latter map is the inclusion $\overline{Y}_{S - \{s\}}^\circ \subseteq \overline{X}_{S - \{s\}}$. Similarly, there are natural injective maps from $W(\bar{x}_0)$ to $V(\bar{x}_0)^l \subseteq X_s^l$ and $\overline{X}_S$. We claim that
\begin{align}
\label{eLower}
|W(\bar{x}_0)| \geq \left(\frac{|V(\bar{x}_0)|}{a^{(l + 1)^{|S| - 1}}}\right)^l.
\end{align}
If $V(\bar{x}_0)$ is the empty set, (\ref{eLower}) clearly holds. So suppose that $V(\bar{x}_0)$ is not empty and choose $l - 1$ elements $\bar{x}_{1, 1}, \ldots, \bar{x}_{1, l - 1}$ from $V(\bar{x}_0)$. We define an equivalence relation $\sim_1$ on $V(\bar{x}_0)$ by declaring $\bar{y}_1 \sim_1 \bar{y}_2$ if and only if for all subsets $T$ satisfying $\{s\} \subseteq T \subseteq S$ and $|T| = 1$ and all $\bar{y}_1' \in \bar{y}_1(T - \{s\}), \bar{y}_2' \in \bar{y}_2(T - \{s\})$ satisfying $\pi_{[d] - T}(\bar{y}_1') = \pi_{[d] - T}(\bar{y}_2')$ we have
\begin{align}
\label{eFT}
F_T(\bar{x}_{1, 1}', \ldots, \bar{x}_{1, l - 1}', \bar{y}_1') = F_T(\bar{x}_{1, 1}', \ldots, \bar{x}_{1, l - 1}', \bar{y}_2'),
\end{align}
where $\bar{x}_{1, 1}', \ldots, \bar{x}_{1, l - 1}'$ are the unique elements of $\bar{x}_{1, 1}(T - \{s\}), \ldots, \bar{x}_{1, l - 1}(T - \{s\})$ satisfying
\[
\pi_{[d] - T}(\bar{y}_1') = \pi_{[d] - T}(\bar{y}_2') = \pi_{[d] - T}(\bar{x}_{1, 1}') =  \ldots =  \pi_{[d] - T}(\bar{x}_{1, l - 1}').
\]
Here we remark that the tuple $(\bar{x}_{1, 1}', \ldots, \bar{x}_{1, l - 1}', \bar{y}_i')$ can naturally be seen as an element of $\overline{X}_T$, so equation (\ref{eFT}) makes sense. There are at most $a^{l^{|S| - 1}}$ equivalence classes. Hence there exists an equivalence class $[\bar{y}]$ with at least
\[
\frac{|V(\bar{x}_0)|}{a^{l^{|S| - 1}}}
\]
elements. Now choose $\bar{x}_{2, 1}, \ldots, \bar{x}_{2, l - 1} \in [\bar{y}]$ and define a new equivalence relation $\sim_2$ on $[\bar{y}]$ by declaring $\bar{y}_1 \sim_2 \bar{y}_2$ if and only if for all subsets $T$ satisfying $\{s\} \subseteq T \subseteq S$ with $|T| = 2$ and all $\bar{y}_1' \in \bar{y}_1(T - \{s\}), \bar{y}_2' \in \bar{y}_2(T - \{s\})$ satisfying $\pi_{[d] - T}(\bar{y}_1') = \pi_{[d] - T}(\bar{y}_2')$ we have
\begin{align}
\label{eFT2}
F_T(\bar{x}_{2, 1}', \ldots, \bar{x}_{2, l - 1}', \bar{y}_1') = F_T(\bar{x}_{2, 1}', \ldots, \bar{x}_{2, l - 1}', \bar{y}_2'),
\end{align}
where $\bar{x}_{2, 1}', \ldots, \bar{x}_{2, l - 1}'$ are the unique elements of $\bar{x}_{2, 1}(T - \{s\}), \ldots, \bar{x}_{2, l - 1}(T - \{s\})$ satisfying
\[
\pi_{[d] - T}(\bar{y}_1') = \pi_{[d] - T}(\bar{y}_2') = \pi_{[d] - T}(\bar{x}_{2, 1}') = \ldots =  \pi_{[d] - T}(\bar{x}_{2, l - 1}').
\]
Since the domain of $F_T$ is $\overline{Y}_T$, equation (\ref{eFT2}) only makes sense if we have
\[
(\bar{x}_{2, 1}', \ldots, \bar{x}_{2, l - 1}', \bar{y}_i') \in \overline{Y}_T. 
\]
This follows from the construction of $\sim_1$ and additivity. 

We inductively proceed until we reach $\sim_{|S|}$. A computation shows that
\[
\prod_{\{s\} \subseteq T \subseteq S} |A_T|^{l^{|S| - |T|}} \leq \prod_{i = 0}^{|S| - 1} a^{\binom{|S| - 1}{i} l^i} = a^{(l + 1)^{|S| - 1}}.
\]
Then we find that there is an equivalence class of $\sim_{|S|}$ with at least
\[
\frac{|V(\bar{x}_0)|}{a^{(l + 1)^{|S| - 1}}}.
\]
elements. Suppose that $\{\bar{y}_1, \ldots, \bar{y}_k\}$ is an equivalence class of $\sim_{|S|}$. From additivity we obtain that $(\bar{y}_{i_1}, \ldots, \bar{y}_{i_l}) \in W(\bar{x}_0)$ for all choices of $1 \leq i_1, \ldots, i_l \leq k$, where we recall that $W(\bar{x}_0)$ can be identified as a subset of $V(\bar{x}_0)^l$. Hence we deduce
\[
|W(\bar{x}_0)| \geq \left(\frac{|V(\bar{x}_0)|}{a^{(l + 1)^{|S| - 1}}}\right)^l,
\]
establishing (\ref{eLower}). Define $\delta_T$ to be the density of $\overline{Y}_T^\circ$ in $\overline{X}_T$, so in particular $\delta = \delta_\emptyset$. Also let $\delta_{\bar{x}_0}$ to be the density of $V(\bar{x}_0)$ in $X_s$. Then the density of $V(\bar{x}_0)^l$ in $X_s^l$ is equal to $\delta_{\bar{x}_0}^l$. Since $\overline{Y}_S^\circ$ is the disjoint union of $W(\bar{x}_0)$ over all $\bar{x}_0 \in \pi_{[d] - \{s\}}(\overline{X}_{S - \{s\}})$, it follows from equation (\ref{eLower}) that
\begin{align*}
\delta_S &= \sum_{\bar{x}_0} \frac{\left|W(\bar{x}_0)\right|}{\left|\overline{X}_S\right|} \geq a^{-l \cdot (l + 1)^{|S| - 1}} \cdot \sum_{\bar{x}_0} \frac{\left|V(\bar{x}_0)\right|^l}{\left|\overline{X}_S\right|} = a^{-l \cdot (l + 1)^{|S| - 1}} \cdot \sum_{\bar{x}_0} \frac{\delta_{\bar{x}_0}^l}{\left|\pi_{[d] - \{s\}}\left(\overline{X}_{S - \{s\}}\right)\right|} \\
&\geq a^{-l \cdot (l + 1)^{|S| - 1}} \cdot \left(\sum_{\bar{x}_0} \frac{\delta_{\bar{x}_0}}{\left|\pi_{[d] - \{s\}}\left(\overline{X}_{S - \{s\}}\right)\right|}\right)^l.
\end{align*}
We observe that $\delta_{S - \{s\}}$ is the average of $\delta_{\bar{x}_0}$ over all $\bar{x}_0 \in \pi_{[d] - \{s\}}(\overline{X}_{S - \{s\}})$. This shows
\begin{align}
\label{eDone}
\delta_S \geq a^{-l \cdot (l + 1)^{|S| - 1}} \cdot \delta_{S - \{s\}}^l \geq a^{-(l + 1)^{|S|}} \cdot \delta_{S - \{s\}}^l.
\end{align}
Repeated application of (\ref{eDone}) yields the proposition.
\end{proof}

Proposition \ref{pAS} shows that there are many $\bar{x} \in \overline{Y}_S^\circ$. The proof of Proposition \ref{pAS} heavily relies on the special structure of $l$-additive systems. It will also be important to find $\bar{x}$ with $\bar{x}(\emptyset) \subseteq \overline{Y}_\emptyset^\circ$. Unlike $l$-additive systems, the set $\overline{Y}_\emptyset^\circ$ has very little structure. Instead we have to rely on Ramsey theory to find such $\bar{x}$.

\begin{prop}
\label{pRam}
Let $d$ be a positive integer and let $X_1, \ldots, X_d$ be finite sets all with cardinality at least $n > 0$. Let $Y$ be a subset of $X = X_1 \times \ldots \times X_d$ of density at least $\delta > 0$. Let $r$ be a positive integer satisfying
\[
r \leq n \cdot (2^{-d - 1} \delta)^{2r^{d - 1}}.
\]
Then there are subsets $Z_i \subseteq X_i$ all of cardinality $r$ such that
\[
Z_1 \times \ldots \times Z_d \subseteq Y.
\]
\end{prop}

\begin{proof}
This is proven in Proposition 4.1 of Smith \cite{Smith}.
\end{proof}

Before we move on, we explain our strategy for proving our main theorems. In Section \ref{sRel} we have seen that under suitable conditions on $\bar{x}$
\begin{align}
\label{eDifferential}
\sum_{x \in \bar{x}(\emptyset)} F(x) = g(\bar{x}),
\end{align}
where $F$ is a class group pairing and $g(\bar{x})$ is an Artin symbol in a relative governing field. If we could directly get a handle on $F$, we would be done, but this seems to be completely out of reach with the current methods available. And indeed, equation (\ref{eDifferential}) would be of little help in such a strategy.

Instead, we will take the following approach that uses (\ref{eDifferential}) in an essential way. First of all, observe that given $g$ there are many functions $F$ satisfying (\ref{eDifferential}). Our goal will be to find one function $g$ for which all $F$ satisfying (\ref{eDifferential}) are equidistributed. Obviously, such a conclusion is only possible if we know that (\ref{eDifferential}) holds for many $\bar{x}$, and this is where Proposition \ref{pAS} and Proposition \ref{pRam} are essential. Then we use the Chebotarev density theorem to make this function $g$ many times, which allows us to conclude equidistribution of $F$. Our next definition formalizes these ideas.

\begin{mydef}
\label{dDifferential}
Let $X_1, \ldots, X_d$ be finite non-empty sets, and put $X = X_1 \times \ldots \times X_d$. Let $S \subseteq [d]$ be a set with $|S| \geq 2$. For $Z \subseteq X$ we define $\FF_l$-vector spaces $V$ and $W$ by
\[
V := \left\{F: Z \rightarrow \FF_l\right\}, \quad W := \left\{g: \left\{\bar{x} \in \overline{X}_S: \bar{x}(\emptyset) \subseteq Z\right\} \rightarrow \FF_l\right\}.
\]
Let $d: V \rightarrow W$ be the linear map given by
\[
dF(\bar{x}) = \sum_{x \in \bar{x}(\emptyset)} F(x),
\]
where we remind the reader that $\bar{x}(\emptyset)$ is a multiset. Equivalently,
\[
dF(\bar{x}) = \sum_{x \in \textup{Set}(\bar{x}(\emptyset))} \left(\prod_{i \in S} \left| \{j \in [l]: \pr_j(\pi_i(\bar{x})) = \pi_i(x)\} \right|\right) F(x).
\]
For $\epsilon > 0$ a real number, we say that $F: Z \rightarrow \FF_l$ is $\epsilon$-balanced if for all $a \in \FF_l$
\[
\left(\frac{1}{l} - \epsilon\right) \cdot |Z| \leq |F^{-1}(a)| \leq \left(\frac{1}{l} + \epsilon\right) \cdot |Z|,
\]
and we say that $F$ is $\epsilon$-unbalanced otherwise. Define $\mathscr{G}_S(Z) := \text{im } d$ and 
\[
\mathscr{G}_S(\epsilon, Z) := \left\{g \in \mathscr{G}_S(Z): g = dF \text{ for some } \epsilon \text{-unbalanced F}\right\}.
\]
\end{mydef}

\begin{lemma}
\label{lGSZ}
Let $X$, $Z$, $S$ and $d$ be as in Definition \ref{dDifferential} such that $\left|\pi_{[d] - S}(Z)\right| = 1$. Further suppose that $\delta > 0$ satisfies
\[
|Z| \geq \delta \cdot |\pi_S(X)|.
\]
If $|X_i| \geq n$ for all $i \in S$, we have for all $\epsilon > 0$
\[
\frac{|\mathscr{G}_S(\epsilon, Z)|}{|\mathscr{G}_S(Z)|} \leq 2 \cdot l \cdot \exp\left(|\pi_S(X)| \cdot \left(-\delta \cdot \epsilon^2 + \log l \cdot 2^{|S| + 2} \cdot n^{-1/l^{|S|}} \right)\right).
\]
\end{lemma}

\begin{proof}
Recall that a cube $\bar{z} \in \overline{X}_S$ is called degenerate if there is $i \in S$ such that
\[
\left|\left\{\text{pr}_1\left(\pi_{\{i\}}(\bar{z})\right), \ldots, \text{pr}_l\left(\pi_{\{i\}}(\bar{z})\right)\right\}\right| = 1.
\]
Let $Z'$ be a maximal subset of $Z$ such that all cubes $\bar{z} \in \overline{X}_S$ with $\bar{z}(\emptyset) \subseteq Z'$ are degenerate. Let $F: Z \rightarrow \FF_l$ be a map with $F(x) \neq 0$ for some $x \in Z - Z'$ and $F(x) = 0$ for all $x \in Z'$. We claim that $F$ is not in the kernel of the linear map $d: V \rightarrow W$. Indeed, let us consider the set $Z' \cup \{x\}$. By construction of $Z'$, we find a non-degenerate $\bar{z} \in \overline{X}_S$ with $\bar{z}(\emptyset) \subseteq Z' \cup \{x\}$ and $x \in \bar{z}(\emptyset)$. Then it follows that $dF(\bar{z}) \neq 0$, establishing our claim.

From our claim we deduce that the kernel of $d$ is at most of size $l^{|Z'|}$. On the other hand, Proposition \ref{pRam} with $r = l$ yields
\[
|Z'| \leq |\pi_S(X)| \cdot 2^{|S| + 2} \cdot n^{-1/l^{|S|}}
\]
and hence
\begin{align}
\label{eGSZ}
\left|\mathscr{G}_S(Z)\right| \geq l^{|Z| - |Z'|} \geq l^{|Z|} \cdot \exp\left(-\log l \cdot |\pi_S(X)| \cdot 2^{|S| + 2} \cdot n^{-1/l^{|S|}}\right).
\end{align}
From Hoeffding's inequality and a straightforward union bound we obtain that the number of $\epsilon$-unbalanced $F$ is bounded by
\[
2 \cdot l^{|Z| + 1} \cdot \exp\left(-2 \cdot \epsilon^2 \cdot |Z|\right).
\]
We conclude that
\begin{align}
\label{eGSZe}
\left|\mathscr{G}_S(\epsilon, Z)\right| \leq 2 \cdot l^{|Z| + 1} \cdot \exp\left(-2 \cdot \epsilon^2 \cdot |Z|\right) 
\leq 2 \cdot l^{|Z| + 1} \cdot \exp\left(-|\pi_S(X)| \cdot \delta \cdot \epsilon^2\right).
\end{align}
The lemma follows upon combining (\ref{eGSZ}) and (\ref{eGSZe}).
\end{proof}

Lemma \ref{lGSZ} is very much in the spirit of the strategy we outlined earlier. Unfortunately, we do not have equality (\ref{eDifferential}) for all $\bar{x} \in \overline{X}_S$ with $\bar{x}(\emptyset) \subseteq \overline{Y}_\emptyset^\circ$. Instead, we will show in Lemma \ref{lAddCon} that equation (\ref{eDifferential}) holds under the much stronger condition that $\bar{x}(T) \cap \overline{Y}_T^\circ(\mathfrak{A})$ is ``large''  for all proper subsets $T$ of $S$, where $\mathfrak{A}$ is a completely explicit $l$-additive system. 

Fortunately, it turns out that $\overline{Y}_\emptyset^\circ$ has some special structure in our application. Namely, in Lemma \ref{lAddCon} we will prove that for ``sufficiently nice'' $\bar{x} \in \overline{X}_S$ we have $\bar{x}(\emptyset) \subseteq \overline{Y}_\emptyset^\circ$. Our next definition formalizes what we mean by ``sufficiently nice'' $\bar{x} \in \overline{X}_S$.

\begin{mydef}
\label{dAcceptable}
For an $l$-additive system $\mathfrak{A}$ on $X$ define
\[
\overline{Z}_S(\mathfrak{A}) := \bigcap_{i \in S} \left\{\bar{x} \in \overline{X}_S: \left|\pi_i\left(\bar{x}(S - \{i\}) \cap \overline{Y}_{S - \{i\}}^\circ(\mathfrak{A})\right)\right|\geq \max\left(1, \left|\pi_i(\bar{x}(S - \{i\}))\right| - 1\right)\right\},
\]
where $\pi_i$ of a multiset is defined to be $\pi_i$ of the underlying set. Let $a \geq 2$ be an integer. Call an $l$-additive system $\mathfrak{A}$ on $X$ $S$-acceptable if the following conditions are satisfied
\begin{itemize}
\item $|A_T(\mathfrak{A})| \leq a$ for all subsets $T$ of $S$;
\item if $\bar{x}$ is in $\overline{Z}_S(\mathfrak{A})$, we have $\textup{Set}(\bar{x}(\emptyset)) \subseteq \overline{Y}_\emptyset^\circ(\mathfrak{A})$.
\end{itemize}
\end{mydef}

Before we state the next proposition, we explain why we will run over all $l$-additive systems $\mathfrak{A}$ on $X$ in this proposition instead of just the special $l$-additive system from Lemma \ref{lAddCon}. To prove our equidistribution statements in the final section, we consider a large interval of primes. Using Chebotarev we split this interval in many sets $A_1, \ldots, A_k$.

Then we apply our next proposition to every $A_i$ with $\mathfrak{A}$ equal to the $l$-additive system from Lemma \ref{lAddCon} restricted to $A_i$. Since we have no control over the restriction of this $l$-additive system to a smaller subset, we simply run over all $l$-additive systems provided that $\overline{Y}_\emptyset^\circ$ has the special property in Definition \ref{dAcceptable}.

\begin{prop}
\label{p4.4}
There exists an absolute constant $A > 0$ such that the following holds. Let $X$ and $S$ be as in Definition \ref{dDifferential}. Let $a \geq 2$, $\epsilon > 0$ and define $n := \min_{i \in S} X_i$. Suppose that $\left|\pi_{[d] - S}(X)\right| = 1$, $\epsilon < a^{-1}$ and
\[
\log n \geq A \cdot (l \cdot (l + 1))^{|S| + 3} \cdot \log \epsilon^{-1}.
\]
Then there exists $g \in \mathscr{G}_S(X)$ such that for all $S$-acceptable $l$-additive systems $\mathfrak{A}$ at $S$ on $X$ and for all $F: \overline{Y}_\emptyset^\circ(\mathfrak{A}) \rightarrow \FF_l$ satisfying
\begin{align}
\label{edFg}
dF(\bar{x}) = g(\bar{x})
\end{align}
for all $\bar{x} \in \overline{Z}_S(\mathfrak{A})$, we have that $F$ is $\frac{|X|}{|\overline{Y}_\emptyset^\circ(\mathfrak{A})|} \cdot \epsilon$-balanced. In case $|\overline{Y}_\emptyset^\circ(\mathfrak{A})| = 0$, this is to be interpreted as $\infty$-balanced.
\end{prop}

\begin{proof}
Let $\mathscr{G}_S(\epsilon, a, X)$ be the set consisting of those $g \in \mathscr{G}_S(X)$ that fail Proposition \ref{p4.4}. We claim that
\begin{align}
\label{eOS}
\frac{\left|\mathscr{G}_S(\epsilon, a, X)\right|}{\left|\mathscr{G}_S(X)\right|} \leq \exp\left(-|X| \cdot \epsilon^{6 + (l + 1)^{|S| + 3}}\right),
\end{align}
which immediately yields the proposition. Define
\[
\delta := \frac{\left|\overline{Y}_\emptyset^\circ(\mathfrak{A})\right|}{|X|}.
\]
Let $g \in \mathscr{G}_S(\epsilon, a, X)$. Then, from the definition of $\mathscr{G}_S(\epsilon, a, X)$, there exists an $S$-acceptable $l$-additive system $\mathfrak{A}$ on $X$ and a $\delta^{-1} \epsilon$-unbalanced $F: \overline{Y}_\emptyset^\circ(\mathfrak{A}) \rightarrow \FF_l$ satisfying
\[
dF(\bar{x}) = g(\bar{x})
\]
for all $\bar{x} \in \overline{Z}_S(\mathfrak{A})$. For $f: [l - 1] \rightarrow \overline{Y}_\emptyset^\circ(\mathfrak{A})$ and $x \in \overline{Y}_\emptyset^\circ(\mathfrak{A})$ we let $c(f, x)$ be the unique element of $\overline{X}_S$ satisfying
\[
\pi_i(f(j)) = \pr_j(\pi_i(c(f, x))) \text{  and  } \pi_i(x) = \pr_l(\pi_i(c(f, x)))
\]
for $i \in S$ and $j \in [l - 1]$. Next define
\begin{align*}
Z_S(\mathfrak{A}, f) := \{x \in X : &\text{ writing } \bar{x} := c(f, x), \text{ we have } \bar{y} \in \overline{Y}_T^\circ(\mathfrak{A}) \text{ for all } T \subsetneq S\\
& \text{ and all } \bar{y} \in \bar{x}(T) \text{ satisfying } f([l - 1]) \cap \bar{y}(\emptyset) \neq \emptyset\}.
\end{align*}
Note that $x \in Z_S(\mathfrak{A}, f)$ implies $x \in \overline{Y}_\emptyset^\circ(\mathfrak{A})$. There is a natural injective map from $\overline{Y}_S^\circ(\mathfrak{A})$ to 
\[
\coprod_{f: [l - 1] \rightarrow \overline{Y}_\emptyset^\circ(\mathfrak{A})} Z_S(\mathfrak{A}, f), 
\]
where $\coprod$ denotes a disjoint union. This map is given by sending $\bar{y}$ to the pair $(f, x)$, where $f: [l - 1] \rightarrow \overline{Y}_\emptyset^\circ(\mathfrak{A})$ and $x \in Z_S(\mathfrak{A}, f)$ are uniquely determined by
\[
\pi_i(f(j)) = \pr_j(\pi_i(\bar{y})) \text{  and  } \pi_i(x) = \pr_l(\pi_i(\bar{y}))
\]
for $i \in S$ and $j \in [l - 1]$. We conclude that
\begin{align}
\label{eYSUu}
\left|\overline{Y}_S^\circ(\mathfrak{A})\right| \leq \sum_{f: [l - 1] \rightarrow \overline{Y}_\emptyset^\circ(\mathfrak{A})} |Z_S(\mathfrak{A}, f)| \leq |X|^{l - 1} \cdot \max_{f: [l - 1] \rightarrow \overline{Y}_\emptyset^\circ(\mathfrak{A})} |Z_S(\mathfrak{A}, f)|.
\end{align}
Since $F$ is $\delta^{-1} \epsilon$-unbalanced, it follows that $\frac{\epsilon}{2} \leq \delta$. Hence the density of $\overline{Y}_\emptyset^\circ(\mathfrak{A})$ in $X$ is at least $\frac{\epsilon}{2}$. From Proposition \ref{pAS} we see that the density of $\overline{Y}_S^\circ(\mathfrak{A})$ in $\overline{X}_S$ is lower bounded by
\begin{align}
\label{eYSUl}
\delta^{l^{|S|}} a^{-(l + 1)^{|S| + 1}} \geq \left(\frac{\epsilon}{2}\right)^{l^{|S|}} \epsilon^{(l + 1)^{|S| + 1}} \geq \epsilon^{(l + 1)^{|S| + 3}}.
\end{align}
Upon combining (\ref{eYSUu}) and (\ref{eYSUl}) we find that there exists $f_1: [l - 1] \rightarrow \overline{Y}_\emptyset^\circ(\mathfrak{A})$ such that $Z_S(\mathfrak{A}, f_1)$ has density at least $\epsilon^{(l + 1)^{|S| + 3}}$ in $X$. If the complement $\overline{Y}_\emptyset^\circ(\mathfrak{A}) - Z_S(\mathfrak{A}, f_1)$ has density at least $\frac{\epsilon}{2}$ in $X$, we can repeat this argument with the $S$-acceptable $l$-additive system $\mathfrak{A}'$ on $X$ given by 
\[
\overline{Y}_\emptyset^\circ(\mathfrak{A}') := \overline{Y}_\emptyset^\circ(\mathfrak{A}) - Z_S(\mathfrak{A},
f_1)
\]
and the same maps $F_T$ and groups $A_T$ as $\mathfrak{A}$. Hence we find a sequence of functions $f_1, \ldots, f_r: [l - 1] \rightarrow \overline{Y}_\emptyset^\circ(\mathfrak{A})$ such that
\[
Z_S(\mathfrak{A}, f_j) - Z_S(\mathfrak{A}, f_{j - 1}) - \ldots - Z_S(\mathfrak{A}, f_1)
\]
has density at least $\epsilon^{(l + 1)^{|S| + 3}}$ in $X$ for $j \geq 1$ and so that
\[
\overline{Y}_\emptyset^\circ(\mathfrak{A}) - Z_S(\mathfrak{A}, f_r) - \ldots - Z_S(\mathfrak{A}, f_1)
\]
has density at most $\frac{\epsilon}{2}$ in $X$. Define
\[
Z'_S(\mathfrak{A}, j) := Z_S(\mathfrak{A}, f_j) - Z_S(\mathfrak{A}, f_{j - 1}) - \ldots - Z_S(\mathfrak{A}, f_1).
\]
Then there exists a $j$ so that $F$ is $\frac{\epsilon}{2}$-unbalanced when restricted to $Z'_S(\mathfrak{A}, j)$. If $\bar{x} \in \overline{X}_S$ satisfies $\text{Set}(\bar{x}(\emptyset)) \subseteq Z'_S(\mathfrak{A}, j)$, (\ref{edFg}) combined with the additivity of $dF$ and $g$ imply $dF(\bar{x}) = g(\bar{x})$. From this we deduce that $g \in \mathscr{G}_S(\epsilon, a, X)$ implies
\[
g|_{\left\{\bar{x} \in \overline{X}_S : \text{Set}(\bar{x}(\emptyset)) \subseteq Z'_S(\mathfrak{A}, j)\right\}} \in \mathscr{G}_S\left(\frac{\epsilon}{2}, Z'_S(\mathfrak{A}, j)\right)
\]
for some $S$-acceptable $l$-additive system $\mathfrak{A}$ on $X$ and some $j$. We get from Lemma \ref{lGSZ} that the number of $g \in \mathscr{G}_S(X)$ with $g|_{\left\{\bar{x} \in \overline{X}_S : \text{Set}(\bar{x}(\emptyset)) \subseteq Z'_S(\mathfrak{A}, j)\right\}} \in \mathscr{G}_S\left(\frac{\epsilon}{2}, Z'_S(\mathfrak{A}, j)\right)$ is bounded by
\begin{align}
\label{eGSZA}
2 \cdot l \cdot \left|\mathscr{G}_S(X)\right| \cdot \exp\left(|X| \cdot \left(-\epsilon^{4 + (l + 1)^{|S| + 3}} + \log l \cdot 2^{|S| + 2} \cdot n^{-1/l^{|S|}} \right)\right).
\end{align}
For $A$ sufficiently large we can simplify (\ref{eGSZA}) as
\begin{align}
\label{eGSZA2}
\left|\mathscr{G}_S(X)\right| \cdot \exp\left(-|X| \cdot \epsilon^{5 + (l + 1)^{|S| + 3}}\right).
\end{align}
Let us now give an upper bound for the number of subsets $E$ of $X$ such that there exists an $S$-acceptable $l$-additive system $\mathfrak{A}$ on $X$ and $f: [l - 1] \rightarrow \overline{Y}_\emptyset^\circ(\mathfrak{A})$ satisfying $E = Z_S(\mathfrak{A}, f)$. A straightforward computation shows the equality
\[
Z_S(\mathfrak{A}, f) = \left\{x \in X: \pi_{S - \{i\}}(x) \in \pi_{S - \{i\}}(Z_S(\mathfrak{A}, f)) \text{ for all } i \in S\right\}.
\]
Therefore, $Z_S(\mathfrak{A}, f)$ is determined by the sets $\pi_{S - \{i\}}(Z_S(\mathfrak{A}, f))$ as $i$ varies through $S$. From this, we obtain the following upper bound for the number of possible sets $E$
\[
2^{|X| \cdot \sum_{i \in S} \frac{1}{|X_i|}} \leq 2^{|X| \cdot |S| \cdot n^{-1}}.
\]
Hence there are at most
\[
2^{r \cdot |X| \cdot |S| \cdot n^{-1}}
\]
sequences $Z'_S(\mathfrak{A}, f_1), \ldots, Z'_S(\mathfrak{A}, f_r)$ and at most $r$ choices of $j$. Multiplying this with the bound from (\ref{eGSZA2}) we conclude that
\begin{align}
\label{eGSeaX}
|\mathscr{G}_S(\epsilon, a, X)| \leq r \cdot 2^{r \cdot |X| \cdot |S| \cdot n^{-1}} \cdot \left|\mathscr{G}_S(X)\right| \cdot \exp\left(-|X| \cdot \epsilon^{5 + (l + 1)^{|S| + 3}}\right).
\end{align}
Using $r \leq \epsilon^{-(l + 1)^{|S| + 3}}$ and (\ref{eGSeaX}) we infer
\begin{align*}
|\mathscr{G}_S(\epsilon, a, X)| &\leq \left|\mathscr{G}_S(X)\right| \cdot \exp\left(|X| \cdot \left(-\epsilon^{5 + (l + 1)^{|S| + 3}} + r \cdot |S| \cdot n^{-1}\right)\right) \\
&\leq \left|\mathscr{G}_S(X)\right| \cdot \exp\left(-|X| \cdot \epsilon^{6 + (l + 1)^{|S| + 3}}\right).
\end{align*}
for $A$ sufficiently large. This establishes (\ref{eOS}), completing our proof.
\end{proof}

\section{Governing expansions} 
\label{sGov}
We will heavily make use of the notation introduced at the beginning of Section \ref{sAdd}; recall that this notation implicitly depends on $l$. Let $d$ be a positive integer and let $X_1, \ldots, X_d$ be disjoint sets of primes $q$ that are either $1$ modulo $l$ or equal to $l$. Fix a subset $\overline{Y}_{\emptyset}$ and a function $f : X_1 \coprod \ldots \coprod X_r \rightarrow [l - 1]$, where we remind the reader that $\coprod$ denotes disjoint union. 

If $\bar{x} \in \overline{X}_S$, we obtain an amalgama $\epsilon$ for the cube $\bar{x}$ by restricting $f$ to $\bar{x}$. Fix furthermore an integer $i_a \in [d]$. The coming definition will depend implicitly on the choice of $f$, and we shall suppress this dependence in the notation. A collection of sets
$$
\{\overline{Y}_S\}_{i_a \in S\subseteq d}
$$
with $\overline{Y}_S \subseteq \overline{X}_S$, together with a collection of continuous $1$-cochains
$$
\{\phi_{S, \bar{x}} : G_{\mathbb{Q}} \to \mathbb{F}_l\}_{\bar{x} \in \overline{Y}_S}
$$
for each $S$ containing $i_a$, is said to be a governing expansion $\mathfrak{G}$ if it satisfies the following requirements

\begin{enumerate}
\item[(1)] for each $\bar{x} \in \overline{Y}_{i_a}$ we have 
$$
\phi_{\{i_a\}, \bar{x}} = j_l^{-1} \circ \sum_{i = 1}^l \chi_{\text{pr}_i(\pi_{i_a}(\bar{x}))}^{f(\text{pr}_i(\pi_{i_a}(\bar{x})))};
$$
\item[(2)] if $\bar{x}_1, \bar{x}_2 \in \overline{Y}_S$ are cubes with the same multisets $X_i$ for each $i \in S$, we have $\phi_{S, \bar{x}_1} = \phi_{S, \bar{x}_2}$;
\item[(3)] if $i_a \in S$, then we have for all $\bar{x} \in \overline{Y}_S$ and all subsets $T$ satisfying $i_a \in T \subseteq S$ that $\bar{x}_T \in \overline{Y}_{T}$ for any choice of $\bar{x}_T \in \bar{x}(T)$. Moreover, the collection $\{\phi_{T, \bar{x}_T}\}_{i_a \in T}$ is a good expansion for all choices of $\bar{x}_T \in \bar{x}(T)$ (here the collection of subsets of $[m]$ not containing $i_a$ is naturally identified with the collection of subsets of $[m]$ containing $i_a$);
\item[(4)] we choose for each rational prime $q$ that ramifies in $\prod_{\bar{x}}L(\phi_{S, \bar{x}})/\mathbb{Q}$ a generator of an inertia subgroup $\sigma_q$. We assume that the $\phi_{S, \bar{x}}$ are such that $\phi_{S, \bar{x}}(\sigma_q) = 0$ for each $\bar{x}$;
\item[(5)] if $\bar{x} \in \overline{X}_S$, then $\bar{x} \in \overline{Y}_S$ if and only if the following two conditions are satisfied
\begin{itemize}
\item we have for all subsets $T$ satisfying $i_a \in T \subseteq S$ and all $\bar{x}_T \in \bar{x}(T)$ that $\bar{x} \in \overline{Y}_T$;
\item we have for each $i \in S$ that $\text{pr}_1(\pi_i(\bar{x})), \ldots, \text{pr}_l(\pi_i(\bar{x}))$ split completely in the field $L\left(\phi_{S - \{i\}, \bar{x}_{S - \{i\}}}\right)$.
\end{itemize} 
\end{enumerate}

These requirements are very similar to those imposed in \cite[p.\ 10-12]{Smith}, but not completely the same. Using the calculations done in Section \ref{governing maps} and following the same proof strategy explained in \cite[Proposition $2.3$]{Smith} one obtains the following important fact. 

\begin{prop} 
\label{gov.exp. are additive}
If $\mathfrak{G}$ is a governing expansion, then the assignment
$$
\bar{x} \mapsto \phi_{S, \bar{x}}(\mathfrak{G}),
$$
is additive for each $S$, see Definition \ref{dAdd}.
\end{prop}

Fix a set $S$ satisfying $i_a \in S \subseteq [d]$. Denote by $\mathcal{A}(\overline{Y}_S, \mathbb{F}_l)$ the $\mathbb{F}_l$-vector space of all additive maps from $\overline{Y}_S$ to $\mathbb{F}_l$. Following the same proof strategy of \cite[Proposition $2.4$]{Smith} one obtains the following proposition.

\begin{prop} 
\label{pGovAdd}
Let $\mathfrak{G}$ be a governing expansion. The assignment
$$
\sigma \mapsto (\bar{x} \mapsto \phi_{S, \bar{x}}\left(\mathfrak{G})(\sigma)\right)_{\bar{x} \in \overline{Y}_S(\mathfrak{G})}
$$
gives an isomorphism between the group 
$$
\left.\emph{Gal}\left(\prod_{\bar{x} \in \overline{Y}_S(\mathfrak{G})} L(\phi_{S, \bar{x}}(\mathfrak{G})) \right/ \prod_{i_a \in T \subsetneq S} \prod_{\bar{x} \in \overline{Y}_T(\mathfrak{G})} L(\phi_{T, \bar{x}}(\mathfrak{G}))\right)
$$ 
and the space of additive maps $\mathcal{A}(\overline{Y}_S(\mathfrak{G}), \mathbb{F}_l)$.
\end{prop}

It will be important in our main application to recognize that for the product space
$$
X = X_1 \times \ldots \times X_d
$$ 
the space of additive maps $\mathcal{A}(\overline{X}_S, \mathbb{F}_l)$ is equal to $\mathscr{G}_S(X)$ as defined in Definition \ref{dDifferential}.

\begin{prop}
\label{pDiff}
The image of the map $d : \emph{Map}(X,\mathbb{F}_l) \to \emph{Map}(\overline{X}_S, \mathbb{F}_l)$ is equal to $\mathcal{A}(\overline{X}_S, \mathbb{F}_l)$. Furthermore, the dimension of $\mathcal{A}(\overline{X}_S, \mathbb{F}_l)$ is equal to
$$
\prod_{i \in S} \left(|X_i| - 1\right) \cdot \prod_{j \in [d] - S} |X_j|.
$$
\end{prop}

\begin{proof}
It is a triviality that
\begin{align}
\label{eIncGSX}
\mathscr{G}_S(X) \subseteq \mathcal{A}(\overline{X}_S, \mathbb{F}_l).
\end{align}
We will now establish that $\mathcal{A}(\overline{X}_S, \mathbb{F}_l) \subseteq \mathscr{G}_S(X)$. To do so we pick once and for all a point $x_0 \in X$. We define the following subset of $X$
$$
\text{Max}(x_0) := \{x \in X : \exists i \in S \text{ with } \pi_i(x) = \pi_i(x_0)\}.
$$
We observe that $\text{Max}(x_0)$ does not contain any product sets of the form
$$
\prod_{i \in S} Y_i \times \prod_{j \in [d] - S} \{y_j\},
$$
where each $Y_i$ has precisely two elements. We claim that $\text{Max}(x_0)$ is a maximal subset of $X$ with the above property. Indeed, take any $y \in X - \text{Max}(x_0)$. Then we have $\pi_i(y) \neq \pi_i(x_0)$ for each $i$ in $S$. Now define
$$
Y(y) := \prod_{i \in S}\{\pi_i(y),\pi_i(x_0)\} \times \prod_{j \in [d] - S}\{\pi_j(y)\}, 
$$
which is clearly contained in $\text{Max}(x_0) \cup \{y\}$. This shows the claim. 

Next observe that the size of the complement of $\text{Max}(x_0)$ is trivially 
$$
\prod_{i \in S} \left(|X_i| - 1\right) \cdot \prod_{j \in [d] - S} |X_j|.
$$
Hence we obtain
$$
|\text{Max}(x_0)| = |X| - \prod_{i \in S} \left(|X_i| - 1\right) \cdot \prod_{j \in [d] - S} |X_j|.
$$
Following the proof of Proposition \ref{lGSZ}, we find that the kernel of $d$ has dimension at most $|\text{Max}(x_0)|$. This gives
\begin{align}
\label{eIneqGSX}
\text{dim}_{\mathbb{F}_l} \mathscr{G}_S(X) \geq \prod_{i \in S} \left(|X_i| - 1\right) \cdot \prod_{j \in [d] - S} |X_j|.
\end{align}
Finally consider the set
$$
\text{Min}(x_0) := \left\{\bar{x} \in \overline{X}_S: \pi_S(\bar{x}(\emptyset)) \text{ contains } \pi_S(x_0) \text{ with multiplicity } (l - 1)^{|S|}\right\}.
$$
It is not difficult to show that an additive function is completely determined by its restriction to $\text{Min}(x_0)$. Hence we conclude that 
\begin{align}
\label{eIneqGSX2}
\text{dim}_{\mathbb{F}_l} \mathcal{A}(\overline{X}_S, \mathbb{F}_l) \leq \prod_{i \in S} \left(|X_i| - 1\right) \cdot \prod_{j \in [d] - S} |X_j|.
\end{align}
The proposition follows upon combining equation (\ref{eIncGSX}), (\ref{eIneqGSX}) and (\ref{eIneqGSX2}).
\end{proof}

\begin{prop}
\label{pDegree}
Let $X := X_1 \times \ldots \times X_d$ be a product space and let $S \subseteq [d]$ with $|S| \geq 2$. We assume that there is some constant $A$ such that $|X_i| = A$ for all $i \in S$. We further assume that there is a governing expansion $\mathfrak{G}$ on $X$ such that $\overline{X}_S = \overline{Y}_S(\mathfrak{G})$. Define
\[
F(X) := \prod_{i \in S} \prod_{p \in X_i} \overline{\Q}^{\textup{ker}(\chi_p)} \Q\left(\zeta_l, \sqrt[l]{p}\right) \prod_{i_a \in T \subsetneq S} \prod_{\bar{x} \in \overline{X}_T} L(\phi_{T, \bar{x}}(\mathfrak{G})),
\]
where $\chi_p$ is any character from $G_\Q$ to $\langle \zeta_l \rangle$ of conductor dividing $p^\infty$. Then the degree of $F(X)$ depends only on $A$ and $|S|$, and we call it $d(A, |S|)$. If $P$ is a set of primes all equal to $0$ or $1$ modulo $l$ and disjoint from $\bigcup_{i \in S} X_i$, we have moreover
\[
\left(\prod_{\bar{x} \in \overline{Y}_S(\mathfrak{G})} L(\phi_{S, \bar{x}}(\mathfrak{G}))\right) \cap \left(\prod_{p \in P} \overline{\Q}^{\textup{ker}(\chi_p)} \Q\left(\zeta_l, \sqrt[l]{p}\right)\right) = \mathbb{Q}.
\]
Let $X' := X_1' \times \ldots \times X_d'$ be another product space with the same conditions as $X$ and further suppose that $|X_i \cap X_i'| = 1$ for all $i \in S$. Then the degree of $F(X) F(X')$ is equal to
\[
\frac{d(A, |S|)^2}{l^{2|S|}}.
\]
\end{prop}

\section{Prime divisors}
\label{sDivisor}
In the previous sections it has been very beneficial to work with product spaces of the shape $X = X_1 \times \ldots \times X_r$, where $X_i$ are disjoint non-empty sets of primes $0$ or $1$ modulo $l$. Let $S_r(N, l)$ be the set of squarefree integers of size at most $N$ with exactly $r$ prime divisors all equal to $0$ or $1$ modulo $l$. Then there is a natural injective map from $X$ to $S_r(\infty, l)$. To prove our analytic results, we will not work with all product spaces $X$, but only those that are sufficiently nice. By carefully studying $S_r(N, l)$ we are able to show that most product spaces $X$ have the nice properties we need. The material in this section is directly based on Section 5 of Smith \cite{Smith}.

\begin{mydef}
\label{dRanges}
Let $N \geq e^{e^{10 \cdot l}}$ be a large real number and let $r$ be an integer satisfying
\begin{align}
\label{er}
1 \leq r \leq 2 \log \log N.
\end{align}
For $n \in S_r(N, l)$ we write $(p_1, \ldots, p_r)$ for the prime divisors of $n$ with $p_1 < \ldots < p_r$.
\begin{itemize}
\item If $D_1 > 100$ is a real number, we say that $n$ is comfortably spaced above $D_1$ if for all $i < r$ satisfying $p_i > D_1$
\[
l^{200}D_1 < l^{200} p_i < p_{i + 1}.
\]
\item Let $C_0 > 1$ be a real number. We call $n$ $C_0$-regular if for all $i \leq \frac{1}{3}r$
\[
\left|i - \frac{r\log \log p_i}{\log \log N}\right| < C_0^{\frac{1}{5}} \cdot \max(i, C_0)^{\frac{4}{5}}.
\]
\end{itemize}
We say that $X = X_1 \times \ldots \times X_r \subseteq S_r(N, l)$ is $C_0$-regular if there is $n \in X$ that is $C_0$-regular.
\end{mydef}

For a general squarefree integer $n$, there is a well-known heuristic model for the values of $\log \log p_i$. This heuristic predicts that the values of $\log \log p_i$ for $i = 1, \ldots, r$ behave as a Poisson point process of intensity $1$. It is not hard to see that this heuristic breaks down for small and large values of $i$, but it is nevertheless a solid heuristic, see for example the work of Granville \cite{Granville}.

The heuristic model needs to be slightly modified in our setting. Recall that, loosely speaking, a typical squarefree integer $n$ has roughly $\log \log n$ prime divisors on average with standard deviation $\sqrt{\log \log n}$. We require only very weak conditions on $r$ in Definition \ref{dRanges} far outside the typical range. This makes the correction factor $r/ \log \log N$ in the definition of $C_0$-regular necessary.

Assuming the heuristic model, it is an exercise in probability theory to show that most integers are comfortably spaced above $D_1$ and $C_0$-regular. This is done in Proposition 5.2 in Smith \cite{Smith}. Remarkably enough, this proposition is then used to establish the analogous result for the integers. 

We will now show that almost all $n \in S_r(N, l)$ are comfortably spaced above $D_1$ and $C_0$-regular following the strategy of Smith \cite{Smith}. Since we are following Smith's strategy, our first goal is to generalize Proposition 5.2 of Smith \cite{Smith}.

\begin{prop}
\label{p5.2}
Let $L > 2$ be a real number and let $r \geq 1$ be an integer. Suppose that $X_1, \ldots, X_r \sim U(0, L)$ are independent, uniformly distributed random variables. Define $U_{(i)}$ to be the $i$-th order statistic of $X_1, \ldots, X_r$. For a real number $C_0 > 0$, we say that $X_1, \ldots, X_r$ are $C_0$-regular if for all $1 \leq i \leq r$
\[
\left|i - \frac{rU_{(i)}}{L}\right| < C_0^{\frac{1}{5}} \cdot \max(i, C_0)^{\frac{4}{5}}.
\]
Then there is an absolute constant $c > 0$ such that
\[
\mathbb{P}(X_1, \ldots, X_r \text{ is not } C_0\text{-regular}) = O\left(\exp(-c \cdot C_0)\right).
\]
\end{prop}

\begin{proof}
Define $L' := rL/L = r$ and $X'_i := rX_i/L$. Now apply Proposition 5.2 of Smith \cite{Smith}.
\end{proof}

Having established Proposition \ref{p5.2}, we are ready to study $S_r(N, l)$. In our proofs, we will frequently encounter the following integral
\[
I_r(u) := \int_{\substack{t_1, \ldots, t_r \geq 1 \\ t_1 + \ldots + t_r \leq u}} \frac{dt_1}{t_1} \cdot \ldots \cdot \frac{dt_r}{t_r},
\]
which was first studied by Ramanujan. It is this integral that provides the connection between $S_r(N, l)$ and the heuristic model. Note that $I_r(u)$ is trivially bounded by $(\log u)^r$. Our next lemma gives a better bound for $I_r(u)$ in some ranges of $u$ and $r$.

\begin{lemma}
\label{l5.1}
Let $\gamma$ be the Euler-Mascheroni constant. Let $u \geq 3$ be a real number and let $r \geq 1$ be an integer. If we set $\alpha := r / \log u$, we have
\[
\left|I_r(u) - \frac{e^{-\gamma \alpha}}{\Gamma(1 + \alpha)} (\log u)^r\right| = O\left((\alpha + 1) (\log u)^r \frac{(\log \log u)^3}{\log u}\right).
\]
\end{lemma}

\begin{proof}
This is Lemma 5.1 of Smith \cite{Smith}.
\end{proof}

Define $S'_r(N, l)$ to be the subset of $S_r(N, l)$ consisting of those integers that are not divisible by $l$. For technical reasons, it turns out to be more convenient to work with $S'_r(N, l)$.

\begin{theorem}
\label{tSprimeNL}
Let $N$ and $r$ be as in Definition \ref{dRanges}.
\begin{itemize}
\item Suppose that $D_1 > 100$. Then we have
\[
\frac{\left|\left\{n \in S'_r(N, l) : n \text{ is not comfortably spaced above } D_1\right\}\right|}{|S'_r(N, l)|} = O\left(\frac{1}{\log D_1} + \frac{1}{\log \log N}\right).
\]
\item There is an absolute constant $c > 0$ such that for all $C_0 > 0$
\[
\frac{\left|\left\{n \in S'_r(N, l) : n \text{ is not } C_0\text{-regular}\right\}\right|}{|S'_r(N, l)|} = O\left(\exp(-c \cdot C_0) + \exp\left(-cr^{\frac{1}{3}}\right)\right).
\]
\end{itemize}
\end{theorem}

\begin{proof}
This is a mostly straightforward generalization of Theorem 5.4 in Smith \cite{Smith}. Define for any real $x > 1$
\[
F_l(x) := \sum_{\substack{p \leq x \\ p \equiv 1 \bmod l}} \frac{1}{p}.
\]
We use the Siegel--Walfisz theorem and partial summation to obtain absolute constants $A, c > 0$ such that for all $x \geq e^l$
\[
\left|F_l(x) - \frac{1}{\varphi(l)} \log \log x - B_1(l)\right| \leq A \cdot e^{-c \sqrt{\log x}},
\]
where $B_1(l)$ is a computable constant in terms of $l$. In particular, there is a constant $A(l) > 0$ depending only on $l$ and an absolute constant $c > 0$ such that for all $x \geq 1.5$
\[
\left|F_l(x) - \frac{1}{\varphi(l)} \log \log x - B_1(l)\right| \leq A(l) \cdot e^{-c \sqrt{\log x}}.
\]
Denote by $\mathcal{P}_l$ the set of primes $1$ modulo $l$ and let $T$ be a subset of $\mathcal{P}_l^r$. Then we define $\text{Grid}(T) \subseteq \mathbb{R}^r$ by
\[
\text{Grid}(T) := \bigcup_{(p_1, \ldots, p_r) \in T} \prod_{1 \leq i \leq r} \left[\varphi(l) \cdot \left(F_l(p_i) - \frac{1}{p_i} - B_1(l)\right), \varphi(l) \cdot \left(F_l(p_i) - B_1(l)\right)\right].
\]
Here $\prod$ is to be interpreted as the Cartesian product of intervals. To facilitate our analysis of $S'_r(N, l)$ we define 
\[
S'_r(N, D, l) := \left\{n \in S'_r(N, l) : p \mid n \implies p > D\right\},
\]
where $D > 1.5$ is a real number. Let $\mathbb{R}_{\geq B}$ be the subset of $\mathbb{R}^r$ with all coordinates at least $B$. Define for $u > 0$ a real number and $r \geq 1$ an integer
\[
V_r(u, D, l) := \left\{(x_1, \ldots, x_r) \in \mathbb{R}_{\geq \varphi(l)(F_l(D) - B_1(l))}^r : e^{x_1} + \ldots + e^{x_r} \leq u\right\}.
\]
Put
\[
T_r(N, D, l) := \{(p_1, \ldots, p_r) \in \mathcal{P}_l^r : p_1 \cdot \ldots \cdot p_r < N, \ p_i > D\}.
\]
There exists a constant $\kappa(l)$ depending only on $A(l)$ and $c$ such that
\[
\exp\left(x + A(l) \exp\left(-c \cdot e^{\frac{x}{2}}\right)\right) - \exp(x) \leq \kappa(l).
\]
This implies that for a good choice of $A(l)$ and $c$
\begin{align}
\label{eVolume}
V_r(\log N - r\kappa(l), D, l) \subseteq \text{Grid}(T_r(N, D, l)) \subseteq V_r(\log N + r\kappa(l), D, l).
\end{align}
A change of variables shows
\begin{align}
\label{eVvolume}
\text{Vol}(V_r(u, D, l)) = I_r\left(e^{-\varphi(l)(F_l(D) - B_1(l))} u\right).
\end{align}
Setting $B(D, l) := e^{\varphi(l)(F_l(D) - B_1(l))}$, we deduce from (\ref{eVolume}) and (\ref{eVvolume}) that
\[
I_r\left(\frac{\log N - r\kappa(l)}{B(D, l)}\right) \leq \varphi(l)^r \sum_{\substack{p_1 \cdot \ldots \cdot p_r < N \\ p_1, \ldots, p_r \in \mathcal{P}_l \\ p_1, \ldots, p_r > D}} \frac{1}{p_1 \cdot \ldots \cdot p_r} \leq I_r\left(\frac{\log N + r\kappa(l)}{B(D, l)}\right).
\]
Now assume that $3 \cdot B(D, l) \leq \log N$ and that $r^2 \leq \log N$. Using the classical differential equation $I'_r(u) = r/u \cdot I_{r - 1}(u - 1)$ due to Ramanujan and the trivial bound for $I_{r - 1}(u - 1)$, we obtain the following equality
\[
\varphi(l)^r \sum_{\substack{p_1 \cdot \ldots \cdot p_r < N \\ p_1, \ldots, p_r \in \mathcal{P}_l \\ p_1, \ldots, p_r > D}} \frac{1}{p_1 \cdot \ldots \cdot p_r} = I_r\left(\frac{\log N}{B(D, l)}\right) + O\left(\frac{r^2}{\log N} \cdot (\log \log N - \log B(D, l))^{r - 1}\right).
\]
We introduce the following sums
\begin{align*}
F'_r(N, D, l) &:= \varphi(l)^r \sum_{\substack{p_1 \cdot \ldots \cdot p_r < N \\ p_1, \ldots, p_r \in \mathcal{P}_l \\ p_1, \ldots, p_r > D}} \frac{1}{p_1 \cdot \ldots \cdot p_r} \\
G'_r(N, D, l) &:= \varphi(l)^r \sum_{\substack{p_1 \cdot \ldots \cdot p_r < N \\ p_1, \ldots, p_r \in \mathcal{P}_l \\ p_1, \ldots, p_r > D}} \log(p_1 \cdot \ldots \cdot p_r) \\
H'_r(N, D, l) &:= \varphi(l)^r \sum_{\substack{p_1 \cdot \ldots \cdot p_r < N \\ p_1, \ldots, p_r \in \mathcal{P}_l \\ p_1, \ldots, p_r > D}} 1.
\end{align*}
Put 
\[
u := \frac{\log N}{B(D, l)}.
\]
Until now we have only assumed that $u \geq 3$ and $r^2 \leq \log N$. We additionally suppose that $\log \log N \geq 1.1 \log \log D$. Under these three assumptions we claim that
\begin{align}
\label{eGr}
G'_r(N, D, l) = rN \cdot I_{r - 1}(u) + O\left(\frac{r^4N}{\log N} \cdot (\log u)^{r + 3}\right)
\end{align}
and
\begin{align}
\label{eHr}
H'_r(N, D, l) = \frac{rN}{\log N} \cdot I_{r - 1}(u) + O\left(\frac{r^4N}{(\log N)^2} \cdot (\log u)^{r + 3}\right).
\end{align}
Recall that we have already shown that
\begin{align}
\label{eFr}
F'_r(N, D, l) = I_r(u) + O\left(\frac{r^2}{\log N} \cdot (\log u)^{r - 1}\right).
\end{align}
Let us start with (\ref{eGr}). It will be convenient to abbreviate $p_1 \cdot \ldots \cdot p_{r - 1}$ as $P$. Then we calculate
\begin{align}
\label{eCalculate}
G'_r(N, D, l) =\ &\varphi(l)^r \sum_{\substack{p_1 \cdot \ldots \cdot p_r < N \\ p_1, \ldots, p_r \in \mathcal{P}_l \\ p_1, \ldots, p_r > D}} \log(p_1 \cdot \ldots \cdot p_r) = r\varphi(l)^r \sum_{\substack{P < N/D \\ p_1, \ldots, p_{r - 1} \in \mathcal{P}_l \\ p_1, \ldots, p_{r - 1} > D}} \sum_{\substack{p \in \mathcal{P}_l \\ p > D}}^{N/P} \log p \nonumber \\
=\ &r\varphi(l)^r \sum_{\substack{P < N/D \\ p_1, \ldots, p_{r - 1} \in \mathcal{P}_l \\ p_1, \ldots, p_{r - 1} > D}} \left(\frac{N}{\varphi(l)P} \cdot \left(1 + \left(O\left(e^{-c\sqrt{\log N/P}}\right)\right)\right) - \sum_{\substack{p \in \mathcal{P}_l \\ p < D}} \log p\right) \nonumber \\
=\ &rN \cdot F'_{r - 1}(N/D, D, l) - r\varphi(l) \cdot H'_{r - 1}(N/D, D, l) \cdot \sum_{\substack{p \in \mathcal{P}_l \\ p < D}} \log p \ + \nonumber \\
&rN \varphi(l)^{r - 1} \sum_{\substack{P < N/D \\ p_1, \ldots, p_{r - 1} \in \mathcal{P}_l \\ p_1, \ldots, p_{r - 1} > D}} \frac{1}{P} \cdot O\left(e^{-c\sqrt{\log N/P}}\right).
\end{align}
To simplify (\ref{eCalculate}) we first attack
\[
rN \varphi(l)^{r - 1} \sum_{\substack{P < N/D \\ p_1, \ldots, p_{r - 1} \in \mathcal{P}_l \\ p_1, \ldots, p_{r - 1} > D}} \frac{1}{P} \cdot O\left(e^{-c\sqrt{\log N/P}}\right).
\]
Define $N_0 := Ne^{-(c^{-1} \log \log N)^2}$. Then splitting this sum in two ranges depending on $P < N_0$ or $P > N_0$ yields
\[
rN \varphi(l)^{r - 1} \sum_{\substack{P < N_0 \\ p_1, \ldots, p_{r - 1} \in \mathcal{P}_l \\ p_1, \ldots, p_{r - 1} > D}} \frac{1}{P} \cdot O\left(e^{-c\sqrt{\log N/P}}\right) + rN \varphi(l)^{r - 1} \sum_{\substack{N_0 < P < N/D \\ p_1, \ldots, p_{r - 1} \in \mathcal{P}_l \\ p_1, \ldots, p_{r - 1} > D}} \frac{1}{P} \cdot O\left(e^{-c\sqrt{\log N/P}}\right).
\]
The former sum is bounded by 
\begin{align}
\label{eBound1}
O\left(rNe^{-c\sqrt{\log N/N_0}}F'_{r - 1}(N_0, D, l)\right),
\end{align}
while the latter sum is bounded by
\begin{align}
\label{eBound2}
O\left(rNe^{-c\sqrt{\log D}}\left(F'_{r - 1}(N, D, l) - F'_{r - 1}(N_0, D, l)\right)\right).
\end{align}
A careful computation using (\ref{eFr}) shows that both (\ref{eBound1}) and (\ref{eBound2}) are in the error term of (\ref{eGr}) for our choice of $N_0$. To further simplify (\ref{eCalculate}) we look at $F'_{r - 1}(N/D, D, l)$. Because of our assumptions on $N$ and $D$ coupled with (\ref{eFr}) we can put $F'_{r - 1}(N, D, l) - F'_{r - 1}(N/D, D, l)$ in the error term of (\ref{eGr}). So far we have shown
\begin{multline}
\label{eGrHr}
G'_r(N, D, l) = rN \cdot F'_{r - 1}(N, D, l) - \\
r\varphi(l) \cdot H'_{r - 1}(N/D, D, l) \cdot \sum_{\substack{p \in \mathcal{P}_l \\ p < D}} \log p +
O\left(\frac{r^4N}{\log N} \cdot (\log u)^{r + 3}\right).
\end{multline}
To finish the proof of (\ref{eGr}), we have to deal with the term $- r\varphi(l) \cdot H'_{r - 1}(N/D, D, l)$. If we carefully go through the proof of (\ref{eGrHr}), we see that
\begin{align}
\label{eGrUpper}
G'_r(N, D, l) = O\left(rN \cdot F'_{r - 1}(N, D, l)\right)
\end{align}
without any restrictions on $D$ and $N$. Then partial summation combined with (\ref{eGrUpper}) shows
\[
H'_r(N, D, l) = \frac{G'_r(N, D, l)}{\log N} + \int_D^N \frac{G'_r(x, D, l)}{x (\log x)^2} dx = \frac{G'_r(N, D, l)}{\log N} + O\left(\frac{rN}{(\log N)^2} \cdot (\log u)^{r - 1}\right).
\]
Plugging this in (\ref{eGrHr}) and using (\ref{eGrUpper}) again establishes (\ref{eGr}) and hence (\ref{eHr}). We have the trivial upper bound
\[
|S'_r(N, D, l)| \leq \frac{1}{r!} H'_r(N, D, l).
\]
To give a lower bound for $|S'_r(N, D, l)|$ in terms of $H'_r(N, D, l)$, we observe that
\[
r! \cdot \frac{|S'_r(N, D, l)|}{H'_r(N, D, l)} = O\left(\sum_{\substack{p \in \mathcal{P}_l \\ p > D}} \frac{1}{p^2}\right) = O\left(\frac{1}{D}\right).
\]
This implies that
\[
|S'_r(N, D, l)| \geq \frac{1}{r!} H'_r(N, D, l) \cdot O\left(\frac{1}{D}\right).
\]
Hence upon taking $D > D_0(l)$ for some constant $D_0(l)$ depending only on $l$, we have the following lower bound for $|S'_r(N, D, l)|$
\[
|S'_r(N, D, l)| \geq \frac{1}{2 \cdot r!} H'_r(N, \max(D, D_0(l)), l).
\]
Altogether we have proven the following claim.

\begin{lemma}
\label{l5.5}
Let $N \geq e^{e^{10 \cdot l}}$ and $D > 1.5$ be real numbers satisfying $\log \log N \geq 1.1 \log \log D$ and $3 \cdot B(D, l) \leq \log N$. Assume that 
\begin{align}
\label{er2}
1 \leq r \leq \epsilon \log \log N
\end{align}
for some $\epsilon > 0$. Then there are positive constants $A_2(\epsilon, l)$ and $A_3(\epsilon, l)$ depending only on $\epsilon$ and $l$ such that
\[
\frac{A_2(\epsilon, l)}{\varphi(l)^r} \cdot \frac{N}{\log N} \cdot \frac{(\log u)^{r - 1}}{(r - 1)!} < |S'_r(N, D, l)| < \frac{A_3(\epsilon, l)}{\varphi(l)^r} \cdot \frac{N}{\log N} \cdot \frac{(\log u)^{r - 1}}{(r - 1)!}.
\]
\end{lemma}

We stress that condition (\ref{er2}) is essential for the correctness of Lemma \ref{l5.5}. Indeed, in general it is not clear that the main term in (\ref{eHr}) dominates the error term. However, if $r$ satisfies (\ref{er2}), we can use Lemma \ref{l5.1} to show that the main term does dominate the error term. Take $N$ and $D$ as in the previous lemma and suppose that $r$ satisfies (\ref{er2}). If $k \leq r$ is an integer we define $S'_{r, k}(N, D, l)$ to be the subset of $S'_r(N, D, l)$ with exactly $k$ prime divisors smaller than
\[
N_1 := \exp\left(\sqrt{\log N \cdot B(D, l)}\right).
\]

\begin{lemma}
\label{l5.6}
Let $N \geq e^{e^{10 \cdot l}}$ and $D > 1.5$ be real numbers satisfying $\log \log N \geq 2 \log \log D$ and $3 \cdot B(D, l) \leq \log N$. If $r$ satisfies (\ref{er}), there exists an absolute constant $c > 0$ such that
\[
\left.\left|\bigcup_{|0.5r - k| \leq r^{\frac{2}{3}}} S'_{r, k}(N, D, l)\right| \ \right/ \ |S'_r(N, D, l)| = O\left(\exp\left(-cr^{\frac{1}{3}}\right)\right).
\]
Now suppose that $k$ satisfies $|0.5r - k| \leq r^{\frac{2}{3}}$. Let
\[
\textup{Tuples}_k(N_1, D, l) := \left\{(p_1, \ldots, p_k) \in \mathcal{P}_l^k : D < p_1 < \ldots < p_k < N_1\right\}
\]
and let $T_1$ and $T_2$ be subsets of $\textup{Tuples}_k(N_1, D, l)$. Define $S'_{r, k}(N, D, l, T)$ to be the subset of $S'_{r, k}(N, D, l)$ for which the $k$ smallest prime factors $(p_1, \ldots, p_k)$ lie in $T$. Then
\[
\frac{\left|S'_{r, k}(N, D, l, T_1)\right|}{\left|S'_{r, k}(N, D, l, T_2)\right|} = O\left(\frac{\textup{Vol}(\textup{Grid}(T_1))}{\textup{Vol}(\textup{Grid}(T_2))}\right).
\]
\end{lemma}

\begin{proof}
We start with the easy formula
\begin{align}
\label{eSrk}
|S'_{r, k}(N, D, l)| = \sum_{\substack{p_1, \ldots, p_k \in \mathcal{P}_l \\ D < p_1 < \ldots < p_k < N_1}} \left|S'_{r - k}\left(\frac{N}{p_1 \cdot \ldots \cdot p_k}, N_1, l\right)\right|.
\end{align}
If $N$ is sufficiently large, we have $N_1^k < \sqrt{N}$. For convenience we will write $P := p_1 \cdot \ldots \cdot p_k$. Then for all choices of $P$
\[
\left(\frac{\log \log N - \log B(D, l)}{\log \log N/P - \log B(D, l)}\right)^{r - 1} = O(1).
\]
We also have for all choices of $P$
\[
\frac{\log N}{\log N/P} = O(1).
\]
Suppose that $r - k > 0$. An appeal to Lemma \ref{l5.5} yields constants $A_4(l)$ and $A_5(l)$ depending only on $l$ with the property
\begin{align}
\label{eSrkNP}
\frac{A_4(l)}{P} \cdot |S'_{r - k}(N, N_1, l)| < |S'_{r - k}(N/P, N_1, l)| < \frac{A_5(l)}{P} \cdot |S'_{r - k}(N, N_1, l)|.
\end{align}
Plugging (\ref{eSrkNP}) into (\ref{eSrk}) shows
\begin{align}
\label{eSrkHoeff}
|S'_{r, k}(N, D, l)| = O\left(\frac{N}{\log N} \cdot \frac{2^{-r + 1} (\log \log N - \log B(D, l))^{r - 1}}{\varphi(l)^r(r - k - 1)! k!}\right)
\end{align}
for $r - k > 0$. In case $r = k$ we have the trivial bound $|S'_{r, k}(N, D, l)| \leq N_1^r$, so we can remove this case from the union. Now we are in the position to apply Hoeffding's inequality to (\ref{eSrkHoeff}), which proves the first part of the lemma. The second part quickly follows from (\ref{eSrkNP}).
\end{proof}

Lemma \ref{l5.6} directly relates $S'_{r, k}(N, D, l)$ and the heuristic model introduced at the beginning of the section. Since we have already dealt with the heuristic model in Proposition \ref{p5.2}, the second part of Theorem \ref{tSprimeNL} is now straightforward. Indeed, we first restrict to those $k$ for which we have $|0.5r - k| \leq r^{\frac{2}{3}}$. Take $T_2 := \text{Tuples}_k(N_1, 1.5, l)$ and take $T_1$ to be the subset of $T_2$ that is not $C_0$-regular. There exists an absolute constant $c > 0$ such that
\[
\text{Vol}(\text{Grid}(T_2)) \geq \frac{c}{k!} \cdot (\log \log N_1)^k.
\]
Furthermore, all elements of $\text{Grid}(T_1)$ are not $C_0 - \kappa(l)$-regular for some constant $\kappa(l) > 0$ depending only on $l$. Proposition \ref{p5.2} shows that
\[
\text{Vol}(\text{Grid}(T_1)) = O\left(\frac{\exp(-c' \cdot C_0)}{k!} \cdot (\log \log N_1)^k\right)
\]
for some absolute constant $c' > 0$. This proves the second part of Theorem \ref{tSprimeNL}.

It remains to prove the first part of Theorem \ref{tSprimeNL}. If we have $r \leq 2$, we can directly prove Theorem \ref{tSprimeNL}, so suppose that $r \geq 3$. The number of $n \in S'_r(N, l)$, that are not comfortably spaced above $D_1$, is bounded by
\begin{align}
\label{eUnc}
\sum_{p > D_1}^N \sum_{q > p}^{l^{200}p} |S'_{r - 2}(N/pq, l)|.
\end{align}
We split (\ref{eUnc}) in two ranges depending on $p < N^{\frac{1}{4}}$ and $p \geq N^{\frac{1}{4}}$. First suppose that $p < N^{\frac{1}{4}}$. Then we use Lemma \ref{l5.5} to bound (\ref{eUnc}) by
\[
O\left(|S'_{r - 2}(N, l)| \cdot \sum_{p > D_1}^N \sum_{q > p}^{l^{200}p} \frac{1}{pq}\right) = O\left(\frac{|S'_r(N, l)|}{\log D_1}\right).
\]
Now suppose that $p \geq N^{\frac{1}{4}}$. From the crude bound $|S'_{r - 2}(N/pq, l)| \leq N/pq$ we deduce
\[
\sum_{p > N^{\frac{1}{4}}}^N \sum_{q > p}^{l^{200}p} |S'_{r - 2}(N/pq, l)| = O\left(\frac{N}{\log N}\right) = O\left(\frac{\left|S'_r(N, l)\right|}{\log \log N}\right),
\]
since $r > 1$. This implies the theorem.
\end{proof}

\noindent Finally, we use Theorem \ref{tSprimeNL} to prove most integers in $S_r(N, l)$ are comfortably spaced above $D_1$ and $C_0$-regular.

\begin{theorem}
\label{tSNL}
Let $N$ and $r$ be as in Definition \ref{dRanges}.
\begin{itemize}
\item Suppose that $D_1 > 100$. Then we have
\[
\frac{\left|\left\{n \in S_r(N, l) : n \text{ is not comfortably spaced above } D_1\right\}\right|}{|S_r(N, l)|} = O\left(\frac{1}{\log D_1} + \frac{1}{\log \log N}\right).
\]
\item There is an absolute constant $c > 0$ such that for all $C_0 > 0$
\[
\frac{\left|\left\{n \in S_r(N, l) : n \text{ is not } C_0\text{-regular}\right\}\right|}{|S_r(N, l)|} = O\left(\exp(-c \cdot C_0) + \exp\left(-cr^{\frac{1}{3}}\right)\right).
\]
\end{itemize}
\end{theorem}

\begin{proof}
This is an easy consequence of Theorem \ref{tSprimeNL}.
\end{proof}

\section{Redei matrices}
To prove our main results, we will arrange cyclic degree $l$ extensions in product spaces called boxes. These boxes provide the combinatorial structure that allows us to apply the results from the previous sections. Let us start by giving a precise definition of a box.

\begin{mydef}
Let $\max(100, l) < D_1$ be a real number and let $1 \leq k \leq r$ be integers. Choose a sequence of primes
\[
p_1 < \ldots < p_k < D_1
\]
all equal to $0$ or $1$ modulo $l$. Also choose real numbers
\[
D_1 < t_{k + 1} < \ldots < t_r.
\]
For $i \geq k + 1$ define
\[
t_i' := \left(1 + \frac{1}{e^{i - k} \cdot \log D_1}\right) \cdot t_i.
\]
We assume that $t_i \geq l^{100}t_{i - 1}'$ for all $k + 1 \leq i \leq r$, where by definition $t_k' := p_k$. Put
\[
X := X_1 \times \ldots \times X_r,
\]
where $X_i := \{p_i\}$ for $i \leq k$ and $X_i$ is the set of primes equal to $1$ modulo $l$ in the interval $(t_i, t_i')$ for $i \geq k + 1$. Finally, choose a function $f : X_1 \coprod \ldots \coprod X_r \rightarrow [l - 1]$. We will say that the set $X$ as constructed above is a box.

\end{mydef}

In Smith \cite{Smith}, the transition from squarefree integers to boxes is done by appealing to his Proposition 6.9. Instead of degree $2$ extensions of $\Q$, we need to keep track of cyclic degree $l$ extensions and this is substantially more difficult. One key difference is the fact that there are precisely $l - 1$ characters $\chi: G_\Q \rightarrow \langle \zeta_l \rangle$ with the same field of definition, while our algebraic results use only one of the $l - 1$ characters.

It will therefore be important to keep track of the characters we have chosen, and this is the reason for introducing $f$. For $x \in X$ we define the character
\[
\chi_{x, f} := \prod_{1 \leq i \leq r} \chi_{\pi_i(x)}^{f(\pi_i(x))},
\]
where $\chi_p$ is the character of conductor dividing $p^\infty$ that we fixed in Subsection \ref{conventions}. Given a box $X$ we also define $\text{Field}(X)$ to be the set of cyclic degree $l$ number fields $K$ over $\Q$ satisfying
\begin{itemize}
\item for all $1 \leq i \leq r$ there is exactly one prime $p \in X_i$ such that $K$ is ramified at $p$;
\item for all primes $p$ that are not in any of the $X_i$ we have that $K$ is unramified.
\end{itemize}
Given $X$ and $f$ there is a natural map $i_f: X \rightarrow \text{Field}(X)$ that sends $x$ to the field fixed by the kernel of $\chi_{x, f}$. There is also a natural map $j: \text{Field}(X) \rightarrow S_r(\infty, l)$ given by sending $K$ to the radical of $D_K$. The compositum of the two maps $j \circ i_f: X \rightarrow S_r(\infty, l)$ is the natural inclusion from $X$ to $S_r(\infty, l)$ that we have seen in Section \ref{sDivisor}.

Finally, define $\text{Field}(N, r, l)$ to be the set of cyclic degree $l$ number fields $K$ over $\Q$ with the following properties
\begin{itemize}
\item the radical of $D_K$ is at most $N$;
\item $D_K$ has exactly $r$ prime divisors.
\end{itemize}
Our first proposition is a variant of Smith's Proposition 6.9 that can deal with the more complicated structure of our new boxes.

\begin{prop}
\label{p6.9}
Let $l$ be a prime and let $N \geq D_1 > \max(100, l)$ with $\log \log N \geq 2 \log \log D_1$. Suppose that $r$ satisfies (\ref{er}) and let $W$ be a subset of $S_r(N, l)$ that is comfortably spaced above $D_1$. Let $\epsilon > 0$ be such that
\[
|W| > (1 - \epsilon) \cdot |S_r(N, l)|.
\]
Let $V$ be a subset of $\textup{Field}(N, r, l)$. Assume that there exists $\delta > 0$ such that for all boxes $X$ with $j(\textup{Field}(X)) \cap W \neq \emptyset$ and $\textup{Field}(X) \subseteq \textup{Field}(N, r, l)$ we have
\[
(\delta - \epsilon) \cdot |\textup{Field}(X)| < |V \cap \textup{Field}(X)| < (\delta + \epsilon) \cdot |\textup{Field}(X)|.
\]
Then
\[
|V| = \delta \cdot |\textup{Field}(N, r, l)| + O\left(\left(\epsilon + \frac{1}{\log D_1}\right) \cdot |\textup{Field}(N, r, l)|\right).
\]
\end{prop}

\begin{proof}
Let $0 \leq k \leq r$ be an integer. Let $T_k$ be the set of tuples $(p_1, \ldots, p_k)$ such that
\[
p_1 < \ldots < p_k < D_1
\]
and all the $p_i$ are primes $0$ or $1$ modulo $l$. Define $T'_k$ to be set of tuples $(t_{k + 1}, \ldots, t_r)$ satisfying
\[
D_1 < t_{k + 1} < \ldots < t_r,
\]
where the $t_i$ are real numbers with $t_{i + 1} \geq 2t_i$. Given $\mathbf{t} \in T_k$ and $\mathbf{t'} \in T'_k$ there is a natural way to construct a box $X(\mathbf{t}, \mathbf{t'})$. Take
\[
T''_k := \left\{(\mathbf{t}, \mathbf{t'}) \in T_k \times T'_k : j(\text{Field}(X(\mathbf{t}, \mathbf{t'}))) \cap W \neq \emptyset \text{ and } \text{Field}(X(\mathbf{t}, \mathbf{t'})) \subseteq \text{Field}(N, r, l)\right\}.
\]
Now consider
\begin{align}
\label{eLU}
\sum_{\mathbf{t} \in T_k} \int\limits_{\substack{\mathbf{t'} \in T'_k \\ (\mathbf{t}, \mathbf{t'}) \in T''_k}} \frac{|V \cap \text{Field}(X(\mathbf{t}, \mathbf{t'}))|}{t_{k + 1} \cdot \ldots \cdot t_r}  \ dt_{k + 1} \cdot \ldots \cdot dt_r.
\end{align}
Let $K \in V$ with $j(K) := (q_1, \ldots, q_r) \in W$. Note that $K \in \text{Field}(X(\mathbf{t}, \mathbf{t'}))$ if and only if
\[
(q_1, \ldots, q_k) = \mathbf{t}
\]
and for all $k + 1 \leq i \leq r$
\[
t_i < q_i < \left(1 + \frac{1}{e^{i - k} \cdot \log D_1}\right) \cdot t_i.
\]
If $K$ also satisfies
\begin{align}
\label{en}
q_1 \cdot \ldots \cdot q_r < N \cdot \prod_{i = k + 1}^r \left(1 + \frac{1}{e^{i - k} \cdot \log D_1}\right)^{-1},
\end{align}
we see that the contribution of $K$ to (\ref{eLU}) is equal to
\begin{align}
\label{eContribution}
\prod_{i = k + 1}^r \log\left(1 + \frac{1}{e^{i - k} \cdot \log D_1}\right).
\end{align}
If $j(K)$ is outside $W$ or $j(K)$ does not satisfy (\ref{en}), we see that the contribution of $K$ to (\ref{eLU}) is bounded by (\ref{eContribution}). Put
\[
H_r(N, l) := \sum_{\substack{p_1 \cdot \ldots \cdot p_r < N \\ p_i \equiv 0, 1 \bmod l}} 1.
\]
Due to (\ref{eHr}) and Lemma \ref{l5.5} we have for all $c \in (0, 0.5)$
\[
\frac{1}{r!} \cdot \left(H_r(N, l) - H_r((1 - c) \cdot N, l)\right) = O\left(c + \frac{(\log \log N)^4}{\log N}\right) \cdot |S_r(N, l)|.
\]
Hence the number of elements in $S_r(N, l)$ failing (\ref{en}) is bounded by $O\left(\left|S_r(N, l)\right| / \log D_1\right)$. In particular, we can bound the number of $K \in \text{Field}(N, r, l)$ with $j(K)$ failing (\ref{en}) by $O\left(\left|\text{Field}(N, r, l)\right| / \log D_1\right)$. Finally consider
\begin{align}
\label{eLU2}
\sum_{k \geq 0}^r \prod_{i = k + 1}^r \left(1 + \frac{1}{e^{i - k} \cdot \log D_1}\right)^{-1} \sum_{\mathbf{t} \in T_k} \int\limits_{\substack{\mathbf{t'} \in T'_k \\ (\mathbf{t}, \mathbf{t'}) \in T''_k}} \frac{|V \cap \text{Field}(X(\mathbf{t}, \mathbf{t'}))|}{t_{k + 1} \cdot \ldots \cdot t_r}  \ dt_{k + 1} \cdot \ldots \cdot dt_r.
\end{align}
Since the contribution of any $K \in V$ to (\ref{eLU}) is always bounded by (\ref{eContribution}), we have an upper bound for (\ref{eLU2}) given by $|V|$. On the other hand, if $K \in V \cap j^{-1}(W)$ and satisfies (\ref{en}), the contribution of $K$ to (\ref{eLU}) is equal to (\ref{eContribution}). This yields a lower bound for (\ref{eLU2}), namely
\[
|V \cap j^{-1}(W)| - O\left(\frac{\left|\text{Field}(N, r, l)\right|}{\log D_1}\right).
\]
Using our assumption
\[
(\delta - \epsilon) \cdot |\text{Field}(X)| < |V \cap \text{Field}(X)| < (\delta + \epsilon) \cdot |\text{Field}(X)|
\]
for all boxes $X$ with $j(\text{Field}(X)) \cap W \neq \emptyset$ and $\text{Field}(X) \subseteq \text{Field}(N, r, l)$, we can again obtain upper and lower bounds for (\ref{eLU2}). Indeed, we have an upper bound for (\ref{eLU2}) given by $(\delta + \epsilon) \cdot |\text{Field}(N, r, l)|$ and a lower bound for (\ref{eLU2}) given by
\[
\delta \cdot |\text{Field}(N, r, l)| - O\left(\left(\epsilon + \frac{1}{\log D_1}\right) \cdot |\text{Field}(N, r, l)|\right).
\]
Combining the various lower and upper bounds finishes the proof of the proposition.
\end{proof}

Proposition \ref{p6.9} usefulness lies in the fact that it allows us to deduce equidistribution of $\text{Field}(N, r, l)$ from equidistribution of product spaces $\text{Field}(X)$. However, our algebraic results work for a product space of the shape $i_f(X)$ and not for the full set $\text{Field}(X)$. To work around this issue, the identity
\[
|\{f: K \in i_f(X)\}| = |\{f: K' \in i_f(X)\}|
\]
for all $K, K' \in \text{Field}(X)$ will be pivotal in our next sections.

In our coming sections it will also be important to have very fine control over $r$. Up until this point we have only assumed that $r$ satisfies (\ref{er}), but we will now introduce the much stronger requirement
\begin{align}
\label{er3}
\left|r - \log \log N\right| \leq \left(\log \log N\right)^\frac{2}{3}.
\end{align}
Our next theorem shows that (\ref{er3}) is usually satisfied.

\begin{theorem}
\label{tRanger}
Recall that $\textup{Field}(N, l)$ is the set of cyclic degree $l$ number fields $K$ over $\Q$ with $\text{rad}(D_K) \leq N$. Then we have
\begin{align}
\label{eRanger}
\left||\textup{Field}(N, l)| - \bigcup_{r \textup{ sat. } (\ref{er3})}|\textup{Field}(N, r, l)|\right| = O\left(\frac{|\textup{Field}(N, l)|}{\left(\log \log N\right)^c}\right)
\end{align}
for some absolute constant $c > 0$.
\end{theorem}

\begin{proof}
Note that the map $j: \text{Field}(N, r, l) \rightarrow S_r(N, l)$ is $(l - 1)^{r - 1}$ to $1$. This observation combined with Lemma \ref{l5.5} shows that
\[
\left|\left\{K \in \text{Field}(N, l) : \omega(D_K) \leq \log \log N - \left(\log \log N\right)^\frac{2}{3}\right\}\right| = O\left(\frac{N}{\left(\log \log N\right)^c}\right).
\]
for some absolute constant $c > 0$. From Lemma 2.2 of \cite{Pollack} we infer that $N = O\left(\left|\text{Field}(N, l)\right|\right)$. We conclude that our error term fits in the error term of the theorem. To deal with the case
\[
\omega(D_K) \geq \log \log N + \left(\log \log N\right)^\frac{2}{3},
\]
we take a different approach. Indeed, Lemma \ref{l5.5} does not directly apply, since condition (\ref{er2}) may not be satisfied. In the classical paper \cite{HR} it is proven that there are absolute constants $C > 0$ and $K > 0$ such that
\[
\pi_k(x) < \frac{Kx}{\log x} \frac{\left(\log \log x + C\right)^k}{k!},
\]
where $\pi_k(x)$ is equal to the number of squarefree integers with size at most $x$ and exactly $k$ prime divisors. A straightforward generalization of their argument proves
\[
|S_r(N, l)| < \frac{KN}{(l - 1)^r\log N} \frac{\left(\log \log N + C\right)^r}{r!}
\]
for some absolute constants $C > 0$ and $K > 0$. Then a small computation finishes the proof of (\ref{eRanger}).
\end{proof}

For $\alpha \in \Z[\zeta_l]$ and $\mathfrak{n}$ an ideal of $\Z[\zeta_l]$ with $(\mathfrak{n}, 1 - \zeta_l) = (1)$ we write
\[
\left(\frac{\alpha}{\mathfrak{n}}\right)_{\Z[\zeta_l], l}
\]
for the $l$-th power residue symbol in $\Z[\zeta_l]$. We assume that the reader is familiar with the basic properties of the power residue symbol. Suppose that $p \neq l$. Then, given $\chi_p$, there is a unique prime ideal $\pp$ of $\Z[\zeta_l]$ satisfying
\[
\chi_p(\text{Frob}(q)) = \left(\frac{q}{\pp}\right)_{\Z[\zeta_l], l} = \frac{\text{Frob}(\pp)\left(\sqrt[l]{q}\right)}{\sqrt[l]{q}}
\]
for all primes $q$. We will now define a generalized Redei matrix.

\begin{mydef}
\label{dRedei}
Take $P$ to be a set of prime numbers $1$ modulo $l$. Choose a function $f : P \coprod X_1 \coprod \ldots \coprod X_r \rightarrow [l - 1]$. Let $X$ be any box with $X_1, \ldots, X_r$ disjoint from the set $P$. Put
\[
M := \{(i, j): 1 \leq i, j \leq r, i \neq j\}.
\]
We also define
\[
M_{P, 1} := [r] \times P, \quad M_{P, 2} := P \times [r].
\]
For an assignment $a: M \coprod M_{P, 1} \coprod M_{P, 2} \rightarrow \langle \zeta_l \rangle$, we define $X(a)$ to be the set of tuples $(x_1, \ldots, x_r) \in X$ satisfying
\begin{itemize}
\item for all $(i, j) \in M$
\[
\chi_{x_j}^{f(x_j)}(\textup{Frob}(x_i)) = a(i, j);
\]
\item for all $(i, p) \in M_{P, 1}$
\[
\chi_p^{f(p)}(\textup{Frob}(x_i)) = a(i, p);
\]
\item for all $(p, j) \in M_{P, 2}$
\[
\chi_{x_j}^{f(x_j)}(\textup{Frob}(p)) = a(p, j).
\]
\end{itemize}
\end{mydef}

Note that $X(a)$ depends on the choice of the $f$. To make this more explicit we will sometimes write $X(a, f)$. We can think of the assignment $a$ as an analogue of the classical Redei matrix. In our final section we need to treat $X(a, f)$ as fixed. For this reason we would like to prove that $X(a, f)$ is of the expected size. Unfortunately, this turns out to be rather hard. The reason for this is that we made one choice of $\chi_p$, but we could just as well have chosen $\chi_p^s$ for some integer $s$ with $(s, l) = 1$. This creates substantial difficulties, for example, when dealing with sums of the type
\[
\sum_{X <  p < Y} \chi_p(\text{Frob}(q))
\]
for fixed $q$. Indeed, by changing many of the $\chi_p$ to $\chi_p^s$ it is easy to make such a sum unbalanced. Instead we will prove something weaker, but still sufficient for our application. There is one final obstacle that we need to deal with; for $i \leq k$ the set $X_i$ consists of only one element. Hence we must restrict our attention to a special set of $a$. 

\begin{mydef}
Let $X$ be a box and let $a: M \coprod M_{P, 1} \coprod M_{P, 2} \rightarrow \langle \zeta_l \rangle$. For $1 \leq i \leq k$, let $x_i$ be the unique element of $X_i$. We say that $a$ agrees with $X$ at stage $k$ if
\begin{itemize}
\item for all $(i, j) \in M$ with $i, j \leq k$
\[
\chi_{x_j}^{f(x_j)}(\textup{Frob}(x_i)) = a(i, j);
\]
\item for all $(i, p) \in M_{P, 1}$ with $i \leq k$
\[
\chi_p^{f(p)}(\textup{Frob}(x_i)) = a(i, p);
\]
\item for all $(p, j) \in M_{P, 2}$ with $j \leq k$
\[
\chi_{x_j}^{f(x_j)}(\textup{Frob}(p)) = a(p, j).
\]
\end{itemize}
We stress that this notion implicitly depends on the choice of $f$. Whenever we need to make the choice of $f$ explicit, we will say instead that $X$ agrees with $a$ and $f$.
\end{mydef}

Clearly, if $a$ does not agree with $X$, we have $X(a) = \emptyset$. It will be convenient to define
\[
g(l, P, k) := l^{|M| + |M_{P, 1}| + |M_{P, 2}| - k(k - 1)  - 2k|P|},
\]
so that $g(l, P, k)$ is equal to the number of $a$ that agree with a given $X$. Let $K$ be a Galois extension of $\Q$. Then we define
\[
X(a, f, K) := \{x \in X(a, f) : \pi_i(x) \text{ splits completely in } K \text{ for all } k + 1 \leq i \leq r\}.
\]
Also define for a set of primes $P$
\[
L(P) := \prod_{p \in P} \overline{\Q}^{\text{ker}(\chi_p)} \Q\left(\zeta_l, \sqrt[l]{p}\right).
\]
We can now show that $X(a, f, K)$ is of the expected size for most choices of $f$.

\begin{theorem}
\label{t4r}
Assume GRH and let $l$ be an odd prime. There are constant $A(l), A'(l) > 0$ such that the following holds. Let $X$ be a box with $D_1 > A(l)$. Let $P$ be a set of primes $1$ modulo $l$ disjoint from all the $X_i$. Write $x_i$ for the unique element of $X_i$ for $1 \leq i \leq k$ and fix a function $g : P \cup \{x_1, \ldots, x_k\} \rightarrow [l - 1]$. Suppose that the assignment $a: M \coprod M_{P, 1} \coprod M_{P, 2} \rightarrow \langle \zeta_l \rangle$ agrees with $X$ at stage $k$. Let $K$ be a Galois extension of $\Q$ of degree $n_K$ that is disjoint from $L(P \cup x)$ for all $x \in X$. Define $B$ to be the maximum of the primes in $P$ and $X_{k + 1}$. We assume that
\begin{align}
\label{eDKUpper}
l^{2|P| + 2k} \cdot \left(n_K \cdot \log B + \log \left|\Delta_{K/\Q}\right|\right) \leq t_{k + 1}^{\frac{1}{8}}.
\end{align}
Then the proportion of $f$ in $\textup{Map}(X_{k + 1} \coprod \ldots \coprod X_r, [l - 1])$ with
\begin{align}
\label{eRedei}
\left||X(a, f, K)| - \frac{|X|}{n_K^{r - k} g(l, P, k)}  \right| > A'(l) \cdot \frac{|X|}{n_K^{r - k} g(l, P, k)} \cdot x_k^{-\frac{1}{4}}
\end{align}
is at most
\[
O\left(e^{-2\left|X_{k + 1}\right|^{\frac{1}{8}}}\right).
\]
\end{theorem}

\begin{proof}
Here and later on we implicitly extend our function $f \in \textup{Map}(X_{k + 1} \coprod \ldots \coprod X_r, [l - 1])$ to $f \in \textup{Map}(P \coprod X_1 \coprod \ldots \coprod X_r, [l - 1])$ by using the function $g$. Define for $k + 1 \leq i \leq r$
\[
s_i := \frac{|X_i|}{n_K \cdot l^{2|P| + 2i - 2}}, \quad s_{r + 1} = 0.
\]
We claim that there exist constants $A_1(l), A_2(l) > 0$ depending only on $l$ such that the proportion of $f$ with
\begin{align}
\label{eRedei2}
\left||X(a, f, K)| - \frac{|X|}{n_K^{r - k} g(l, P, k)} \right| > A_1(l) \cdot \frac{|X|}{n_K^{r - k} g(l, P, k)} \cdot \left(\sum_{k + 1 \leq i \leq r} t_i^{-\frac{1}{4}} \right)
\end{align}
is bounded by
\[
A_2(l) \cdot e^{-2s_{k + 1}^{\frac{1}{4}}}.
\]
Once we establish the claim, we immediately deduce (\ref{eRedei}) and the theorem. We proceed by downwards induction on $k$ with base case $k = r$. In the base case (\ref{eRedei2}) is trivial, so henceforth we shall assume $k < r$. Define the number field
\[
L := L(P \cup \{x_1, \ldots, x_k\}).
\]
There is an isomorphism between $\Gal(L/\Q(\zeta_l))$ and $(\langle \zeta_l \rangle \times \langle \zeta_l \rangle)^{|P| + k}$ given by
\[
\sigma \mapsto \left(\chi_p(\sigma), \frac{\sigma\left(\sqrt[l]{p}\right)}{\sqrt[l]{p}}\right)
\]
on each coordinate, where $p$ runs trough $P$ and $x_1, \ldots, x_k$. Then the Galois group of $L$ over $\Q$ is naturally isomorphic to
\[
\Gal(L/\Q) \simeq (\langle \zeta_l \rangle \times \langle \zeta_l \rangle)^{|P| + k} \rtimes \mathbb{F}_l^\ast.
\]
If $s \in \mathbb{F}_l^\ast$, the $\mathbb{F}_l^\ast$-action is given by
\[
(x, y) \mapsto (x, y^s)
\]
on every copy of $\langle \zeta_l \rangle \times \langle \zeta_l \rangle$. We denote this automorphism of $(\langle \zeta_l \rangle \times \langle \zeta_l \rangle)^{|P| + k}$ by $T_s$. Using the classical formula
\[
\Delta_{L/\Q} = N_{\Q(\zeta_l)/\Q}\left(\Delta_{L/\Q(\zeta_l)}\right) \cdot \Delta_{\Q(\zeta_l)/\Q}^{[L : \Q(\zeta_l)]},
\]
we obtain the following bound
\[
\log \left|\Delta_{L/\Q}\right|  \ll (|P| + k) \cdot l^{2|P| + 2k} \cdot \log B \ll l^{2|P| + 2k} \cdot \log B.
\]
This implies
\begin{align*}
\log \left|\Delta_{KL/\Q}\right| &\leq n_K \log \left|\Delta_{L/\Q}\right| + (l - 1) \cdot l^{2|P| + 2k} \log \left|\Delta_{K/\Q}\right| \\
&\ll l^{2|P| + 2k} \cdot \left(n_K \cdot \log B + \log \left|\Delta_{K/\Q}\right|\right) \leq t_{k + 1}^{\frac{1}{8}}
\end{align*}
where the last inequality is just equation (\ref{eDKUpper}). We know that
\[
\Gal(KL/\Q) \simeq \Gal(K/\Q) \times \Gal(L/\Q),
\]
and we let $p_1$ and $p_2$ be the natural projection maps. Let $C$ be a conjugacy class of $\Gal(KL/\Q)$ with $p_1(C) = \text{id}$ and $p_2(C) \subseteq \Gal(L/\Q(\zeta_l))$. Then $C$ is equal to $\{(\text{id}, T_s\sigma): s \in \mathbb{F}_l^\ast\}$ for some $\sigma \in \Gal(L/\Q(\zeta_l))$. The Chebotarev density theorem yields conditional on GRH \cite{LO}
\[
\sum_{\substack{t_{k + 1} < p < t'_{k + 1} \\ p \equiv 1 \bmod l \\ \text{Frob}(p) = C}} 1 = \frac{|C|}{n_K \cdot l^{2|P| + 2k}} \left(\left|X_{k + 1}\right| + O\left(t_{k + 1}^{\frac{5}{8}}\right)\right).
\]
Recall that we have chosen prime ideals $\pp$ in $\Z[\zeta_l]$ above every $p \equiv 1 \mod l$ in the range $t_{k + 1} < p < t'_{k + 1}$. Hence we obtain
\begin{align}
\label{eCheb}
\sum_{\substack{t_{k + 1} < p < t'_{k + 1} \\ p \equiv 1 \bmod l \\ \text{Frob}(\pp) \in C}} 1 = \frac{|C|}{n_K \cdot l^{2|P| + 2k}} \left(\left|X_{k + 1}\right| + O\left(t_{k + 1}^{\frac{5}{8}}\right)\right).
\end{align}
Let $a: M \coprod M_{P, 1} \coprod M_{P, 2} \rightarrow \langle \zeta_l \rangle$ be an assignment that agrees with $X$ at stage $k$. Our next step is to attach a conjugacy class $C$ to $a$, for which we will use (\ref{eCheb}). There exists exactly one element $(a_{1i}, a_{2i})_{1 \leq i \leq |P| + k}$ satisfying
\begin{align*}
a(k + 1, i) = a_{1i} \text{ for all } 1 \leq i \leq k, &\quad a(i, k + 1) = a_{2i} \text{ for all } 1 \leq i \leq k \\
a(k + 1, p) = a_{1p} \text{ for all } p \in P, &\quad a(p, k + 1) = a_{2p} \text{ for all } p \in P.
\end{align*}
Define $C$ to be the conjugacy class of $\text{Gal}(KL/\Q)$ with $p_1(C) = \text{id}$ and
\[
(a_{1i}, a_{2i})_{1 \leq i \leq |P| + k} \in p_2(C).
\]
Take $\sigma := (a_{1i}, a_{2i})_{1 \leq i \leq |P| + k}$. Then we say that $f$ is balanced if
\begin{align}
\label{eBalance}
\left|\sum_{\substack{t_{k + 1} < p < t'_{k + 1}, \ p \equiv 1 \bmod l \\ p_1(\text{Frob}(p)) = \text{id}, \ p_2(\text{Frob}(\pp)) = \sigma}} 1 - \sum_{\substack{t_{k + 1} < p < t'_{k + 1}, \ p \equiv 1 \bmod l \\ \text{Frob}(\pp) \in C}} \frac{1}{|C|}\right| \leq \left(\frac{|C||X_{k + 1}|}{n_K \cdot l^{2|P| + 2k}}\right)^{\frac{5}{8}},
\end{align}
and we say that $f$ is unbalanced otherwise. We deduce from Hoeffding's inequality that the proportion of unbalanced $f$ is bounded by 
\[
\frac{A_2(l)}{2} \cdot e^{-2s_{k + 1}^{\frac{1}{4}}}
\]
for a good choice of $A_2(l)$. Take $p \in X_{k + 1}$ with $p_1(\text{Frob}(p)) = \text{id}$ and $p_2(\text{Frob}(\pp)) = \sigma$. Define the box
\[
X_p := X_1 \times \ldots X_k \times \{p\} \times X_{k + 2} \times \ldots \times X_r.
\]
Then we have the decomposition
\begin{align}
\label{eCombinatorial}
|X(a, f, K)| = \sum_{\substack{t_{k + 1} < p < t'_{k + 1} \\ p \equiv 1 \bmod l \\ p_1(\text{Frob}(p)) = \text{id} \\ p_2(\text{Frob}(\pp)) = \sigma}} |X_p(a, f, K)|.
\end{align}
To apply the induction hypothesis we must check that equation (\ref{eDKUpper}) is still valid. This is a straightforward computation and an appeal to the induction hypothesis gives
\begin{align}
\label{eIH}
\left||X_p(a, f, K)| - \frac{|X_p|}{n_K^{r - k - 1}g(l, P, k + 1)}\right| \leq A_1(l) \cdot \frac{|X_p|}{n_K^{r - k - 1}g(l, P, k + 1)} \cdot \left(\sum_{k + 2 \leq i \leq r} t_i^{-\frac{1}{4}} \right)
\end{align}
except for a proportion of $f$ bounded in magnitude by
\[
A_2(l) \cdot e^{-2s_{k + 2}^{\frac{1}{4}}}.
\]
We say that $f$ is exceptional if $f$ is unbalanced or fails equation (\ref{eIH}) for some choice of $X_p$. Then the total proportion of exceptional $f$ is bounded by
\[
\frac{A_2(l)}{2} \cdot e^{-2s_{k + 1}^{\frac{1}{4}}} + A_2(l) \cdot \left|X_{k + 1}\right| \cdot e^{-2s_{k + 2}^{\frac{1}{4}}} \leq A_2(l) \cdot e^{-2s_{k + 1}^{\frac{1}{4}}},
\]
if $D_1$ is sufficiently large. We employ the triangle inequality to bound the LHS of equation (\ref{eRedei2}) as follows
\begin{multline}
\label{eTriangle}
\left||X(a, f, K)| - \sum_{\substack{t_{k + 1} < p < t'_{k + 1}, \ p \equiv 1 \bmod l \\ p_1(\text{Frob}(p)) = \text{id}, \ p_2(\text{Frob}(\pp)) = \sigma}} \frac{|X_p|}{n_K^{r - k - 1}g(l, P, k + 1)}\right| + \\
\left|\sum_{\substack{t_{k + 1} < p < t'_{k + 1}, \ p \equiv 1 \bmod l \\ p_1(\text{Frob}(p)) = \text{id}, \ p_2(\text{Frob}(\pp)) = \sigma}} \frac{|X_p|}{n_K^{r - k - 1}g(l, P, k + 1)} - \frac{|X|}{n_K^{r - k} g(l, P, k)} \right|.
\end{multline}
If $f$ is not exceptional, we deduce from (\ref{eCombinatorial}) and (\ref{eIH}) that the first term of equation (\ref{eTriangle}) is bounded by
\begin{align}
\label{eFromIH}
\frac{A_1(l)|X|/|X_{k + 1}|}{n_K^{r - k - 1}g(l, P, k + 1)} \cdot \left(\sum_{k + 2 \leq i \leq r} t_i^{-\frac{1}{4}} \right) \cdot \sum_{\substack{t_{k + 1} < p < t'_{k + 1} \\ p \equiv 1 \bmod l \\ p_1(\text{Frob}(p)) = \text{id} \\ p_2(\text{Frob}(\pp)) = \sigma}} 1.
\end{align}
For sufficiently large $D_1$, we use (\ref{eCheb}) and (\ref{eBalance}) to bound the second term of equation (\ref{eTriangle}) and to bound equation (\ref{eFromIH}). Then a straightforward computation completes the proof of the theorem.
\end{proof}

\section{Klys revisited}
Suppose that $1 \leq k \leq r, s$ are integers. Define 
\[
P(r, s, l, j) := \frac{\left|\left\{A \in \text{Mat}(r, s, \mathbb{F}_l) : \text{dim}(\text{ker}(A)) = j\right\}\right|}{\left|\text{Mat}(r, s, \mathbb{F}_l)\right|}.
\]
Fix $M \in \text{Mat}(k, k, \mathbb{F}_l)$. Let $\text{Mat}(r, s, \mathbb{F}_l, M)$ to be the subset of $\text{Mat}(r, s, \mathbb{F}_l)$ consisting of those matrices $A$ satisfying $A(i, j) = M(i, j)$ for all $1 \leq i, j \leq k$. Then we set
\[
Q(r, s, l, M, j) := \frac{\left|\left\{A \in \text{Mat}(r, s, \mathbb{F}_l, M) : \text{dim}(\text{ker}(A)) = j\right\}\right|}{\left|\text{Mat}(r, s, \mathbb{F}_l, M)\right|}.
\]
We are interested in the difference $P(r, r - 1, l, j) - Q(r, r - 1, l, M, j)$ as $r$ goes to infinity, independent of the choice of $M$.

\begin{lemma}
\label{lRM}
Suppose that $r \geq 2k$. We have
\[
\left|P(r, r - 1, l, j) - Q(r, r - 1, l, M, j)\right| \leq 2k \cdot l^{2k - r}.
\]
\end{lemma}

\begin{proof}
For a matrix $A$, let $a_1, \ldots, a_k$ denote its first $k$ columns. We define
\[
P(r, s, l, j, k) := \frac{\left|\left\{A \in \text{Mat}(r, s, \mathbb{F}_l) : \text{dim}(\text{ker}(A)) = j \text{ and } \text{dim}(a_1, \ldots, a_k) = k\right\}\right|}{\left|\text{Mat}(r, s, \mathbb{F}_l)\right|}
\]
and
\[
Q(r, s, l, M, j, k) := \frac{\left|\left\{A \in \text{Mat}(r, s, \mathbb{F}_l, M) : \text{dim}(\text{ker}(A)) = j \text{ and } \text{dim}(a_1, \ldots, a_k) = k\right\}\right|}{\left|\text{Mat}(r, s, \mathbb{F}_l, M)\right|}.
\]
Then we have
\[
P(r, r - 1, l, j) - P(r, r - 1, l, j, k) \leq 1 - P(r, k, l, 0).
\]
and
\[
Q(r, r - 1, l, M, j) - Q(r, r - 1, l, M, j, k) \leq 1 - P(r - k, k, l, 0)
\]
due to our assumption $r \geq 2k$. We observe that $P(r, r - 1, l, j, k) = Q(r, r - 1, l, M, j, k)$. Combining this with the previous two inequalities gives 
\begin{align}
\label{ePrr}
\left|P(r, r - 1, l, j) - Q(r, r - 1, l, M, j)\right| \leq 2 - 2P(r - k, k, l, 0).
\end{align}
Using the classical formula for $P(r - k, k, l, 0)$, we obtain
\begin{align*}
P(r - k, k, l, 0) &= \frac{\prod_{j = 0}^{k - 1} \left(l^{r - k} - l^j\right)}{l^{k(r - k)}} = \prod_{j = 0}^{k - 1} \left(1 - l^{j + k - r}\right) \\
&\geq \left(1 - l^{2k - r}\right)^k \geq 1 - k \cdot l^{2k - r},
\end{align*}
where the last inequality follows from Bernouilli's inequality. Inserting this in (\ref{ePrr}) ends the proof of our theorem.
\end{proof}

With this lemma we have done all the preparatory work needed for understanding the $(1 - \zeta_l)^2$-rank when $K$ varies in $\text{Field}(N, l)$. Recall that we have defined a matrix $\text{Redei}(K)$ in Section \ref{Redei matrices}. It will be useful to observe that an assignment $a : M \rightarrow \langle \zeta_l \rangle$ uniquely determines a Redei matrix and vice versa. Proposition \ref{First description of Redei matrices} implies that the $(1 - \zeta_l)^2$-rank of $K$ is equal to $r - 1 - \text{rank}_{\mathbb{F}_l} \ \text{Redei}(K)$, see also Theorem 1 in \cite{Stevenhagen}. Our next theorem is similar to Theorem 4 in Klys \cite{Klys}, but has the benefit of providing an error term.

\begin{theorem}
\label{tKlys}
Assume GRH and let $l$ be an odd prime. Let $\textup{Field}(N, l, j)$ be the subset of $\textup{Field}(N, l)$ consisting of those fields $K$ with $(1 - \zeta_l)^2$-rank equal to $j$. Then we have
\[
\left|\lim_{s \rightarrow \infty} P(s, s - 1, l, j) \cdot \left|\textup{Field}(N, l)\right| - \left|\textup{Field}(N, l, j)\right|\right| = O\left(\frac{\left|\textup{Field}(N, l)\right|}{(\log \log N)^{c}}\right)
\]
for some absolute constant $c > 0$.
\end{theorem}

\begin{proof}
By Theorem \ref{tRanger} we know that almost all $r := \omega(D_K)$ satisfy (\ref{er3}). Hence it suffices to prove that there exists an absolute constant $c > 0$ with
\[
\left|\lim_{s \rightarrow \infty} P(s, s - 1, l, j) \cdot \left|\text{Field}(N, r, l)\right| - \left|\text{Field}(N, r, l, j)\right|\right| = O\left(\frac{\left|\text{Field}(N, r, l)\right|}{(\log \log N)^c}\right),
\]
where $\text{Field}(N, r, l, j)$ is defined in the obvious way. An easy computation shows that we may replace $\lim_{s \rightarrow \infty} P(s, s - 1, l, j)$ with $P(r, r - 1, l, j)$. Put
\[
D_1 := \log N, \quad C_0 := \frac{1}{10} \log \log \log N.
\]
Let $W$ be the subset of $S_r(N, l)$ that is comfortably spaced above $D_1$ and $C_0$-regular. By Proposition \ref{p6.9} and Theorem \ref{tSNL} it is enough to show
\[
\left|P(r, r - 1, l, j) \cdot \left|\text{Field}(X)\right| - \left|\text{Field}(X) \cap \text{Field}(N, r, l, j)\right|\right| = O\left(\frac{\left|\text{Field}(X)\right|}{(\log \log N)^c}\right)
\]
for all boxes $X$ with $j(\text{Field}(X)) \cap W \neq \emptyset$ and $\text{Field}(X) \subseteq \text{Field}(N, r, l)$. Let $X$ be such a box and write $x_1, \ldots, x_k$ for the unique elements of $X_1, \ldots, X_k$. Fix a function $g : \{x_1, \ldots, x_k\} \rightarrow [l - 1]$. Then we have the identity
\begin{align}
\label{eIFx}
\left|\text{Field}(X)\right| = \frac{1}{W(X)} \sum_f \left|i_f(X)\right|,
\end{align}
where $W(X)$ is a weight depending only on $X$ and the sum is taken over all $f$ in the set $\text{Map}(X_{k + 1} \coprod \ldots \coprod X_r, [l - 1])$. We implicitly extend $f$ to $\text{Map}(X_1 \coprod \ldots \coprod X_r, [l - 1])$ using our function $g$. We have another identity
\begin{align}
\label{eIFx2}
\left|\text{Field}(X) \cap \text{Field}(N, r, l, j)\right| = \frac{1}{W(X)} \sum_f \left|i_f(X) \cap \text{Field}(N, r, l, j)\right|.
\end{align}
Due to (\ref{eIFx}) and (\ref{eIFx2}) it suffices to establish
\begin{align}
\label{eIFx3}
\frac{1}{W(X)} \sum_f \left|P(r, r - 1, l, j) \cdot \left|i_f(X)\right| - \left|i_f(X) \cap \text{Field}(N, r, l, j)\right|\right| = O\left(\frac{\left|\text{Field}(X)\right|}{(\log \log N)^c}\right).
\end{align}
Define $g(l, \emptyset, k, j)$ to be the number of functions $a$ that satisfy the following two properties
\begin{itemize}
\item $a$ agrees with $X$ at stage $k$;
\item the Redei matrix $A$ associated to $a$ has kernel of rank $j$.
\end{itemize}
Then we claim
\begin{align}
\label{eMatrixCount}
\left|\frac{g(l, \emptyset, k, j)}{g(l, \emptyset, k)} - P(r, r - 1, l, j)\right| = O\left(\frac{1}{\sqrt{\log N}}\right).
\end{align}
We have
\begin{align}
\label{eSumg1}
g(l, \emptyset, k) = l^{r^2  - r - k^2 + k}.
\end{align}
Let $M \in \text{Mat}\left(k, k, \mathbb{F}_l\right)$ be such that $M$ agrees with $X$, i.e. we have for all $1 \leq i, j \leq k$ with $i \neq j$ the equality
\[
\zeta_l^{M(i, j)} = \chi_{x_j}\left(\text{Frob}(x_i)\right),
\]
where $x_i$ and $x_j$ are the unique elements of $X_i$ and $X_j$ respectively. Then we have
\begin{align}
\label{eSumg2}
g(l, \emptyset, k, j) = \sum_{\substack{M \in \text{Mat}\left(k, k, \mathbb{F}_l\right) \\ M \text{ agrees with } X}} Q(r, r - 1, l, M, j) l^{r^2 - r - k^2}.
\end{align}
We combine (\ref{eSumg1}) and (\ref{eSumg2}) to deduce
\begin{align}
\label{eSumg3}
\frac{g(l, \emptyset, k, j)}{g(l, \emptyset, k)} = \frac{1}{l^k} \sum_{\substack{M \in \text{Mat}\left(k, k, \mathbb{F}_l\right) \\ M \text{ agrees with } X}} Q(r, r - 1, l, M, j).
\end{align}
Since our box $X$ is $C_0$-regular, we are in the position to apply Lemma \ref{lRM} to the sum in equation (\ref{eSumg3}). Using once more that $X$ is $C_0$-regular, we see that $k$ is roughly equal to $\log r$. Hence we can fit the difference $|P(r, r - 1, l, M, j) - Q(r, r - 1, l, M, j)|$ in the error of  (\ref{eMatrixCount}), thus establishing (\ref{eMatrixCount}). Because of equation (\ref{eIFx3}) and (\ref{eMatrixCount}) we are left to prove
\[
\frac{1}{W(X)} \sum_f \left|\frac{g(l, \emptyset, k, j)}{g(l, \emptyset, k)} \cdot \left|i_f(X)\right| - \left|i_f(X) \cap \text{Field}(N, r, l, j)\right|\right| = O\left(\frac{\left|\text{Field}(X)\right|}{(\log \log N)^c}\right).
\]
We observe that $i_f(X)$ is equal to the disjoint union of $X(a, f)$ over $a$. An application of Theorem \ref{t4r} finishes the proof.
\end{proof}

\section{Proof of Theorem \ref{tCyclic}}
\label{sMain}
The goal of this section is to prove Theorem \ref{tCyclic}, which will follow from a combination of Theorem \ref{tMin}, Theorem \ref{tAgree} and Proposition \ref{p4.4}. Unfortunately, these results are only valid under very strong conditions. Hence most of the work in this section are reduction steps. Before we start the proof of Theorem \ref{tCyclic}, we need a definition.

\begin{mydef}
\label{dNiceBox}
Let $N$ be a large real and let $X$ be a box. Put
\[
D_1 := e^{(\log \log N)^2}, \quad C_0 := \frac{1}{10} \log \log \log N.
\]
Define $W$ to be the subset of $S_r(N, l)$ that is comfortably spaced above $D_1$ and $C_0$-regular. We say that $X$ is a nice box for $N$ if the following three conditions are satisfied
\begin{itemize}
\item $r$ satisfies (\ref{eRanger});
\item $j(\textup{Field}(X)) \cap W \neq \emptyset$;
\item $\textup{Field}(X) \subseteq \textup{Field}(N, l)$.
\end{itemize}
\end{mydef}

\begin{prop}
\label{p7.3}
Assume GRH and let $l$ be an odd prime. There are $c, A, N_0 > 0$ such that for all $N > N_0$, all nice boxes $X$ for $N$, all integers $m \geq 2$ and all sequences $n_2 \geq \ldots \geq n_{m + 1} \geq 0$ of integers, we have
\[
\left|\left|\textup{Field}(X) \cap \bigcap_{k = 2}^{m + 1} D_{l, k}(n_k)\right| - P(n_{m + 1} | n_m) \cdot \left|\textup{Field}(X) \cap \bigcap_{k = 2}^m D_{l, k}(n_k)\right|\right| \leq \frac{A\left|\textup{Field}(X)\right|}{\left(\log \log N\right)^{\frac{c}{m^2(l^2 + l)^m}}}.
\]
\end{prop}

\begin{proof}[Proof that Proposition \ref{p7.3} implies Theorem \ref{tCyclic}.] From Theorem \ref{tRanger} it follows that we need only consider $\text{Field}(N, r, l)$ with $r$ satisfying (\ref{eRanger}). Now apply Proposition \ref{p6.9} with $W$ as in Definition \ref{dNiceBox} and use the lower bound for $W$ established in Theorem \ref{tSNL}.
\end{proof}

For the remainder of this paper $a$ will always denote an assignment from $M$ to $\langle \zeta_l \rangle$. We let $A$ be the Redei matrix associated to $a$, i.e. $A$ is the unique matrix with entries $a(i, j)$ and the property
\[
\prod_{j = 1}^r a(i, j) = 1
\]
for all $1 \leq i \leq r$, which uniquely specifies $a(i, i)$. In this section it is essential to keep track of the characters we have chosen. If $S$ is a subset of $[r]$, we define 
$$
\text{Ch}(S) := \text{Map}\left(\coprod_{i \in S} X_i, [l - 1]\right).
$$
Furthermore, we set
\[
W(X, S) := \left|\{f \in \text{Ch}(S) : K \in i_f(X)\}\right|,
\]
where $K$ is any field in $i_f(X)$. Note that this does not depend on the choice of $K$.

\begin{mydef}
Let $V$ be the $\mathbb{F}_l$-vector space $\mathbb{F}_l^r$, which we think of as column vectors. Given the assignment $a: M \rightarrow \langle \zeta_l \rangle$ and associated Redei matrix $A$, we define
\[
D_{a, 2} := \{v \in V : v^TA = 0\}, \quad D_{a, 2}^\vee := \{v \in V : Av = 0\},
\]
where we think of $A$ as having entries in $\mathbb{F}_l$ through the isomorphism $j_l^{-1}$. Put
\[
n_{\textup{max}} := \left\lfloor \sqrt{\frac{c}{m^2(l^2 + l)^m} \log \log \log N} \right\rfloor, \quad n_2 := -1 + \dim_{\mathbb{F}_l} D_{a, 2}
\]
with $c$ a small constant depending only on $l$. Define $R := (1, \ldots, 1) \in \mathbb{F}_l^r$ and 
\[
\alpha := \left|\left\{j \in [r] : \frac{r}{4} \leq j \leq \frac{r}{3}\right\}\right|.
\]
We say that the assignment $a: M \rightarrow \langle \zeta_l \rangle$ is generic if the following conditions are satisfied
\begin{itemize}
\item $n_2 \leq n_{\textup{max}}$;
\item we have for all $i \in \mathbb{F}_l$, for all $T_1 \in D_{a, 2}$ and all $T_2 \in D_{a, 2}^\vee$ such that $T_1 \neq 0$ or $T_2 \not \in \langle R \rangle$
\begin{align}
\label{eGeneric}
\left|\left|\left\{j \in [r] : \frac{r}{4} \leq j \leq \frac{r}{3} \text{ and } \pi_j(T_1 + T_2) = i\right\}\right| - \frac{\alpha}{l}\right| \leq 2^{-10n_{\textup{max}}} \cdot r,
\end{align}
where $\pi_j$ is the projection on the $j$-th coordinate.
\end{itemize}
\end{mydef}

During our proof we will fix all previous Artin pairings, and then prove that the $m$-th Artin pairing is equidistributed. We formalize this in the following definition.

\begin{mydef}
Fix an assignment $a: M \rightarrow \langle \zeta_l \rangle$. Let $m \geq 2$ and choose filtrations of $\mathbb{F}_l$-vector spaces
\[
D_{a, 2} \supseteq \ldots \supseteq D_{a, m}, \quad D_{a, 2}^\vee \supseteq \ldots \supseteq D_{a, m}^\vee
\]
with $R \in D_{a, m}^\vee$. For $2 \leq k \leq m$ we define an integer $n_k$ by
\[
n_k := -1 + \dim_{\mathbb{F}_l} D_{a, k}.
\]
For $2 \leq k < m$, choose a bilinear pairing
\[
\textup{Art}_k : D_{a, k} \times D_{a, k}^\vee \rightarrow \mathbb{F}_l
\]
with left kernel $D_{a, k + 1}$ and right kernel $D_{a, k + 1}^\vee$. We call the set $\left\{\textup{Art}_k\right\}_{2 \leq k < m}$ a sequence of valid Artin pairings. Given a sequence of valid Artin pairings, we define
\[
X(a, f, i) := \left\{x \in X(a, f) : \textup{the Artin pairing of } x \textup{ agrees with } \left\{\textup{Art}_k\right\}_{2 \leq k \leq i}\right\}.
\]
\end{mydef}

\begin{prop}
\label{p7.4}
Assume GRH and let $l$ be an odd prime. There are $c, A, N_0 > 0$ such that for all $N > N_0$, all nice boxes $X$ for $N$, all generic assignments $a: M \rightarrow \langle \zeta_l \rangle$ that agree with $X$, all integers $m \geq 2$, all sequences of valid Artin pairings $\left\{\textup{Art}_k\right\}_{2 \leq k < m}$ and a valid Artin pairing $\textup{Art}_m$, we have with $S := [r] - [k]$
\begin{multline}
\frac{1}{W(X, S)} \sum_{f \in \textup{Ch}(S)} \left|\left|X(a, f, m)\right| - l^{-n_m(n_m + 1)} \cdot \left|X(a, f, m - 1)\right|\right| \leq \\
\frac{A}{W(X, S)} \sum_{f \in \textup{Ch}(S)} \frac{\left|X(a, f)\right|}{\left(\log \log N\right)^{\frac{c}{m(l^2 + l)^m}}}. \nonumber
\end{multline}
\end{prop}

\begin{proof}[Proof that Proposition \ref{p7.4} implies Proposition \ref{p7.3}.] 
Note that $X(a, f)$ is in fact a slight abuse of notation, since this is only defined for $f \in \text{Ch}([r])$. However, one of the assumptions in the proposition statement is that $a$ agrees with $X$, and this involves a choice of function $g$ in $\text{Map}(\{x_1, \ldots, x_k\}, [l - 1])$, where $x_1, \ldots, x_k$ are the unique elements of $X_1, \ldots, X_k$. Hence, whenever we write $f \in \text{Ch}(S)$, we mean the function $f$ extended to $\text{Ch}([r])$ using $g$. We have the identities
\[
\left|\textup{Field}(X) \cap \bigcap_{k = 2}^{m + 1} D_{l, k}(n_k)\right| = \frac{1}{W(X, S)} \sum_{f \in \text{Ch}(S)} \left|i_f(X) \cap \bigcap_{k = 2}^{m + 1} D_{l, k}(n_k)\right|
\]
and
\[
\left|\textup{Field}(X) \cap \bigcap_{k = 2}^m D_{l, k}(n_k)\right| = \frac{1}{W(X, S)} \sum_{f \in \text{Ch}(S)} \left|i_f(X) \cap \bigcap_{k = 2}^m D_{l, k}(n_k)\right|.
\]
Hence the LHS of Proposition \ref{p7.3} is upper bounded by
\begin{align}
\label{eifX}
\frac{1}{W(X, S)} \sum_{f \in \text{Ch}(S)} \left|\left|i_f(X) \cap \bigcap_{k = 2}^{m + 1} D_{l, k}(n_k)\right| - P(n_{m + 1}|n_m) \cdot \left|i_f(X) \cap \bigcap_{k = 2}^m D_{l, k}(n_k)\right|\right|.
\end{align}
Since $i_f(X)$ is the disjoint union of $X(a, f, m)$, the quantity in equation (\ref{eifX}) is at most
\begin{align}
\label{eSplita}
\frac{1}{W(X, S)} \sum_a \sum_{f \in \text{Ch}(S)} \sum_{\left\{\textup{Art}_k\right\}_{2 \leq k \leq m}} \left|\left|X(a, f, m)\right| - l^{-n_m(n_m + 1)} \cdot \left|X(a, f, m - 1)\right|\right|,
\end{align}
where the first sum is over all $a$ that agree with $X$ and the last sum is over all sequences of valid Artin pairings with the property 
\[
\text{dim}_{\mathbb{F}_l} D_{a, k} = n_k + 1
\]
for all $2 \leq k \leq m + 1$. We split the sum in equation (\ref{eSplita}) in two parts depending on the genericity of $a$. If $a$ is generic, we use Proposition \ref{p7.4} to bound the sum. There are at most 
\[
l^{mn_2(n_2 + 1)} \leq l^{mn_{\text{max}}(n_{\text{max}} + 1)}
\]
sequences of valid Artin pairings, so the sum is within the error of Proposition \ref{p7.3}.  If $a$ is not generic, we employ the trivial bound to (\ref{eSplita}) inducing an error of size at most
\begin{align}
\label{eErrorTerm}
\frac{2}{W(X, S)} \sum_{a \text{ not generic}} \sum_{f \in \text{Ch}(S)} |X(a, f)|.
\end{align}
The sum over $f$ is easily bounded by Theorem \ref{t4r}. So it remains to count the number of $a$ that agree with $X$ and are not generic, which is a purely combinatorial problem. We first deal with the $a$ for which $n_2 > n_{\text{max}}$. These $a$ are easily bounded using the ideas from the proof of Theorem \ref{tKlys}.

Now consider the assignments $a : M \rightarrow \langle \zeta_l \rangle$ not satisfying equation (\ref{eGeneric}). We have to estimate
\[
\frac{\left|\left\{a \text{ assignment} : a \text{ agrees with } X \text{ and } a \text{ fails equation } (\ref{eGeneric})\right\}\right|}{\left|\left\{a \text{ assignment} : a \text{ agrees with } X\right\}\right|},
\]
which is clearly upper bounded by
\begin{align}
\label{eMatrixSpace}
2^{k^2} \frac{\left|\left\{a \text{ assignment} : a \text{ fails equation } (\ref{eGeneric})\right\}\right|}{\left|\left\{a \text{ assignment}\right\}\right|}.
\end{align}
We first count the number of pairs $(T_1, T_2)$ that fail equation (\ref{eGeneric}) with $T_1 \neq 0$ and $T_2$ linearly independent from $R$. From Hoeffding's inequality we deduce that the proportion of such pairs is at most
\[
O\left(e^{-r \cdot 2^{-20n_{\text{max}}}}\right).
\]
Now observe that the number of $a$ for which $T_1 \in D_{a, 2}$ and $T_2 \in D_{a, 2}^\vee$ does not depend on the pair $(T_1, T_2)$ provided that $T_1 \neq 0$ and that $T_2$ is linearly independent from $R$. Hence we get the desired upper bound for equation (\ref{eMatrixSpace}). We still need to deal with the case $T_1 = 0$ and $T_2 \in \langle R \rangle$. In both cases we apply Hoeffding's inequality once more, and proceed along the same lines. This proves the proposition.
\end{proof}

For generic $a$, our next goal is to find sets $S$ for which we can apply Theorem \ref{tMin} and Theorem \ref{tAgree}. Following Smith \cite{Smith}, we call such sets $S$ variable indices.

\begin{mydef}
Let $a: M \rightarrow \langle \zeta_l \rangle$ be an assignment and let $m \geq 2$ be an integer. For the rest of this paper, fix a basis $w_{2, 1}, \ldots, w_{2, n_2 + 1}$ for $D_{a, 2}$ and fix a basis $w_{1, 1}, \ldots, w_{1, n_2}, R$ for $D_{a, 2}^\vee$ in such a way that for all $2 \leq k \leq m$, $w_{2, 1}, \ldots, w_{2, n_k + 1}$ is a basis for $D_{a, k}$ and $w_{1, 1}, \ldots, w_{1, n_k}, R$ is a basis for $D_{a, k}^\vee$. Let $1 \leq j_1, j_2 \leq n_m$ be integers. We say that $S(j_1, j_2) \subseteq [r]$ is a set of variable indices for $(j_1, j_2)$ if there are integers $i_1(j_1, j_2)$ and $i_2(j_1, j_2)$ with the following properties
\begin{itemize}
\item $|S(j_1, j_2)| = m + 1$;
\item $i_1(j_1, j_2), i_2(j_1, j_2) \in S(j_1, j_2)$;
\item $S(j_1, j_2)$ lies in the zero set of all $w_{1, j}$ for all $j \leq n_2$ other than $j_1$ or $j_2$;
\item $S(j_1, j_2)$ lies in the zero set of all $w_{2, j}$ for all $j \leq n_2 + 1$ other than $j_1$ or $j_2$;
\item $S(j_1, j_2)$ lies in the zero set of $w_{2, j_1}$ and $w_{1, j_2}$;
\item $\{i \in [r] : \pi_i(w_{1, j_1}) \neq 0\} \cap S(j_1, j_2) = \{i_1(j_1, j_2)\}$;
\item $\{i \in [r] : \pi_i(w_{2, j_2}) \neq 0\} \cap S(j_1, j_2) = \{i_2(j_1, j_2)\}$.
\end{itemize}
\end{mydef}

With this definition in place, we are ready to find variable indices for generic $a$. We do so with the following lemma.

\begin{lemma}
\label{lVariable}
Let $a: M \rightarrow \langle \zeta_l \rangle$ be a generic assignment. If $w_1, \ldots, w_d \in D_{a, 2}$ are linearly independent and also $w_{d + 1}, \ldots, w_e, R \in D_{a, 2}^\vee$ are linearly independent, we have for all $\mathbf{v} \in \mathbb{F}_l^e$
\[
\left|\left|\left\{i \in [r] : \frac{r}{4} \leq i \leq \frac{r}{3} \text{ and } \pi_i(w_j) = \pi_j(\mathbf{v}) \text{ for all } 1 \leq j \leq l\right\}\right| - \frac{\alpha}{l^e}\right| \leq 100^e \cdot 2^{-10n_{\textup{max}}} \cdot r.
\]
\end{lemma}

\begin{proof}
The case $e = 1$ follows easily from the genericity condition on $a$. We start with the special case $d = 1$, $e = 2$. Define for $\mathbf{x} \in \mathbb{F}_l^2$
\[
g(\mathbf{x}) = \left|\left\{i \in [r] : \frac{r}{4} \leq i \leq \frac{r}{3} \text{ and } \pi_i(w_j) = \pi_j(\mathbf{x}) \text{ for all } 1 \leq j \leq 2\right\}\right|.
\]
We use equation (\ref{eGeneric}) with the pair $(T_1, T_2)$ in the following set
\[
\{(w_1, \beta w_2) : \beta \in \mathbb{F}_l\} \cup \{(0, w_2)\}.
\]
Then, if $\mathbf{x}_1, \dots, \mathbf{x}_l$ lie on an affine line in $\mathbb{F}_l^2$, we have
\begin{align}
\label{eAffine}
\left|\left(\sum_{i = 1}^l g(\mathbf{x}_i)\right) - \frac{\alpha}{l}\right| \leq 2^{-10n_{\textup{max}}} \cdot r.
\end{align}
Now take an element $\mathbf{v} \in \mathbb{F}_l^2$. Let $L(\mathbf{v})$ be the collection of affine lines in $\mathbb{F}_l^2$ through $\mathbf{v}$. We use equation (\ref{eAffine}) for all elements in $L(\mathbf{v})$ to deduce the following inequality
\[
l\left|g(\mathbf{v}) - \frac{\alpha}{l^2}\right| = \left|\sum_{L(\mathbf{v})} \sum_{\mathbf{x} \in L(\mathbf{v})} \left(g(\mathbf{x}) - \frac{\alpha}{l}\right)\right| \leq \sum_{L(\mathbf{v})} \left|\sum_{\mathbf{x} \in L(\mathbf{v})} \left(g(\mathbf{x}) - \frac{\alpha}{l}\right)\right| \leq (l + 1) \cdot 2^{-10n_{\textup{max}}} \cdot r.
\]
Since the special cases $d = 0$, $e = 2$ and $d = e = 2$ are trivial, this settles the case $e = 2$. An easy induction establishes the lemma for all $e > 2$.
\end{proof}

Fix two integers $1 \leq j_1, j_2 \leq n_m$. We will now demonstrate how to find variable indices $S(j_1, j_2)$ for $(j_1, j_2)$ using Lemma \ref{lVariable}. We apply Lemma \ref{lVariable} with $w_{2, 1}, \ldots, w_{2, n_2 + 1} \in D_{a, 2}$ and $w_{1, 1}, \ldots, w_{1, n_2}, R \in D_{a, 2}^\vee$, so $d = n_2 + 1$ and $e = 2n_2 + 1$. We let $\mathbf{v} \in \mathbb{F}_l^e$ be the unique vector satisfying $\pi_j(\mathbf{v}) = 1$ for $j = j_1$ and $j = d + j_2$, and furthermore $\pi_j(\mathbf{v}) = 0$ for all other $j$. With these choices we can choose $S(j_1, j_2)$ to be any subset of
\[
\left|\left|\left\{i \in [r] : \frac{r}{4} \leq i \leq \frac{r}{3} \text{ and } \pi_i(w_j) = \pi_j(\mathbf{v}) \text{ for all } 1 \leq j \leq l\right\}\right| - \frac{\alpha}{l^e}\right|,
\]
provided that $|S(j_1, j_2)| = m + 1$. Then we must have
\begin{align}
\label{eBoundm}
m + 1 \leq C \cdot \frac{r}{l^{2n_2}}
\end{align}
for some small constant $C > 0$ depending only on $l$. We always have 
$$
m < \log \log \log \log N,
$$
since otherwise Theorem \ref{tCyclic} is trivial. If $n_{\text{max}}$ is sufficiently small, this implies (\ref{eBoundm}). Having found our variable indices, we are ready for our next reduction step.

\begin{prop}
\label{pChar}
Assume GRH and let $l$ be an odd prime. There are $c, A, N_0 > 0$ such that for all $N > N_0$, all nice boxes $X$ for $N$, all generic assignments $a: M \rightarrow \langle \zeta_l \rangle$ that agree with $X$, all integers $m \geq 2$, all sequences of valid Artin pairings $\left\{\textup{Art}_k\right\}_{2 \leq k < m}$, a valid Artin pairing $\textup{Art}_m$ and a non-zero multiplicative character $F$ from $\textup{Mat}(n_m + 1, n_m, \mathbb{F}_l)$ to $\langle \zeta_l \rangle$, we have
\[
\frac{1}{W(X, S)} \sum_{f \in \textup{Ch}(S)} \left|\sum_{x \in X(a, f, m - 1)} F(\textup{Art}(x, f, m)) \right| \leq \frac{A}{W(X, S)} \sum_{f \in \textup{Ch}(S)} \frac{\left|X(a, f)\right|}{\left(\log \log N\right)^{\frac{c}{m(l^2 + l)^m}}},
\]
where $S := [r] - [k]$ and $\textup{Art}(x, f, m)$ is the $m$-th Artin pairing of the field $i_f(x)$.
\end{prop}

\begin{proof}[Proof that Proposition \ref{pChar} implies Proposition \ref{p7.4}.]
This is straightforward.
\end{proof}

Before we can make our final reduction step, we must restrict ourselves to rather special product spaces. Our next definition will make this precise.

\begin{mydef}
Let $a : M \rightarrow \langle \zeta_l \rangle$ be a generic assignment. Define $S_{\textup{var}} := S(j_1, j_2) - \{i_2(j_1, j_2)\}$. For each $i \in S_{\textup{var}}$, let $Z_i$ be subsets of $X_i$ of cardinality equal to
\[
M_{\textup{box}} := \left\lfloor (\log \log N)^{\frac{1}{10m}} \right\rfloor,
\]
which is at least $l$ for sufficiently large $N$. Put
\[
Z := \prod_{i \in S_{\text{var}}} Z_i.
\]
We say that $Z$ is well-governed if there is a governing expansion $\mathfrak{G}$ on $Z$ such that $i_1(j_1, j_2) = i_a(\mathfrak{G})$ and furthermore all $\bar{x} \in \overline{Z}_{S_{\textup{var}}}$ satisfy $\bar{x} \in \overline{Y}_{S_{\textup{var}}}(\mathfrak{G})$. Set
\[
M(Z) := \prod_{\bar{x} \in \overline{Z}_{S_{\textup{var}}}} \phi_{S_{\textup{var}}, \bar{x}}, \quad M_\circ(Z) := \prod_{S \subsetneq S_{\textup{var}}} \prod_{\bar{x} \in \overline{Z}_S} \phi_{S, \bar{x}}.
\]
Define
\[
M^\circ(Z) := L(Z \cup \{x_1, \ldots, x_k\}) M_\circ(Z).
\]
For $i > k$ not in $S_{\textup{var}}$, we define $X_i(M^\circ(Z))$ to be the subset consisting of the $p \in X_i$ satisfying
\begin{itemize}
\item $p$ splits completely in the extension $M_\circ(Z)/\Q$;
\item we have for all $j \in [k] \cup S_{\textup{var}}$ and all $x \in X_j$ 
\[
\chi_x^{f(x)}(\textup{Frob}(p)) = a(i, j), \quad \chi_p^{f(p)}(\textup{Frob}(x)) = a(j, i).
\]
\end{itemize}
Take
\[
\widetilde{Z} := \prod_{1 \leq i \leq k} X_i \times Z \times \prod_{\substack{i > k \\ i \not \in S_{\textup{var}}}} X_i(M^\circ(Z)).
\]
We call $\widetilde{Z}$ a satisfactory product space if the following three conditions are satisfied
\begin{itemize}
\item $Z$ is well-governed;
\item $x_1, \ldots, x_k$ split completely in the extension $M_\circ(Z)/\Q$;
\item for all distinct $i, j \in S_{\textup{var}}$, for all $x_i \in \pi_i(Z)$ and for all $x_j \in \pi_j(Z)$ we have
\[
\chi_{x_j}^{f(x_j)}(\textup{Frob}(x_i)) = a(i, j).
\]
\end{itemize}
\end{mydef}

We remind the reader that our equidistribution comes from a combination of Theorem \ref{tMin}, Theorem \ref{tAgree} and Proposition \ref{p4.4}. We will apply these results to satisfactory product spaces $\widetilde{Z}$. In order to do so, we need to construct an $l$-additive system $\mathfrak{A}$ on $\widetilde{Z}$ that satisfies the conditions of Proposition \ref{p4.4}. This is done in our next lemma.

\begin{lemma}
\label{lAddCon}
Let $l$ be an odd prime and let $\widetilde{Z}$ be a satisfactory product space. Let $P$ be an element of $\pi_{[r] - S(j_1, j_2)}\left(\widetilde{Z}\right)$ and define
\[
\widetilde{Z}(P) := \{P\} \times Z \times X_{i_2(j_1, j_2)}(M^\circ(Z)).
\]
Let $F$ be a non-zero multiplicative character from $\textup{Mat}(n_m + 1, n_m, \mathbb{F}_l)$ to $\langle \zeta_l \rangle$ that depends on the entry $(j_1, j_2)$. There exists an $l$-additive system $\mathfrak{A}$ on $\widetilde{Z}(P)$ with the following properties
\begin{itemize}
\item $\overline{Y}_\emptyset^\circ(\mathfrak{A}) = X(a, f, m - 1) \cap \widetilde{Z}(P)$;
\item $\mathfrak{A}$ is $S(j_1, j_2)$-acceptable, see Definition \ref{dAcceptable}, with $\left|A_T(\mathfrak{A})\right|$ bounded by $l^{n_2(n_2 + m + 1)}$;
\item for all $\bar{x} \in \overline{Z}_{S(j_1, j_2)}(\mathfrak{A})$, see once more Definition \ref{dAcceptable}, we have
\[
d\widetilde{F}(\bar{x}) = j_l^{-1}(F(j_1, j_2)) \cdot \pi_{i_2(j_1, j_2)}(w_{2, j_2}) \cdot \phi_{S_{\textup{var}}, \bar{z}}\left(\textup{Frob}(p_1) \cdot \ldots \cdot \textup{Frob}(p_l)\right),
\]
where $p_i := \pr_i\left(\pi_{i_2(j_1, j_2)}(\bar{x})\right)$ for $i \in [l]$, $\widetilde{F}(x) := F(\textup{Art}(x, f, m))$ and $\bar{z}$ is any element of $\bar{x}(S_{\textup{var}})$.
\end{itemize}
\end{lemma}

\begin{proof}
We will start by constructing an $l$-additive system $\mathfrak{A}$ and then verify its required properties. To do this, we need to introduce the concept of acceptable ramification, which is based on Smith \cite[p.\ 33]{Smith}. We write $W$ for $\widetilde{Z}(P)$. If $w \in D_{a, 2}^\vee$, we define a raw cocycle $\mathfrak{R}(w)$ for $(W, f)$ to be a choice of raw cocycle for each $x \in W$ such that the $\psi_1$ of this raw cocycle is equal to the character naturally associated to $w$. Here and later we shall often suppress the dependence on $f$. Now choose a raw cocycle $\mathfrak{R}(w_{1, j})$ for $(W, f)$, where $j$ runs through $1, \ldots, n_2$. Define $M$ to be the compositum of all the $L(\psi_k(\mathfrak{R}, x))$, where $\mathfrak{R}$ is one of the raw cocycles $\mathfrak{R}(w_{1, j})$, $x \in W$ and $k \leq \min(m, \text{rk}(\mathfrak{R})(x))$. 

Let $M'$ be the compositum of $M$ with $\phi_{S_{\text{var}}, \bar{x}}$ for all $\bar{x} \in \overline{W}_{S_{\text{var}}}$. If $p$ ramifies in $M'$, then the ramification degree of $p$ is $l$, and we choose an element $\sigma_p \in \Gal(M'/\Q)$ such that $\langle\sigma_p\rangle$ is the inertia group of some prime dividing $p$ in $M'$. We assume that the $\phi_{S, \bar{x}}$ are normalized in such a way that $\phi_{S, \bar{x}}(\sigma_p) = 0$ for all $p$ ramifying in $M$. Here and for the rest of the proof, $\phi_{S, \bar{x}}$ is defined to be the zero map if $S$ does not contain $i_1(j_1, j_2)$. Let $S$ be a subset of $S(j_1, j_2)$, $\bar{x} \in \overline{W}_S$ and suppose that 
\[
\text{rk}\left(\mathfrak{R}\left(w_{1, j}\right)\right) \geq |S| + 1 \text{ for all } x \in \bar{x}(\emptyset).
\]
We say that $\mathfrak{R}(w_{1, j})$ is acceptably ramified at $(\bar{x}, i)$ for $i \in S_{\text{var}} - S$ if
\[
\sum_{x \in \bar{x}(\emptyset)} \psi_{|S| + 1}(\mathfrak{R}(w_{1, j}), x)\left(\sigma_{\pi_i(\bar{x})}\right) = 0.
\]
We inductively define a subset $\overline{Y}_S^\circ(\mathfrak{A})$ in $\overline{Y}_S(\mathfrak{A})$. If $S = \emptyset$, we have already done this by the first property of $\mathfrak{A}$ in the lemma. Now suppose that $S$ is a subset of $S_{\text{var}}$ of cardinality at most $|S_{\text{var}}| - 1$. If $S$ does not satisfy these two conditions, we let $F_S(\mathfrak{A})$ be the zero map and $\overline{Y}_S^\circ(\mathfrak{A}) = \overline{Y}_S(\mathfrak{A})$.

Let $\bar{x} \in \overline{Y}_S(\mathfrak{A})$. If $j \neq i_1(j_1, j_2)$, we know that 
\[
\psi\left(\mathfrak{R}(w_{1, j}), \bar{x}\right) := \sum_{x \in \bar{x}(\emptyset)} \psi_{|S|}(\mathfrak{R}(w_{1, j}), x).
\]
is a character. If instead $j = i_1(j_1, j_2)$, then $\psi\left(\mathfrak{R}\left(w_{1, i_1(j_1, j_2)}\right), \bar{x}\right) - \phi_{S, \bar{x}}$ is a character. We call this character $\chi(\bar{x}, j)$ in both cases. We define $\overline{Y}_S^\circ(\mathfrak{A})$ to be those $\bar{x} \in \overline{Y}_S(\mathfrak{A})$ such that $\chi(\bar{x}, j) = 0$ for all $j$ and furthermore $\mathfrak{R}(w_{1, j})$ is acceptably ramified for all $j$ and $i \in S_{\text{var}} - S$. We will now describe how to encode this as a map $F_S(\mathfrak{A})$ and we do so for each $j$ separately.

The acceptable ramification can be encoded with an additive map to $\mathbb{F}_l^{|S_{\text{var}} - S|}$. To deal with $\chi(\bar{x}, j)$, we remark that the acceptable ramification conditions at the stages $S - \{i\}$ ensures that $\chi(\bar{x}, j)$ is an unramified character above $K_{\chi_{x, f}}$ for all $x \in \bar{x}(\emptyset)$. In particular, we deduce that $\chi(\bar{x}, j)$ is supported outside the primes in $S$.

Let $p$ be a prime ramifying in $K_{\chi_{x, f}}$ and not in $\pi_S(x)$. We claim that the prime above $p$ splits in the extension $K_{\chi(\bar{x}, j)} K_{\chi_{x, f}}/K_{\chi_{x, f}}$. By construction of $\overline{Y}_\emptyset^\circ$ we have 
\[
\text{rk}\left(\mathfrak{R}\left(w_{1, j}\right)\right)(x) > |S| \text{ for each } x \in \bar{x}(\emptyset).
\]
This implies that $p$ has residue field degree $1$ in $\psi_S(\mathfrak{R}(w_{1, j}), x)$. Since $\overline{Y}_\emptyset^\circ$ is also contained in $X(a, f)$ by construction, we conclude that $p$ has residue field degree $1$ in the compositum of the $\psi_S(\mathfrak{R}(w_{1, j}), x)$ over all $x \in \bar{x}(\emptyset)$ and $\phi_{S, \bar{x}}$. Since $K_{\chi(\bar{x}, j)}$ is contained in this compositum, we have proven our claim.

Hence, in order to test if $\chi(\bar{x}, j)$ is zero, we merely have to check that the primes in $\pi_S(x)$ split completely in $K_{\chi(\bar{x}, j)} K_{\chi_{x, f}} / K_{\chi_{x, f}}$ and furthermore that $\chi(\bar{x}, j)$ is zero in the vector space $D_{a, 2}^\vee$. We make this precise in the following way. Fix an $x \in \bar{x}(\emptyset)$ and define the map $F_S(\mathfrak{A})$ that sends $\chi(\bar{x}, j)$ to
\[
\chi(\bar{x}, j)\left(\text{Frob}\left(\text{Up}_{K_{\chi_{x, f}}}(z)\right)\right).
\]
Here $z$ runs through the primes in $\pi_S(x)$ and the integers
\[
\prod_{i = 1}^r \pi_i(x)^{\pi_i(w_{2, j})}
\]
as $j$ varies from $1$ to $n_2 + 1$. The key property is that this does not depend on $x$, which follows from the fact that $\overline{Y}_\emptyset^\circ$ is contained in $X(a, f)$ and $\chi(\bar{x}, j)$ is supported outside the primes in $S$. Once this is established, it is not hard to show that $F_S(\mathfrak{A})$ is additive. This describes our $l$-additive system $\mathfrak{A}$. We will now demonstrate that $\mathfrak{A}$ is $S(j_1, j_2)$-acceptable.

Let $\bar{x} \in \overline{Z}_{S(j_1, j_2)}(\mathfrak{A})$. If $\pi_i(\bar{x}(S(j_1, j_2) - \{i\}))$ consists of only one element for some $i \in S$, the condition $\bar{x}(\emptyset) \subseteq \overline{Y}_\emptyset^\circ(\mathfrak{A})$ is trivially satisfied. So now suppose that
\[
\left|\pi_i(\bar{x}(S(j_1, j_2) - \{i\}))\right| > 1
\]
for all $i \in S$. Let $\bar{z}_{i, m_1}, \ldots, \bar{z}_{i, m_i}$ be a distinct list of elements of $\bar{x}(S(j_1, j_2) - \{i\})$. By assumption we have 
\[
m_i \geq \left|\pi_i(\bar{x}(S(j_1, j_2) - \{i\}))\right| - 1
\]
for all $i \in S$. In case we have strict inequality for some $i \in S$, it is immediate that $\mathfrak{A}$ is $S(j_1, j_2)$-acceptable. So suppose that we have equality for all $i \in S$, and let $x_0$ be the unique element outside all of the $\bar{z}_{i, m}(\emptyset)$. We have to show that $x_0 \in X(a, f, m - 1)$. 

We start by checking that $x_0$ is in $X(a, f)$, so take two distinct indices $i, i' \in S$. Take the $l - 1$ other points, with multiplicity, in the set $\bar{x}(\emptyset)$ having the property
\[
\pi_{[r] - i'}(x_0) = \pi_{[r] - i'}(x).
\]
All these $l - 1$ points are in $X(a, f)$. Now the splitting conditions coming from the existence of $\phi_{S, \bar{x}}$ show that $x_0$ must be in $X(a, f)$ as well. It remains to prove that $\text{Art}(x, f, i)$ is equal to the Artin pairing $\text{Art}_i$ for all $2 \leq i \leq m - 1$.

So take such an $i$ and let $S$ a subset of $S(j_1, j_2)$ with $i + 1$ elements not containing $i_1(j_1, j_2)$. Take any cube $\bar{y} \in \bar{x}(S)$ with $x_0 \in \bar{y}(\emptyset)$. We apply Theorem \ref{tMin} with $\bar{y}$, $S$ and for all raw cocycles $\mathfrak{R}(w_{1, j})$ and all $b$ corresponding to some $w_{2, j'}$. Since $i_1(j_1, j_2) \not \in S$, we have minimality in all cases. From Theorem \ref{tMin} we deduce that $\text{Art}(x_0, f, m)$ is equal to $\text{Art}_i$.

Finally we must check that $\mathfrak{A}$ satisfies the third property listed in the lemma. But this follows from an application of Theorem \ref{tMin} and Theorem \ref{tAgree}.
\end{proof}

\begin{prop}
\label{pFinal}
Assume GRH and let $l$ be an odd prime. There are $c, A, N_0 > 0$ such that for all $N > N_0$, all nice boxes $X$ for $N$, all generic assignments $a: M \rightarrow \langle \zeta_l \rangle$ that agree with $X$, all integers $m \geq 2$, all sequences of valid Artin pairings $\left\{\textup{Art}_k\right\}_{2 \leq k < m}$, a valid Artin pairing $\textup{Art}_m$, a non-zero multiplicative character $F$ from $\textup{Mat}(n_m + 1, n_m, \mathbb{F}_l)$ to $\langle \zeta_l \rangle$ that depends on the entry $(j_1, j_2)$ and all satisfactory produce spaces $\widetilde{Z}$, we have with $S := [r] - [k] - S_{\textup{var}}$
\[
\frac{1}{W(X, S)} \sum_{f \in \textup{Ch}(S)} \left|\sum_{x \in \widetilde{Z} \cap X(a, f, m - 1)} F(\textup{Art}(x, f, m)) \right| \leq \frac{A}{W(X, S)} \sum_{f \in \textup{Ch}(S)} \frac{\left|\widetilde{Z} \cap X(a, f)\right|}{\left(\log \log N\right)^{\frac{c}{m(l^2 + l)^m}}}.
\]
\end{prop}

\begin{proof}[Proof that Proposition \ref{pFinal} implies Proposition \ref{pChar}.]
Define for $i > k$
\[
X_i(a) := \left\{x \in X_i : \chi_x^{f(x)}(\text{Frob}(x_j)) = a(j, i) \text{ and } \chi_{x_j}^{f(x_j)}(\text{Frob}(x)) = a(i, j) \text{ for all } 1 \leq j \leq k\right\}
\]
and
\[
X_{\text{var}} := \prod_{i \in S_{\text{var}}} X_i(a).
\]
We let $V_{\text{var}}$ be the subset of $X_{\text{var}}$ that is consistent with $a$, i.e. $P \in V_{\text{var}}$ if and only if for all distinct $i, j \in S_{\text{var}}$ we have 
$$
\chi_{\pi_j(P)}^{f(\pi_j(P))}(\text{Frob}(\pi_i(P))) = a(i, j).
$$
Set
\[
R := \exp\left(\exp\left(\frac{1}{10} \log \log N\right)\right).
\]
We let $Z_{\text{var}}^1, \ldots, Z_{\text{var}}^t$ be a longest sequence of subsets of $X_{\text{var}}$ with the following properties
\begin{itemize}
\item we have for all $1 \leq s \leq t$
\[
Z_{\text{var}}^s = \prod_{i \in S_{\text{var}}} Z_i^s
\]
for some subset $Z_i^s$ of $X_i(a)$ of cardinality $M_{\text{box}}$;
\item each $Z_{\text{var}}^s$ is a subset of $V_{\text{var}}$, and furthermore any element of $V_{\text{var}}$ is in at most $R$ different $Z_{\text{var}}^s$;
\item for all distinct $s$ and $s'$ we have $\left|Z_{\text{var}}^s \cap Z_{\text{var}}^{s'}\right| \leq 1$;
\item $Z_{\text{var}}^s$ is well-governed and $x_1, \ldots, x_k$ split completely in the extension $M_\circ(Z_{\text{var}}^s)/\Q$.
\end{itemize}
We make the important remark that the sequence $Z_{\text{var}}^1, \ldots, Z_{\text{var}}^t$ depends only on the choice of characters for $X_i$ with $i \in S_{\text{var}}$. Hence if $f, f' \in \text{Ch}([r] - [k])$ restrict to the same function in $\text{Ch}(S_{\text{var}})$, we may and will take the same sequence $Z_{\text{var}}^1, \ldots, Z_{\text{var}}^t$ for $f$ and $f'$.

Define $V_{\text{var}}^{\text{bad}}$ be the subset of points in $V_{\text{var}}$ that are in fewer than $R$ of the $Z_{\text{var}}^s$ and let $\delta$ be the density of $V_{\text{var}}^{\text{bad}}$ in $X_{\text{var}}$. Our goal is to give an upper bound for $\delta$ using the tools from Section \ref{sAdd}. Using a straightforward greedy algorithm, we can find a subset $W$ of $V_{\text{var}}^{\text{bad}}$ of density at least $\delta/RM_{\text{box}}^m$ such that $\left|W \cap Z_{\text{var}}^s\right| \leq 1$ for all $s$.

Our next step is to construct a ``nice'' $l$-additive system $\mathfrak{A}$ on $X_{\text{var}}$ with $\overline{Y}_\emptyset^\circ = W$. Using this $l$-additive system $\mathfrak{A}$ we are going to find a well-governed subset of $W$ if $\delta$ is sufficiently large. Since this is impossible by construction of $Z_{\text{var}}^1, \ldots, Z_{\text{var}}^t$, we obtain the desired upper bound for $\delta$. We will now define the $l$-additive system $\mathfrak{A}$.

First suppose $S$ is a singleton $\{j\}$. Then we define $F_j(\bar{x})$ to be
\[
j_l^{-1} \circ \chi_p\left(\text{Frob } \pr_1(\pi_j(\bar{x})) \cdot \ldots \cdot \text{Frob } \pr_l(\pi_j(\bar{x}))\right),
\]
where $p$ runs through the primes in $\pi_{S_{\text{var}} - \{j\}}(\bar{x})$ and $l$. Hence we can take $A_S(\mathfrak{A}) := \mathbb{F}_l^m$. Now suppose that $S$ is such that $|S| > 1$ and $i_1(j_1, j_2) \in S$. In this case we define $F_S(\bar{x})$ as
\[
\phi_{S, \bar{x}}\left(\text{Frob } p\right),
\]
where $p$ runs through $x_1, \ldots, x_k$ and the primes in $\pi_{S_{\text{var}} - S}(\bar{x})$, so that $A_S(\mathfrak{A}) := \mathbb{F}_l^{m - |S| + k}$. We remark that both $\chi_p$ and $\phi_{S, \bar{x}}$ implicitly depend on $f$. By Proposition \ref{pAS} the density of $\overline{Y}_{S_{\text{var}}}^\circ$ in $X_{\text{var}}^l$ is at least
\[
\delta' := \left(\frac{\delta}{RM_{\text{box}}^ml^{m + k}}\right)^{(l + 1)^m}.
\]
We will now use some of the techniques and notation from the proof of Proposition \ref{p4.4}. For $g: [l - 1] \rightarrow \overline{Y}_\emptyset^\circ(\mathfrak{A})$ and $x \in \overline{Y}_\emptyset^\circ(\mathfrak{A})$ we defined an element $c(g, x)$ of $\overline{X}_{S_{\text{var}}}$. Also define
\[
Z(\mathfrak{A}, S_{\text{var}}, g) := \{x \in \overline{Y}_\emptyset^\circ(\mathfrak{A}): \text{ writing } \bar{x} := c(g, x), \text{ we have } \bar{x} \in \overline{Y}_{S_{\text{var}}}^\circ(\mathfrak{A})\}.
\]
There is a natural injective map from $\overline{Y}_{S_{\text{var}}}^\circ(\mathfrak{A})$ to 
\[
\coprod_{g: [l - 1] \rightarrow \overline{Y}_\emptyset^\circ(\mathfrak{A})} Z(\mathfrak{A}, S_{\text{var}}, g), 
\]
where $\coprod$ denotes a disjoint union. This map is given by sending $\bar{y}$ to the pair $(g, x)$, where $g: [l - 1] \rightarrow \overline{Y}_\emptyset^\circ(\mathfrak{A})$ and $x \in Z(\mathfrak{A}, S_{\text{var}}, g)$ are uniquely determined by
\[
\pi_i(g(j)) = \pr_j(\pi_i(\bar{y})) \text{  and  } \pi_i(x) = \pr_l(\pi_i(\bar{y}))
\]
for $i \in S_{\text{var}}$ and $j \in [l - 1]$. We conclude that
\[
\left|\overline{Y}_{S_{\text{var}}}^\circ(\mathfrak{A})\right| \leq \sum_{g: [l - 1] \rightarrow \overline{Y}_\emptyset^\circ(\mathfrak{A})} \left|Z(\mathfrak{A}, S_{\text{var}}, g)\right| \leq |X_{\text{var}}|^{l - 1} \cdot \max_{g: [l - 1] \rightarrow \overline{Y}_\emptyset^\circ(\mathfrak{A})} \left|Z(\mathfrak{A}, S_{\text{var}}, g)\right|.
\]
This implies that $Z(\mathfrak{A}, S_{\text{var}}, g)$ has density at least $\delta'$ in $X_{\text{var}}$ for a good choice of $g$. Now the key observation is that there are no subsets $Z_i$ of $X_i(a)$ all of cardinality $M_{\text{box}}$ with the property
\[
\prod_{i \in S_{\text{var}}} Z_i \subseteq Z(\mathfrak{A}, S_{\text{var}}, g),
\]
since then we would be able to extend the sequence $Z_{\text{var}}^1, \ldots, Z_{\text{var}}^t$ to a longer sequence. Hence we can apply the contrapositive of Proposition \ref{pRam}. This yields the inequality
\[
\delta'^{2M_{\text{box}}^{m - 1}} < 2^{m + 2} \left(\min_{i \in S_{\text{var}}} |X_i(a)|\right)^{-1}.
\]
For almost all $f$ we know that $|X_i(a)|$ is of the expected size due to Theorem \ref{t4r}. For these $f$ we get the desired upper bound for $\delta$ and for the remaining $f$ we employ the trivial bound. It follows from Theorem \ref{t4r} that the contribution of those $x \in X(a, f)$ with $\pi_{S_{\text{var}}}(x) \in V_{\text{var}}^{\text{bad}}$ fits in the error term of Proposition \ref{pChar}. Define for $x \in X(a, f)$
\[
\Lambda(x) := \left|\left\{1 \leq s \leq t : x \in \widetilde{Z_{\text{var}}^s}\right\}\right|.
\]
We will compute the first and second moment of $\Lambda(x)$. Let $y \in V_{\text{var}}$ and let $d(M_{\text{box}}, m)$ be the degree of $M^\circ(Z_{\text{var}}^s)$ over $L(y, x_1, \ldots, x_k)$. Then $d(M_{\text{box}}, m)$ does not depend on $s$ or $y$ by Proposition \ref{pDegree}. For $y \in Z_{\text{var}}^s$, Theorem \ref{t4r} implies that the proportion of $f$ violating
\[
\left|\left|X(a, f) \cap \widetilde{Z_{\text{var}}^s} \cap \pi_{S_{\text{var}}}^{-1}(y)\right| - \frac{\left|X(a, f) \cap \pi_{S_{\text{var}}}^{-1}(y)\right|}{d(M_{\text{box}}, m)^{r - k - m}}\right| \leq \frac{A'(l)}{\log N} \frac{\left|X(a, f) \cap \pi_{S_{\text{var}}}^{-1}(y)\right|}{d(M_{\text{box}}, m)^{r - k - m}}
\]
for some $1 \leq s \leq t$ is within the error of the proposition, where $A'(l)$ is a sufficiently large constant. Therefore we have
\begin{align*}
\sum_{x \in X(a, f)} \Lambda(x) &= \sum_{y \in V_{\text{var}}} \sum_{\substack{x \in X(a, f) \\ \pi_{S_{\text{var}}}(x) = y}} \sum_{1 \leq s \leq t} \mathbf{1}_{x \in \widetilde{Z_{\text{var}}^s}} = \sum_{y \in V_{\text{var}}} \sum_{1 \leq s \leq t} \left|X(a, f) \cap \widetilde{Z_{\text{var}}^s} \cap \pi_{S_{\text{var}}}^{-1}(y)\right| \\
&= \sum_{y \in V_{\text{var}}} \sum_{\substack{1 \leq s \leq t \\ y \in Z_{\text{var}}^s}} \frac{\left|X(a, f) \cap \pi_{S_{\text{var}}}^{-1}(y)\right|}{d(M_{\text{box}}, m)^{r - k - m}} + O\left(\frac{1}{\log N} \frac{\left|X(a, f) \cap \pi_{S_{\text{var}}}^{-1}(y)\right|}{d(M_{\text{box}}, m)^{r - k - m}}\right) \\
&= \frac{R\left|X(a, f)\right|}{d(M_{\text{box}}, m)^{r - k - m}} + O\left(\frac{R}{\log N} \frac{\left|X(a, f)\right|}{d(M_{\text{box}}, m)^{r - k - m}}\right).
\end{align*}
An application of Proposition \ref{pDegree} shows that for distinct $s$ and $s'$
\[
\left[M^\circ(Z_{\text{var}}^s)M^\circ\left(Z_{\text{var}}^{s'}\right) : L(y, x_1, \ldots, x_k)\right] = d(M_{\text{box}}, m)^2,
\]
where we use that 
\[
\left|Z_{\text{var}}^s \cap Z_{\text{var}}^{s'}\right| \leq 1.
\]
If $y \in Z_{\text{var}}^s \cap Z_{\text{var}}^{s'}$ for distinct $s$ and $s'$, we obtain from Theorem \ref{t4r}
\[
\left|\left|X(a, f) \cap \widetilde{Z_{\text{var}}^s} \cap \widetilde{Z_{\text{var}}^{s'}} \cap \pi_{S_{\text{var}}}^{-1}(y)\right| - \frac{\left|X(a, f) \cap \pi_{S_{\text{var}}}^{-1}(y)\right|}{d(M_{\text{box}}, m)^{2(r - k - m)}}\right| \leq \frac{A'(l)}{\log N} \frac{\left|X(a, f) \cap \pi_{S_{\text{var}}}^{-1}(y)\right|}{d(M_{\text{box}}, m)^{2(r - k - m)}}.
\]
except for a vanishingly small proportion of $f$ that fit in the error term. Then we can compute the second moment of $\Lambda(x)$ in exactly the same way as we computed the first moment, and this yields
\begin{align*}
\sum_{x \in X(a, f)} \Lambda(x)^2 &= \frac{R\left|X(a, f)\right|}{d(M_{\text{box}}, m)^{r - k - m}} + \frac{(R^2 - R)\left|X(a, f)\right|}{d(M_{\text{box}}, m)^{2(r - k - m)}} + O\left(\frac{R^2}{\log N} \frac{\left|X(a, f)\right|}{d(M_{\text{box}}, m)^{2(r - k - m)}}\right)\\
&= \frac{R^2\left|X(a, f)\right|}{d(M_{\text{box}}, m)^{2(r - k - m)}} + O\left(\frac{R^2}{\log N} \frac{\left|X(a, f)\right|}{d(M_{\text{box}}, m)^{2(r - k - m)}}\right).
\end{align*}
Now use Chebyshev's inequality to finish the proof of the proposition.
\end{proof}

\begin{proof}[Proof of Proposition \ref{pFinal}.] Let $P$ be an element of $\pi_{[r] - S(j_1, j_2)}\left(\widetilde{Z} \cap X(a, f)\right)$ and make a choice of characters for all primes in $P$. Then it suffices to prove
\begin{multline}
\frac{1}{W(X, \{i_2(j_1, j_2)\})} \sum_{f \in \textup{Ch}(\{i_2(j_1, j_2)\})} \left|\sum_{x \in \widetilde{Z} \cap X(a, f, P, m - 1)} F(\textup{Art}(x, f, m)) \right| \leq \\ \nonumber
\frac{A}{W(X, \{i_2(j_1, j_2)\})} \sum_{f \in \textup{Ch}(\{i_2(j_1, j_2)\})} \frac{\left|\widetilde{Z} \cap X(a, f, P)\right|}{\left(\log \log N\right)^{\frac{c}{m(l^2 + l)^m}}}
\end{multline}
for all $P$, where $X(a, f, P, m - 1)$ and $X(a, f, P)$ are the subsets of $X(a, f, m - 1)$ and $X(a, f)$ equal to $P$ on $[r] - S(j_1, j_2)$. Define $Z := \pi_{S_{\text{var}}}\left(\widetilde{Z}\right)$. Then Proposition \ref{pDiff} implies that
\[
\Gal\left(M(Z)/M_\circ(Z)\right) \simeq \mathscr{G}_{S_{\text{var}}}(Z),
\]
where the isomorphism sends $\sigma$ to the map $\bar{x} \mapsto \phi_{S_{\text{var}}, \bar{x}}(\sigma)$. This is well-defined by our assumption that $Z$ is well-governed. Furthermore, the Galois group $\Gal\left(M(Z)/M_\circ(Z)\right)$ is the center of $\Gal\left(M(Z)/\Q\right)$.

We formally apply Proposition \ref{p4.4} with the space $X := Z \times [M_{\text{box}}]$. There is a natural bijection between $\mathscr{G}_S(X)$ and the set of maps $g$ from $[M_{\text{box}}]^l$ to $\Gal\left(M(Z)/M_\circ(Z)\right)$ with the property 
\[
g(i_1, \ldots, i_{l - 1}, k_1) + g(i_1, \ldots, i_{l - 1}, k_2) + \ldots + g(i_1, \ldots, i_{l - 1}, k_l) = g(k_1, \ldots, k_l).
\]
For a prime $p \in X_{i_2(j_1, j_2)}$, let $\mathfrak{p}$ be the prime ideal in $\Z[\zeta_l]$ above $p$ corresponding to the character $\chi_p^{f(p)}$. Given primes $p_1, \ldots, p_{M_{\text{box}}}$ in $X_{i_2(j_1, j_2)}$ one can construct such a function $g$ by defining
\[
g(i_1, \ldots, i_l) = j_l^{-1}(F(j_1, j_2)) \cdot \pi_{i_2(j_1, j_2)}(w_{2, j_2}) \cdot \left(\text{Frob}(\mathfrak{p}_{i_1}) + \ldots + \text{Frob}(\mathfrak{p}_{i_l})\right).
\]
Proposition \ref{p4.4} gives us a specific function $g_0$ such that we have equidistribution for all acceptable $l$-additive systems $\mathfrak{A}$, all choices of $\overline{Y}_\emptyset^\circ$ and all choices of $\widetilde{F}$ with
\[
d\widetilde{F}(\bar{x}) = g_0(\bar{x}).
\]
We are now going to use Theorem \ref{t4r} to partition $X_{i_2(j_1, j_2)}$ into sets of size $M_{\text{box}}$ with the property that they all give the function $g_0$, i.e.
\begin{align}
\label{eg0}
g_0(i_1, \ldots, i_l) = j_l^{-1}(F(j_1, j_2)) \cdot \pi_{i_2(j_1, j_2)}(w_{2, j_2}) \cdot \left(\text{Frob}(\mathfrak{p}_{i_1}) + \ldots + \text{Frob}(\mathfrak{p}_{i_l})\right).
\end{align}
Now define
\[
X_{i_2(j_1, j_2)}(\sigma) = \left\{p \in X_{i_2(j_1, j_2)} : \text{Frob}(\pp) = \sigma\right\},
\]
where $\sigma$ is an element of $\Gal(M(Z)L(P \cup Z)/\Q)$ that restricts to the element of $\Gal(M^\circ(Z)/\Q)$ corresponding to $X_{i_2(j_1, j_2)}(M^\circ(Z))$. Then Theorem \ref{t4r} implies
\begin{align}
\label{eXs}
\left|\left|X_{i_2(j_1, j_2)}(\sigma)\right| - \frac{\left|X_{i_2(j_1, j_2)}(M^\circ(Z))\right|}{l^{(M_{\text{box}} - 1)^m} \cdot l^{2r - 2k - 2m - 2}}\right| \leq \frac{A'(l)}{\log N} \frac{\left|X_{i_2(j_1, j_2)}(M^\circ(Z))\right|}{l^{(M_{\text{box}} - 1)^m} \cdot l^{2r - 2k - 2m - 2}}
\end{align}
for all but very few $f$, where $A'(l)$ is a sufficiently large constant. Now take any $p_1$. Then given $\text{Frob}(\mathfrak{p}_1)$, there is a unique choice of $\text{Frob}(\mathfrak{p}_2), \ldots, \text{Frob}(\mathfrak{p}_{M_{\text{box}}})$ such that equation (\ref{eg0}) is satisfied; in fact $\text{Frob}(\mathfrak{p}_2), \ldots, \text{Frob}(\mathfrak{p}_{M_{\text{box}}})$ is simply a linear function of $\text{Frob}(\mathfrak{p}_1)$ and the fixed function $g_0$. From this observation and equation (\ref{eXs}), we conclude that $X_{i_2(j_1, j_2)}(M^\circ(Z)L(P \cup Z))$ can be partitioned in sets $A$ of size $M_{\text{box}}$ such that equation (\ref{eg0}) is valid except for a small set that fits in the error. 

If $A$ is such a set of size $M_{\text{box}}$, we have an explicit bijection between $A$ and $[M_{\text{box}}]$ coming from our choice of $p_1, \ldots, p_{M_{\text{box}}}$. We first restrict the $l$-additive system $\mathfrak{A}$ constructed in Lemma \ref{lAddCon} to $Z \times A$ and then use this bijection to get an $l$-additive system on $Z \times [M_{\text{box}}]$. This gives the desired equidistribution for $F$ on $Z \times A$, and hence also $\widetilde{Z}$.
\end{proof}

\section{\texorpdfstring{Equidistribution in $(\mathcal{G}_{\mathbb{Z}_l[\zeta_l]},\mu_{\text{C.L.}}^{1})$}{Theorem 1.2 implies Theorem 1.1}}
In Proposition \ref{units are locally free} we have shown that $O_{K_\chi}^\ast \otimes_{\mathbb{Z}} \mathbb{Z}_l$ is a free module over $\mathbb{Z}_l[\zeta_l]$ of rank $1$ for $\chi \in\Gamma_{\mu_l}(\mathbb{Q})$, where the action of $\mathbb{Z}_l[\zeta_l]$ is as usual induced by the Galois action through $\chi$; observe that after tensoring with $\mathbb{Z}_l$ the norm operator acts trivially. Thus, by analogy with real quadratic fields, it is natural to expect that, as $\chi$ varies in $\Gamma_{\mu_l}(\mathbb{Q})$, the $\mathbb{Z}_l[\zeta_l]$-module $(1 - \zeta_l)\text{Cl}(K_\chi)[(1 - \zeta_l)^\infty]$ should equidistribute in the Cohen--Lenstra probability space $(\mathcal{G}_{\mathbb{Z}_l[\zeta_l]}, \mu_{\text{C.L.}}^1)$ which is defined as follows. 

The set $\mathcal{G}_{\mathbb{Z}_l[\zeta_l]}$ consists of the set of isomorphism classes of $\mathbb{Z}_l[\zeta_l]$-modules with finite cardinality. To each $A \in \mathcal{G}_{\mathbb{Z}_l[\zeta_l]}$ we give weight
$$
\mu_{\text{C.L.}}^1(A) := \frac{\eta_{\infty}^{1}(l)}{\left|A\right| \cdot \left|\text{Aut}_{\mathbb{Z}_l[\zeta_l]}(A)\right|},
$$
where $\eta_{\infty}^{1}(l) := \prod_{i = 2}^\infty (1 - \frac{1}{l^{i}})$. We will prove later in this section that this formula defines a probability measure. The goal of this section is to show that this statistical model is equivalent to the statistical model for the sequence of ranks established in Theorem \ref{tCyclic}. In particular with the material of this section one sees that Theorem \ref{tCyclic} implies Theorem \ref{tMain}. 

Let $\mathcal{D}$ be the set of non-increasing functions $f:\mathbb{Z}_{\geq 1}  \to \mathbb{Z}_{\geq 0}$ that are eventually $0$. Recall that the map
$$
\text{rk}:\mathcal{G}_{\mathbb{Z}_l[\zeta_l]} \to \mathcal{D},
$$
defined by the formula $A \mapsto (i \mapsto \text{rk}_{(1 - \zeta_l)^i} A)$, is a bijection of sets. The next proposition gives the pushforward of $\mu_{\text{C.L.}}^{1}$ under the bijection $\text{rk}$. For $0 \leq j \leq n$, recall that $P(j | n)$ is the probability that a $n \times (n + 1)$ matrix with entries in $\mathbb{F}_l$ has rank $n - j$. Moreover, we set
$$
\eta_n(l) := \prod_{i = 1}^{n} \left(1 - \frac{1}{l^i}\right).
$$
Then we have the following proposition.
 
\begin{prop} 
\label{We show Cohen--Lenstra}
Let $j$ be a positive integer. Let $i_1 \geq \ldots \geq i_j \geq 0$ be a sequence of integers. Then
\begin{multline*}
\mu_{\emph{C.L.}}^1 \left(\emph{rk}^{-1}\left(\{f \in \mathcal{D} : f(1) = i_1, \ldots, f(j) = i_j\}\right)\right) = \\ 
\frac{\eta_\infty(l)}{l^{i_1(i_1 + 1)} \eta_{i_1}(l) \eta_{i_1 + 1}(l)} \cdot \prod_{1 \leq k < j} P(i_{k + 1} | i_k).
\end{multline*}
\end{prop}

The rest of this section is devoted to the proof of Proposition \ref{We show Cohen--Lenstra}. First recall that for each $A \in \mathcal{G}_{\mathbb{Z}_l[\zeta_l]}$, the measure $\mu_{\text{C.L.}}^{1}(A)$ can be obtained as the limit
$$
\lim_{N \to \infty} \mu_{\text{Haar}}\left(\left\{(v_1, \ldots, v_{N + 1}) \in \left(\mathbb{Z}_l[\zeta_l]^N\right)^{N + 1} : \frac{\mathbb{Z}_l[\zeta_l]^N}{\langle v_1, \ldots, v_{N + 1}\rangle} \simeq A\right\}\right) = \mu_{\text{C.L.}}^{1}(A).
$$
Indeed, for each positive integer $N$, denote by $\mathcal{L}_{A, N}$ the set of $\mathbb{Z}_l[\zeta_l]$-submodules $L$ of $\mathbb{Z}_l[\zeta_l]^N$ satisfying 
$$
\frac{\mathbb{Z}_l[\zeta_l]^N}{L} \simeq A.
$$
We have that
\begin{multline*}
\mu_{\text{Haar}}\left(\left\{(v_1, \ldots, v_{N + 1}) \in \left(\mathbb{Z}_l[\zeta_l]^N\right)^{N + 1} : \frac{\mathbb{Z}_l[\zeta_l]^N}{\langle v_1, \ldots, v_{N + 1} \rangle} \simeq A\right\}\right) = \\
\sum_{L \in \mathcal{L}_{A, N}} \mu_{\text{Haar}}\left(\left\{(v_1, \ldots, v_{N + 1}) \in \left(\mathbb{Z}_l[\zeta_l]^N\right)^{N + 1} : \langle v_1, \ldots, v_{N + 1} \rangle = L\right\}\right).
\end{multline*}
We further have
\begin{multline*}
\mu_{\text{Haar}}\left(\left\{(v_1, \ldots, v_{N + 1}) \in \left(\mathbb{Z}_l[\zeta_l]^N\right)^{N + 1} : \langle v_1, \ldots, v_{N + 1} \rangle = L\right\}\right) = \\
\prod_{i = 2}^{N + 1} \left(1 - \frac{1}{l^i}\right) \cdot \frac{1}{\left|A\right|^{N + 1}} \cdot \left|\mathcal{L}_{A, N}\right|
\end{multline*}
and the following simple formula for $\left|\mathcal{L}_{A, N}\right|$ 
$$
\left|\mathcal{L}_{A, N}\right| = \frac{\left|\text{Epi}_{\mathbb{Z}_l[\zeta_l]}\left(\mathbb{Z}_l[\zeta_l]^{N}, A\right)\right|}{\left|\text{Aut}_{\mathbb{Z}_l[\zeta_l]}(A)\right|}.
$$
But observe that 
$$
\frac{\left|\text{Epi}_{\mathbb{Z}_l[\zeta_l]}(\mathbb{Z}_l[\zeta_l]^N, A)\right|}{\left|A\right|^N} \to 1
$$ 
as $N$ goes to infinity. This gives 
$$
\lim_{N \to \infty} \mu_{\text{Haar}}\left(\left\{(v_1, \ldots, v_{N + 1}) \in \left(\mathbb{Z}_l[\zeta_l]^N\right)^{N + 1} : \frac{\mathbb{Z}_l[\zeta_l]^N}{\langle v_1, \ldots, v_{N + 1} \rangle} \simeq A\right\}\right) = \mu_{\text{C.L.}}^1(A).
$$ 
It is not difficult to show the slightly refined conclusion that the convergence also holds if we take a subset of $\mathcal{G}_{\mathbb{Z}_l[\zeta_l]}$, which thus shows that $\mu_{\text{C.L.}}^{1}$ is a probability measure. Using this we can show Proposition \ref{We show Cohen--Lenstra} by first computing the pushforward of $\mu_{\text{C.L.}}^{1}$ by $\text{rk}$ at stage $N$, i.e. the $N$-th approximation of the pushforward. Sending $N$ to infinity will yield the desired conclusion. 

In the notation of Proposition \ref{We show Cohen--Lenstra}, let us begin with $j = 1$ and fix an integer $i_1 \geq 0$. Observe that the probability that $f(1) = i_1$ at stage $N$ is given by the probability that the reduction of $v_1, \ldots, v_{N + 1}$ modulo $(1 - \zeta_l)$ generates a subspace of dimension $N - i_1$. Splitting the probability by the contribution coming from each subspace of dimension $N - i_1$ one gets
$$
\left|\text{subspaces of dimension $N - i_1$ in $\mathbb{F}_l^N$}\right| \cdot \frac{\mathbb{P}\left((w_1, \ldots, w_{N + 1}) \in \mathbb{F}_l^{N - i_1} \text{ generate}\right)}{l^{i_1(N + 1)}}.
$$
This we can rewrite as
$$
\frac{\mathbb{P}\left((w_1, \ldots, w_{N + 1}) \in \mathbb{F}_l^{N-i_1} \text{ generate}\right)}{\text{Aut}_{\mathbb{F}_l}(\mathbb{F}_l^{i_1})} \cdot \frac{1}{l^{i_1}} \cdot \frac{\left|\text{Epi}\left(\mathbb{F}_l^N, \mathbb{F}_l^{i_1}\right)\right|}{l^{i_1N}}.
$$
Again the factor $\frac{\left|\text{Epi}\left(\mathbb{F}_l^N, \mathbb{F}_l^{i_1}\right)\right|}{l^{i_1N}}$ approaches $1$ as $N$ goes to infinity. Moreover
$$ 
\mathbb{P}\left((w_1, \ldots, w_{N + 1}) \in \mathbb{F}_l^{N - i_1} \text{ generate}\right) = \prod_{i = i_1 + 2}^{N + 1} \left(1 - \frac{1}{l^i}\right).
$$ 
Plugging in and sending $N$ to infinity yields the case $j = 1$ for Proposition \ref{We show Cohen--Lenstra}. We now continue by induction to compute the $N$-th approximation for any $N > i_1$. Observe that whether $(v_1, \ldots, v_{N + 1})$ is giving an $A$ with $f(1) = i_1, \ldots, f(j) = i_j$ can be decided completely by the image of $(v_1, \ldots ,v_{N + 1})$ modulo $(1 - \zeta_l)^j$. Hence we proceed to show that if we \emph{fix} the image of $(v_1, \ldots, v_{N + 1})$ modulo $(1 - \zeta_l)^j$, then the $N$-th probability that $f(j + 1) = i_{j + 1}$, conditional on the image modulo $(1 - \zeta_l)^j$ being fixed, is always $P(i_{j + 1} | i_j)$. From this the desired conclusion follows immediately.

Since the image modulo $(1 - \zeta_l)^j$ has been fixed, we fix a subset $\mathcal{B}$ of $[N + 1]$ such that $\{v_i\}_{i \in \mathcal{B}}$ forms a minimal set of generators for the image modulo $(1 - \zeta_l)^j$. By construction the set $\mathcal{B}$ has size $N - i_j$. By multiplying each element of $\mathcal{B}$ with suitable powers of $1 - \zeta_l$ we obtain a subset of 
$$
V := \frac{(1 - \zeta_l)^j \mathbb{Z}_l[\zeta_l]^N}{(1 - \zeta_l)^{j + 1} \mathbb{Z}_l[\zeta_l]^N}
$$
generating a space of dimension $N - i_j$, which we call $V'$. Note that $V'$ is fixed as $(v_1, \ldots, v_{N + 1})$ varies among vectors with fixed image modulo $(1 - \zeta_l)^j$. Since $\mathcal{B}$ is a minimal set of generators modulo $(1 - \zeta_l)^j$, we see that there is a natural map $F$ that sends the $i_j + 1$ vectors $(v_i)_{i \not \in \mathcal{B}}$ in the $i_j$-dimensional $\mathbb{F}_l$-vector space
$$
\frac{\frac{(1 - \zeta_l)^j \mathbb{Z}_l[\zeta_l]^N}{(1 - \zeta_l)^{j + 1}\mathbb{Z}_l[\zeta_l]^N}}{V'}.
$$
It is seen at once that if the image of $F$ spans a space of dimension $k$, then the resulting $A$ will satisfy $f(j + 1) = i_j - k$. Moreover, it is easy to see that each vector is obtained equally often through $F$. Thus we obtain the desired conclusion. 

\appendix
\section{Cyclic algebras} 
\label{general facts about cyclic algebras}
In this small appendix we collect several basic facts that are used in this paper coming from local and global class field theory along with some more general facts about cyclic algebras over general fields. 

Let $K$ be any field. Denote by $K^{\text{sep}}$ a fixed separable closure of $K$ and by $G_K$ the group of $K$-algebra automorphisms of $K^{\text{sep}}$. Let $\chi : G_K \to \mathbb{C}^\ast$ be a continuous character and define $K(\chi) := (K^{\text{sep}})^{\text{ker}(\chi)}$. Let $n$ be the degree of $K(\chi)$ over $K$ and let $\theta$ be in $K^\ast$. Following the notation from \cite[ch.\ 9]{Weil} we denote by $\{\chi, \theta\}$ the twisted polynomial ring $K(\chi) \langle \beta \rangle$ with the relations
$$
\beta^n = \theta \quad \text{  and  } \beta \lambda \beta^{-1} = \chi^{-1}\left(\exp\left(\frac{2 \pi \text{i}}{\text{ord}(\chi)}\right)\right)(\lambda),
$$
which is a cyclic algebra. Denote by $\Phi_\chi$ the unique map from $G_K$ to $\mathbb{R}$ such that $\text{Im}(\Phi_\chi) \subseteq [0,1)$ and
$$\
\exp\left(2 \pi \text{i} \cdot \Phi_\chi\right) = \chi.
$$
The map $\Phi_{\chi}$ is a locally constant map, whose values are in the set 
$$
\left\{0, \frac{1}{\text{ord}(\chi)}, \ldots, \frac{\text{ord}(\chi) - 1}{\text{ord}(\chi)}\right\}.
$$
Denote by $\widetilde{\Phi_{\chi}}$ the map $\text{ord}(\chi) \cdot \Phi_{\chi}$. The map $\widetilde{\Phi_\chi}$ is a locally constant map with values in $\{0, \ldots , \text{ord}(\chi)-1 \}$. Observe that since $\chi$ is a character, for each $\sigma, \tau \in G_K$ the element
$$
\Phi_\chi(\sigma) + \Phi_\chi(\tau) - \Phi_\chi(\sigma \tau)
$$
is an integer. This allows us to define a $2$-cocycle $h_{\{\chi, \tau\}}$ of $G_K$ with values in $K^\ast$ by the formula
$$
(\sigma, \tau) \mapsto \theta^{\Phi_\chi(\sigma) + \Phi_\chi(\tau) - \Phi_\chi(\sigma \tau)}.
$$
If we could separate the three values on the exponent of $h_{\{\chi, \tau \}}$ we would obtain trivially a coboundary, for this reason we already know that the above formula defines a $2$-cocycle. But the three terms in general can not be separated, since it is only the total sum that is an integer. This observation will be useful in Proposition \ref{cup products and cyclic algebras}. The reason why we introduced this particular $2$-cocycle is that the class of $h_{\{\chi, \theta\}}$ in $\text{Br}_K$ is precisely the class of $\{\chi,\theta\}$. This fact is established in \cite{Weil}. Recall the following fundamental fact from \cite[p.\ 223]{Weil}.

\begin{prop} 
\label{inv and Artin symbols}
Let $K$ be a local field. There is a unique isomorphism
$$
\eta_K : \emph{Br}_K \to \mu_{\infty}(\mathbb{C})
$$
such that for any continuous unramified character $\chi : G_K \to \mathbb{C}^\ast$ and any uniformizer $\pi$ of $K$ we have 
$$
\eta_K\left(\left\{\chi, \pi\right\}\right) = \chi\left(\emph{Frob}_K \ \pi\right).
$$
\end{prop}

The map $\eta_K$ actually equals $\exp(2 \pi \text{i} \cdot\text{inv}_K)$ (for a definition of $\text{inv}_K$ see \cite{Serre}). We shall use $\eta_K$ instead of $\text{inv}$ since our main reference is \cite{Weil}. It is defined also for $K = \mathbb{R}$ or $K = \mathbb{C}$ being trivial in the latter case and being the unique isomorphism between $\text{Br}_{\mathbb{R}}$ and $\langle -1 \rangle$ in the former. Recall the following reformulation of Hilbert's reciprocity law, whose proof can be found in \cite[p.\ 255]{Weil}. 

\begin{prop} 
\label{Hilbert reciprocity}
\emph{(Hilbert reciprocity law)}. Let $K$ be a number field and let $\alpha \in \emph{Br}_K$. Then $\eta_{K_v}(\alpha)$ is trivial for all but finitely many values of $v \in \Omega_{K_v}$. It is trivial at all places if and only if $\alpha$ itself is trivial. Moreover, we have that
$$ 
\prod_{v \in \Omega_K} \eta_{K_v}(\alpha) = 1.
$$
\end{prop}

Let $l$ be an odd prime. We now turn to recall a relation between cyclic algebras and cup products in case there are $l$-th roots of unity. We shall confine ourselves to classes killed by $l$, since this is the relevant case in our application. For more general results the reader can consult \cite{Szamuely}. For a field $K$ provided with a distinguished generator $\zeta_l$ for $\mu_l(K)$ we shall use precisely the same symbolic formulas introduced in Section \ref{conventions}. Moreover, for such a $K$ and for an element $\theta \in K^\ast$ we denote by
$$ 
\chi_{\theta} : G_K \to \mathbb{F}_l,
$$
the unique continuous character such that for each $\beta \in K^{\sep}$ with $\beta^l = \theta$ we have
$$
\sigma(\beta) = \left(j_l \circ \chi_\theta(\sigma)\right) \beta.
$$ 
In what follows, when we consider the cup product, the trivial Galois modules $\mathbb{F}_l \otimes \mathbb{F}_l$ and $\mathbb{F}_l$ are identified with the isomorphism $a \otimes b \mapsto a \cdot b$. In particular the cup product of two characters $\chi_1, \chi_2$ in $\mathbb{F}_l$ is literally just the product map $(\sigma, \tau) \mapsto \chi_1(\sigma) \chi_2(\tau)$.

\begin{prop} 
\label{cup products and cyclic algebras}
Suppose $K$ is equipped with an element $\zeta_l$ of multiplicative order equal to $l$. Let $\chi$ be a continuous homomorphism from $G_K$ to $\mathbb{F}_l$. Let $\theta$ be in $K^\ast$. We have the following equality in $\emph{Br}_K$
$$
h_{\{j_l \circ \chi, \theta\}} = j_l \circ (\chi \cup \chi_\theta).
$$
\end{prop}

\begin{proof}
We divide the 2-cocycle $h_{\{j_l \circ \chi,\theta\}}$ by the $1$-coboundary
$$
(\sigma, \tau) \mapsto \frac{\sigma\left(\beta^{\widetilde{\Phi_\chi}(\tau)}\right)}{\beta^{\widetilde{\Phi_\chi}(\sigma \tau) - \widetilde{\Phi_\chi}(\sigma)}},
$$
where $\beta$ is any element of $K^{\text{sep}}$ with $\beta^l = \theta$ and $\widetilde{\Phi_\chi}$ is shorthand for $\widetilde{\Phi_{j_l \circ \chi}}$. Now using the formula 
$$
\sigma\left(\beta^{\widetilde{\Phi_\chi}(\tau)}\right) = \zeta_l^{\chi_{\theta}(\sigma) \chi(\tau)}\beta^{\widetilde{\Phi_\chi}(\tau)}, 
$$
we obtain that the cocycle we are writing is $j_l \circ (-\chi_{\theta} \cup \chi)$. Recalling that, in cohomology, the cup is antisymmetric we conclude immediately. 
\end{proof}

We in particular deduce the following corollary, which holds in greater generality, see \cite{Weil}, than for cyclic degree $l$ characters and without any restriction on the characteristic. However this generality is the one we need and on the other hand we propose an unusual argument based only on the material of Section \ref{central extensions}, that is anyway of fundamental use in this paper. 

\begin{corollary} 
\label{if and only if it is a norm}
Suppose $\emph{char}(K) \neq l$. Let $\chi$ be a continuous character from $G_K$ to $\mathbb{F}_l$. Let $\theta$ be in $K^\ast$. Then $\{\chi, \theta\}$ is trivial in $\emph{Br}_K$ if and only if $\theta$ is a norm from $K(\chi)$.
\end{corollary}

\begin{proof}
We show how to reduce to the case that $\mu_l(K) \neq \{1\}$. Once that is done, we reach the desired conclusion by an application of Proposition \ref{Heisenberg group and norms}. Since $\text{char}(K) \neq l$ we can adjoin in any case to $K$ an element $\zeta_l$ from $K^{\text{sep}}$ having multiplicative order equal to $l$. Thanks to the co-restriction map, we see that $\{\chi, \theta\}$ is trivial in $\text{Br}_{K(\zeta_l)}$ if and only if $\{\chi, \theta\}$ is trivial in $\text{Br}_{K}$. Here we use that $[K(\zeta_l) : K]$ divides $l - 1$ and hence is coprime to $l$.

It remains to prove that $\theta$ is a norm from $K(\zeta_l)(\chi)$ if and only if $\theta$ is a norm from $K(\chi)$. Suppose that $\theta$ is a norm from $K(\zeta_l)(\chi)$, say
$$
\theta = N_{K(\zeta_l)(\chi)/K(\zeta_l)}(\gamma)
$$ 
for some $\gamma \in K(\zeta_l)(\chi)$. Then we see that 
$$
\theta^{[K(\zeta_l):K]}=N_{K(\zeta_l)(\chi)/K}(\gamma)
$$
and hence 
$$
\theta^{[K(\zeta_l) : K]} = N_{K(\chi)/K}(N_{K(\zeta_l)(\chi)/K(\chi)}(\gamma)).
$$
Using once more that $[K(\zeta_l) : K] \mid l - 1$, we conclude that $\theta$ is a norm from $K(\chi)$. The other direction is more general. Now suppose $\theta = N_{K(\chi)/K}(\gamma)$ for some $\gamma \in K(\chi)$. Observe that $\text{Gal}(K(\zeta_l)(\chi)/K(\zeta_l))$ maps injectively into a normal subgroup of $\text{Gal}(K(\chi)/K)$. Fix a set of representatives $\mathcal{S}$ for the quotient of this normal subgroup. Write $\gamma':=\prod_{g \in \mathcal{S}} g(\gamma)$, then we have 
$$
N_{K(\zeta_l)(\chi)/K(\zeta_l)}(\gamma') = \prod_{h \in \text{Gal}(K(\zeta_l)(\chi)/K(\zeta_l))} h(\gamma') = \prod_{g \in \text{Gal}(K(\chi)/K)} g(\gamma) = \theta.
$$ 
This shows the other direction.
\end{proof}

We end this section by recalling how the field of definition of the character $\chi_q \in \Gamma_{\mu_l}(\mathbb{Q})$, introduced in Section \ref{conventions}, looks locally at $q$. We begin by recalling the following basic fact.

\begin{prop} 
\label{Tame lemma}
Let $K$ be a local field and $d$ a positive integer coprime with the size of the residue field of $K$. Let $f(x)$ be a degree $d$ Eisenstein polynomial over $K$. Then
$$
K[x]/f(x) \simeq_{K \emph{-alg}} K\left[\sqrt[d]{-f(0)}\right].
$$
\end{prop}

\begin{proof}
Without loss of generality we may assume that $f$ is monic. In $K[x]/f(x)$ we have 
$$
x^d = -f(0) + \sum_{i = 1}^{d - 1} a_i x^i,
$$
where the $-a_i$ are the various coefficients of $f(x)$. Dividing out by $-f(0)$ we obtain
$$
\frac{x^d}{-f(0)} = 1 + \sum_{i = 1}^{d - 1} \frac{a_i}{-f(0)} x^i.
$$
Since $f(x)$ is Eisenstein, we have that $\frac{a_i}{-f(0)}$ is still integral. Hence $\frac{a_i}{-f(0)}x^i$ is in the maximal ideal of $O_{K[x]/f(x)}$ for each $i$ between $1$ and $d - 1$. This implies
$$
1 + \sum_{i = 1}^{d - 1}\frac{a_i}{-f(0)} x^i \in U_1(K[x]/f(x)).
$$ 
The topological group $U_1(K[x]/f(x))$ is a $\mathbb{Z}_l$-module, where $l$ is the residue characteristic of $K$. In particular it is $d$-divisible, since $d$ is coprime to $l$. Therefore we conclude that 
$$
1 + \sum_{i=1}^{d - 1} \frac{a_i}{-f(0)} x^i
$$
is a $d$-th power and hence also $-f(0)$ is a $d$-th power, since it is the ratio of two $d$-th powers. Finally, the polynomial $T^d + f(0)$ is again Eisenstein; so if we pick $\beta \in K[x]/f(x)$ with $\beta^d = -f(0)$ we see that $K(\beta) = K[x]/f(x)$ and 
$$
K(\beta) \simeq_{K \text{-alg}} K\left[\sqrt[d]{-f(0)}\right],
$$
which is the desired isomorphism.
\end{proof}

Therefore we conclude the following fact, which also follows from class field theory as we shall see in the second proof.

\begin{corollary}
\label{chiq locally at q}
Let $q$ be a prime that is $1$ modulo $l$ and let $\chi_q$ be as given in Section \ref{conventions}. Then we have
$$
(K_{\chi_q})_{\emph{Up}_{K_{\chi_q}}(q)} \simeq_{\mathbb{Q}_q\emph{-alg}} \mathbb{Q}_q(\sqrt[l]{q}).
$$
\end{corollary}

\begin{proof}[First proof]
If we denote by $\Phi_q(T)$ the $q$-th cyclotomic polynomial, we observe that $\Phi_q(T+1)$ is Eisenstein of degree $q - 1$. Moreover, if evaluated in $0$, $\Phi_q(T + 1)$ is equal to $q$. Therefore we conclude by Proposition \ref{Tame lemma} that $\mathbb{Q}_q(\zeta_q)$ completed at $(1 - \zeta_q)$ is the extension $\mathbb{Q}_q(\sqrt[q-1]{-q})$. In particular its unique degree $l$ subextension is given by $\mathbb{Q}_q(\sqrt[l]{-q})$. Since $l$ is odd, the minus sign is irrelevant and the conclusion follows. 
\end{proof}

\begin{proof}[Second proof]
Still by looking at the polynomial $\Phi_q(T)$, we see that $-q$ is a norm locally from $\mathbb{Q}_q(\zeta_q)$. Hence $-q$ is also a norm from the degree $l$ subextension, which is of odd degree, so $q$ is a norm from the degree $l$ subextension. On the other hand, since $q$ is $1$ modulo $l$, we have that the extension $\mathbb{Q}_q(\sqrt[l]{q})$ is cyclic of degree $l$. Moreover, taking the norm of $-\sqrt[l]{q}$ we see that $q$ is a norm also in this extension. It is not difficult to conclude from this that the two fields in the isomorphism have the same norm group. The conclusion follows from local class field theory. 
\end{proof}

\end{document}